\documentclass[11pt,a4paper]{article}
\usepackage{mathrsfs}
\usepackage{pict2e}
\usepackage{booktabs}
\usepackage{color}
\usepackage{epsfig}
\usepackage{epstopdf}
\usepackage{longtable}
\usepackage[T1]{fontenc}
\usepackage{caption}
\usepackage{amsmath}
\usepackage{cases}
\usepackage{geometry}
\usepackage{subfigure}
\usepackage{amsbsy,latexsym,amsfonts, epsfig, color, authblk, amssymb, graphics, bm}
\usepackage{epsf,slidesec,epic,eepic}
\usepackage{fancybox}
\usepackage{fancyhdr}
\usepackage{setspace}
\usepackage{algorithm}
\usepackage{algorithmic}
\usepackage{nccmath}
\usepackage{diagbox}
\usepackage{breakurl}
\usepackage[colorlinks, citecolor=blue]{hyperref}

\newtheorem{theorem}{Theorem}[section]
\newtheorem{lemma}{Lemma}[section]
\newtheorem{definition}{Definition}[section]

\newtheorem{proposition}{Proposition}[section]
\newtheorem{corollary}{Corollary}[section]

\newtheorem{remark}{Remark}[section]

\newenvironment{proof}{{\noindent \bf Proof:}}{\hfill$\Box$\medskip}

\definecolor{lred}{rgb}{1,0.8,0.8}
\definecolor{lblue}{rgb}{0.8,0.8,1}
\definecolor{dred}{rgb}{0.6,0,0}
\definecolor{dblue}{rgb}{0,0,0.5}
\definecolor{dgreen}{rgb}{0,0.5,0.5}

 \title{Computing one-bit compressive sensing via zero-norm regularized DC loss model and its surrogate}

 \author{Kai Chen\footnote{School of Mathematics, South China University of Technology, Guangzhou, China}\ \ \
 Ling Liang\footnote{School of Mathematics, South China University of Technology, Guangzhou, China}\ \ {\rm and}\ \
 Shaohua Pan\footnote{Corresponding author (shhpan@scut.edu.cn), School of Mathematics, South China University of Technology, Guangzhou, China}
 }

 \date{}

\begin{document}

 \maketitle

 \begin{abstract}
  One-bit compressed sensing is very popular in signal processing and communications
  due to its low storage costs and low hardware complexity, but it is a challenging task
  to recover the signal by using the one-bit information. In this paper, we propose
  a zero-norm regularized smooth difference of convexity (DC) loss model and
  derive a family of equivalent nonconvex surrogates covering the MCP and SCAD
  surrogates as special cases. Compared to the existing models, the new model and
  its SCAD surrogate have better robustness. To compute their $\tau$-stationary points,
  we develop a proximal gradient algorithm with extrapolation and establish the convergence
  of the whole iterate sequence. Also, the convergence is proved to have a linear rate
  under a mild condition by studying the KL property of exponent $0$ of the models.
  Numerical comparisons with several state-of-art methods show that
  in terms of the quality of solution, the proposed model and its SCAD surrogate
  are remarkably superior to the $\ell_p$-norm regularized models,
  and are comparable even superior to those sparsity constrained models
  with the true sparsity and the sign flip ratio as inputs.
 \end{abstract}


%----------------------------------------------------------------------------------------------
 \section{Introduction}\label{sec1.0}

  Compressive sensing (CS) has gained significant progress in theory and algorithms over
  the past few decades since the seminal works \cite{Candes05,Donoho06}. It aims
  to recover a sparse signal $x^{\rm true}\in\mathbb{R}^n$ from a small number of
  linear measurements. One-bit compressive sensing, as a variant of the CS,
  was proposed in \cite{Boufounos08} and had attracted considerable interests in
  the past few years (see, e.g., \cite{Dai16,Jacques13,Plan13a,Yan12,Zhang14,HuangS18}).
  Unlike the conventional CS which relies on real-valued measurements, one-bit CS aims to
  reconstruct the sparse signal $x^{\rm true}$ from the sign of measurement. Such a new setup is
  appealing because (i) the hardware implementation of one-bit quantizer is low-cost and efficient;
  (ii) one-bit measurement is robust to nonlinear distortions \cite{Boufounos10};
  and (iii) in certain situations, for example, when the signal-to-noise ratio
  is low, one-bit CS performs even better than the conventional one \cite{Laska12}.
  For the applications of one-bit CS, we refer to the recent survey paper \cite{Li18}.
%----------------------------------------------------------------------------------------------
 \subsection{Review on the related works}\label{sec1.1}

  In the noiseless setup, the one-bit CS acquires the measurements via the linear
  model $b={\rm sgn}(\Phi x^{\rm true})$, where $\Phi\in\mathbb{R}^{m\times n}$
  is the measurement matrix and the function ${\rm sgn}(\cdot)$ is applied to
  $\Phi x^{\rm true}$ in a component-wise way. Here, for any $t\in\mathbb{R}$,
  ${\rm sgn}(t)=1$ if $t>0$ and $-1$ otherwise, which has a little difference
  from the common ${\rm sign}(\cdot)$. By following the theory
  of conventional CS, the ideal optimization model for one-bit CS is as follows:
  \begin{equation}\label{znorm-min}
   \min_{x\in\mathbb{R}^n}\Big\{\|x\|_0\ \ {\rm s.t.}\ \ b={\rm sgn}(\Phi x),\,\|x\|=1\Big\},
  \end{equation}
  where $\|x\|_0$ denotes the zero-norm (i.e., the number of nonzero entries) of $x\in\mathbb{R}^n$,
  and $\|x\|$ means the Euclidean norm of $x$. The unit sphere constraint is introduced
  into \eqref{znorm-min} to address the issue that the scale information of a signal is lost during
  the one-bit quantization. Due to the combinatorial properties of the functions
  ${\rm sgn}(\cdot)$ and $\|\cdot\|_0$, the problem \eqref{znorm-min} is NP-hard.
  Some earlier works (see, e.g., \cite{Boufounos08,Plan13a,Wang15}) mainly focus on
  its convex relaxation model, obtained by replacing the zero-norm by the $\ell_1$-norm
  and relaxing the consistency constraint $b={\rm sgn}(\Phi x)$ into the linear constraint
  $b\circ(\Phi x)\ge 0$, where the notation ``$\circ$'' means the Hadamard operation of vectors.

  In practice the measurement is often contaminated by noise before the quantization and
  some signs will be flipped after quantization due to quantization distoration, i.e.,
  \begin{equation}\label{observation}
    b=\zeta\circ{\rm sgn}(\Phi x^{\rm true}+\varepsilon)
  \end{equation}
  where $\zeta\in\{-1,1\}^m$ is a random binary vector and $\varepsilon\in\mathbb{R}^m$ denotes
  the noise vector. Let $L\!:\mathbb{R}^m\to\mathbb{R}_{+}$ be a loss function to ensure
  data fidelity as well as to tolerate the existence of sign flips.
  Then, it is natural to consider the zero-norm regularized loss model:
  \begin{equation}\label{znorm-reg}
   \min_{x\in\mathbb{R}^n}\Big\{L(Ax)+\lambda\|x\|_0\ \ {\rm s.t.}\ \ \|x\|=1\Big\}\ \ {\rm with}\
   A:={\rm Diag}(b)\Phi,
  \end{equation}
  and achieve a desirable estimation for the true signal $x^{\rm true}$ by tuning
  the parameter $\lambda>0$. Consider that the projection mapping onto the intersection
  of the sparsity constraint set and the unit sphere has a closed norm. Some researchers
  prefer the following model or a similar variant to achieve a desirable estimation
  for $x^{\rm true}$ (see, e.g., \cite{Boufounos09,Yan12,Dai16,Zhou21}):
  \begin{equation}\label{znorm-constr}
   \min_{x\in\mathbb{R}^n}\Big\{L(Ax)\ \ {\rm s.t.}\ \ \|x\|_0\le s,\,\|x\|=1\Big\},
  \end{equation}
 where the positive integer $s$ is an estimation for the sparsity of $x^{\rm true}$.
 For this model, if there is a big difference between the estimation $s$ from the true sparsity $s^*$,
 the mean-squared-error (MSE) of the associated solutions will become worse.
 Take the model in \cite{Zhou21} for example. If the difference between the estimation $s$
 from the true sparsity $s^*$ is $2$, the MSE of the associated solutions will have a difference
 at least $20\%$ (see Figure \ref{fig_K}). Moreover, now it is unclear how to achieve
 such a tight estimation for $s^*$. We find that the numerical experiments for the zero-norm constrained 
 model all use the true sparsity as an input (see \cite{Yan12,Zhou21}). 
 In this work, we are interested in the regularization models.
  \begin{figure}[!h]
   \centering
  \includegraphics[width=8cm,height=6cm]{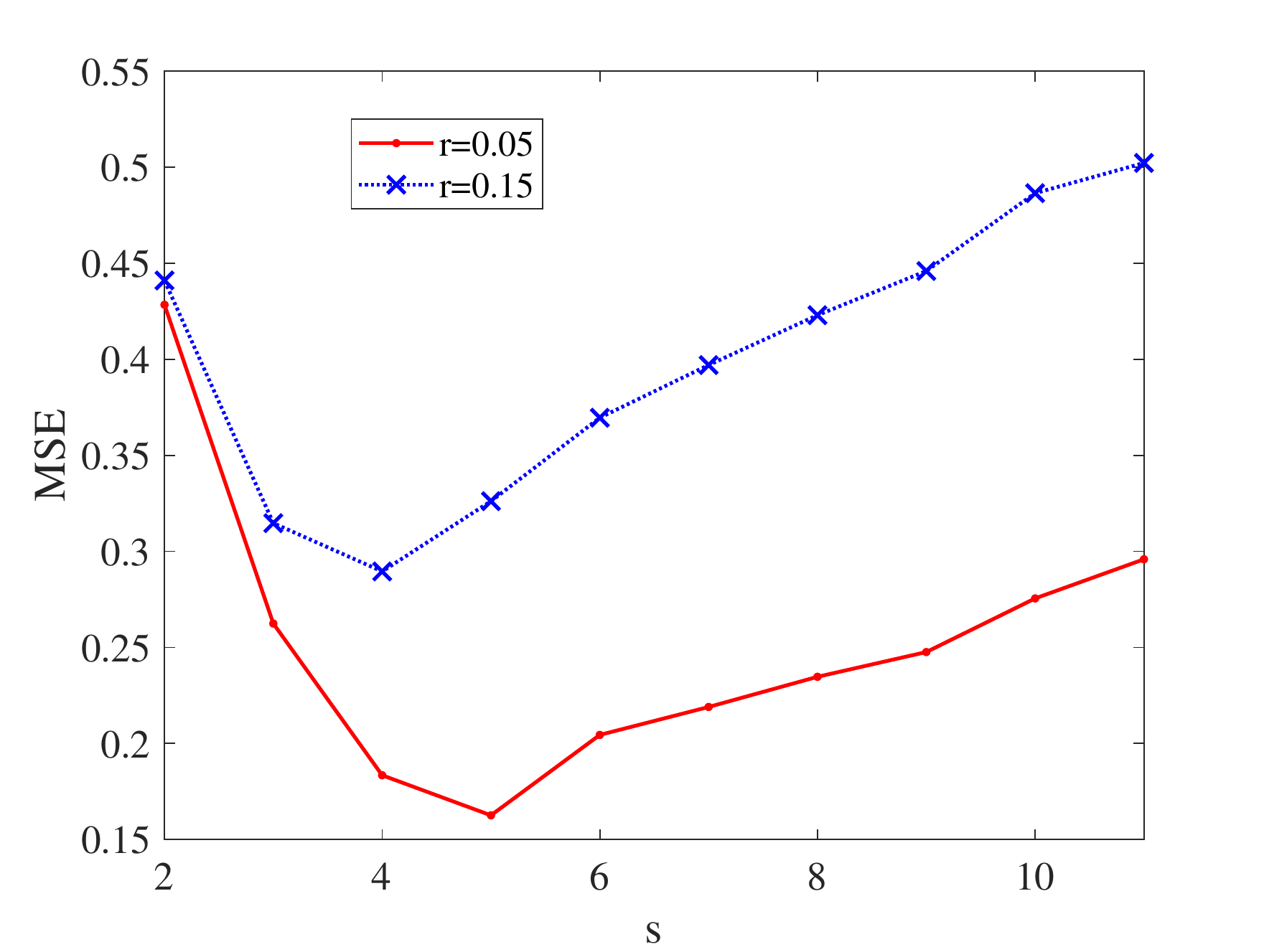}
   \caption{MSE of the solution yielded by GPSP with different $s$ (\small the data generated
   in the same way as in Section 5.1 with $(m,n,s^*)=(500,1000,5),(\mu,\varpi)=(0.3,0.1)$ and $\Phi$ of type I)}
  \label{fig_K}
 \end{figure}

  The existing loss functions for the one-bit CS are mostly convex, including the one-sided
  $\ell_2$ loss \cite{Jacques13,Yan12}, the linear loss \cite{Plan13b,Zhang14},
  the one-sided $\ell_1$ loss \cite{Jacques13,Yan12,Peng19b}, the pinball loss \cite{HuangS18}
  and the logistic loss \cite{Fang14}. Among others, the one-sided $\ell_1$ loss is closely
  related to the hinge loss function in machine learning \cite{Cucker05,ZhangT04},
  which was reported to have a superior performance to the one-sided $\ell_2$ loss (see \cite{Jacques13}),
  and the pinball loss provides a bridge between the hinge loss and the linear loss.
  One can observe that these convex loss functions all impose a large penalty on the flipped samples,
  which inevitably imposes a negative effect on the solution quality of the model \eqref{znorm-reg}.
  In fact, for the pinball loss in \cite{HuangS18}, when the involved parameter $\tau$
  is closer to $0$, the penalty degree on the flipped samples becomes smaller.
  This partly accounts for $\tau=-0.2$ instead of $\tau=-1$ used for numerical experiments there.
  Recently, Dai et al. \cite{Dai16} derived a one-sided zero-norm loss by maximizing
  a posteriori estimation of the true signal. This loss function and its lower semicontinuous
  (lsc) majorization proposed there impose a constant penalty for those flipped samples,
  but their combinatorial property brings much difficulty to the solution of the associated
  optimization models. Inspired by the superiority of the ramp loss in SVM \cite{Brooks11,HuangS14},
  in this work we are interested in a more general DC loss: 
  \begin{equation}\label{DC-loss}
    L_{\sigma}(z):=\sum_{i=1}^m\vartheta_{\!\sigma}(z_i)
   \ \ {\rm with}\ \vartheta_{\!\sigma}(t):=
   \left\{\begin{array}{cl}
    \max(0,-t)&{\rm if}\ t\ge-\sigma,\\
    \sigma&{\rm if}\ t<-\sigma,
   \end{array}\right.
  \end{equation}
  where $\sigma\in(0,1]$ is a constant representing the penalty degree imposed on the flip outlier.
  Clearly, the DC function $\vartheta_{\!\sigma}$ imposes a small fixed penalty for those flip outliers.

  Due to the nonconvexity of the zero-norm and the sphere constraint, some researchers
  are interested in the convex relaxation of \eqref{znorm-reg} obtained by replacing
  the zero-norm by the $\ell_1$-norm and the unit sphere constraint by the unit ball constraint;
  see \cite{Zhang14,HuangS18,Laska11,Plan13b}. However, as in the conventional CS,
  the $\ell_1$-norm convex relaxation not only has a weak sparsity-promoting ability
  but also leads to a biased solution; see the discussion in \cite{Fan01}. Motivated by this,
  many researchers resort to the nonconvex surrogate functions of the zero-norm,
  such as the minimax concave penalty (MCP) \cite{Zhu15,HuangY18},
  the sorted $\ell_1$ penalty \cite{HuangY18}, the logarithmic
  smoothing functions \cite{Shen16}, the $\ell_q\,(0<\!q\!<1)$-norm \cite{Fan21},
  and the Schur-concave functions \cite{Peng19a}, and then develop algorithms for
  solving the associated nonconvex surrogate problems to achieve a better sparse solution.
  To the best of our knowledge, most of these algorithms are lack of convergence certificate.
  Although many nonconvex surrogates of the zero-norm are used for the one-bit CS,
  there is no work to investigate the equivalence between the surrogate problems
  and the model \eqref{znorm-reg} in a global sense.
%----------------------------------------------------------------------------------------------
 \subsection{Main contributions}\label{sec1.2}

  The nonsmooth DC loss $L_{\sigma}$ is desirable to ensure data fidelity and tolerate
  the existence of sign flips, but its nonsmoothness is inconvenient for the solution
  of the associated regularization model \eqref{znorm-reg}. With $0<\!\gamma<\!{\sigma}/{2}$
  we construct a smooth approximation to it:
 \begin{equation}\label{Lsig-gam}
  L_{\sigma,\gamma}(z):=\sum_{i=1}^m\vartheta_{\sigma,\gamma}(z_i)
  \ \ {\rm with}\ \
  \vartheta_{\sigma,\gamma}(t)\!:=\!
    \left\{\begin{array}{cl}
     0           &{\rm if}\ t>0,\\
    t^2/(2\gamma)&{\rm if}\ -\!\gamma<t\le 0,\\
   -t-\gamma/2&{\rm if}\ -\!\sigma\!+\!\gamma<t<-\gamma,\\
    \!\sigma\!-\!\frac{\gamma}{2}\!-\!\frac{(t+\sigma+\gamma)^2}{4\gamma}&{\rm if}\ -\!(\sigma\!+\!\gamma)\le t\le\!\gamma\!-\!\sigma,\\
   \!\sigma\!-\!{\gamma}/2&{\rm if} \  t<-(\sigma\!+\!\gamma).
   \end{array}\right.
 \end{equation}
 Clearly, as the parameter $\gamma$ approaches to $0$, $\vartheta_{\!\sigma,\gamma}$ is closer to $\vartheta_{\!\sigma}$.
 As illustrated in Figure \ref{fig_gam}, the smooth function $\vartheta_{\!\sigma,\gamma}$ approximates
 $\vartheta_{\!\sigma}$ very well even with $\gamma=0.05$. Therefore, in this paper
 we are interested in the zero-norm regularized smooth DC loss model
 \begin{equation}\label{znorm-Moreau}
  \min_{x\in\mathbb{R}^n}F_{\sigma,\gamma}(x):=L_{\sigma,\gamma}(Ax)+\delta_{\mathcal{S}}(x)+\lambda\|x\|_0,
 \end{equation}
 where $\mathcal{S}$ denotes a unit sphere whose dimension is known from the context,
 and $\delta_{\mathcal{S}}$ means the indicator function of $\mathcal{S}$, i.e.,
 $\delta_{\mathcal{S}}(x)=0$ if $x\in\mathcal{S}$ and otherwise $\delta_{\mathcal{S}}(x)=+\infty$.
 %----------------------------------------------------------------------------------------------
 \begin{figure}[!h]
   \centering
  \includegraphics[width=8cm,height=5cm]{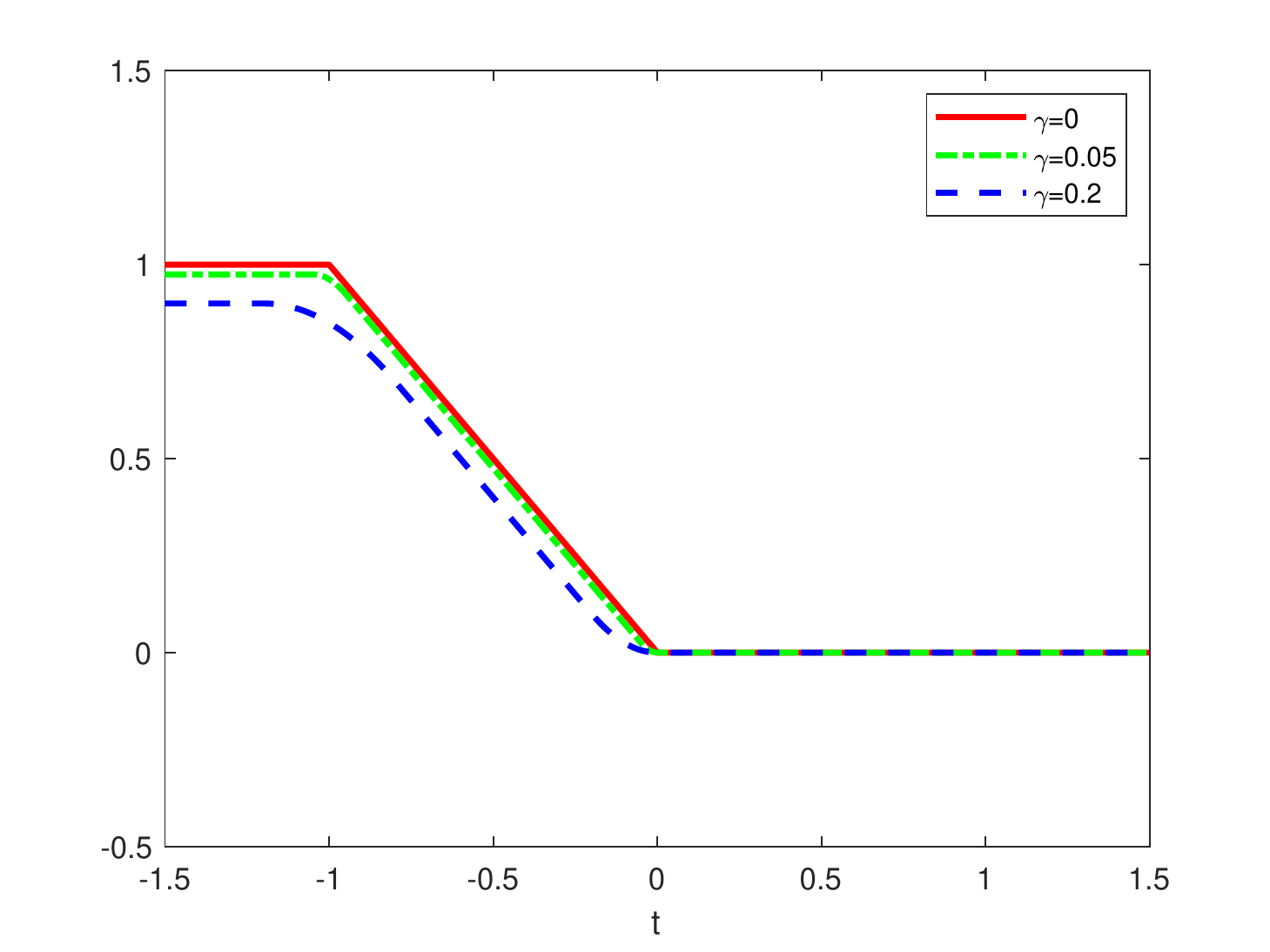}
   \caption{Approximation degree of $\vartheta_{\!\sigma,\gamma}$ with different $\gamma$ to $\vartheta_{\!\sigma}$ for $\sigma=1$}
  \label{fig_gam}
 \end{figure}
 \noindent

 Let $\mathscr{L}$ denote the family of proper lsc convex functions
 $\phi\!:\mathbb{R}\to(-\infty,+\infty]$ satisfying
 \begin{equation}\label{phi-assump}
   {\rm int}({\rm dom}\,\phi)\supseteq[0,1],\ 1>t^*\!:=\mathop{\arg\min}_{0\le t\le 1}\phi(t),\ \phi(t^*)=0
   \ \ {\rm and}\ \ \phi(1)=1.
 \end{equation}
 With an arbitrary $\phi\in\!\mathscr{L}$, the model \eqref{znorm-Moreau} is reformulated
 as a mathematical program with an equilibrium constraint (MPEC) in Section \ref{sec3},
 and by studying its global exact penalty induced by the equilibrium constraint,
 we derive a family of equivalent surrogates
 \begin{equation}\label{Esurrogate}
  \mathop{\min}_{x\in\mathbb{R}^n}G_{\sigma,\gamma,\rho}(x)
  :=L_{\sigma,\gamma}(Ax)+\delta_{\mathcal{S}}(x)+\lambda\rho\varphi_{\rho}(x),
 \end{equation}
 in the sense that the problem \eqref{Esurrogate} associated to every $\rho>\overline{\rho}$
 has the same global optimal solution set as the problem \eqref{znorm-Moreau} does.
 Here $\varphi_{\rho}(x)\!:=\!\|x\|_1-\!\frac{1}{\rho}\!\sum_{i=1}^n\psi^*(\rho|x_i|)$ with
 $\rho>0$ being the penalty parameter and $\psi^*$ being the conjugate function of $\psi$:
 \[
   \psi^*(\omega):=\sup_{t\in\mathbb{R}}\big\{\omega t-\psi(t)\big\}
   \ \ {\rm for}\ \psi(t)\!:=\!\left\{\begin{array}{cl}
                  \phi(t)&{\rm if}\ t\in[0,1],\\
                  +\infty &{\rm otherwise}.
             \end{array}\right.
 \]
 This family of equivalent surrogates is illustrated to include the one associated to
 the MCP function (see \cite{Zhang10,Zhu15,HuangY18}) and the SCAD function \cite{Fan01}.
 The SCAD function corresponds to $\phi(t)=\frac{a-1}{a+1}t^2+\frac{2}{a+1}t\ (a>1)$
 for $t\in\mathbb{R}$, whose conjugate has the form
 \begin{equation}\label{psi-star}
   \psi^*(\omega)=\begin{cases}
        0&  \omega\le \frac{2}{a+1},\\
        \frac{((a+1)\omega-2)^2}{4(a^2-1)}&  \frac{2}{a+1}<\omega\le\frac{2a}{a+1},\\
        \omega-1&   \omega>\frac{2a}{a+1},
  \end{cases}\quad{\rm for}\ \omega\in\mathbb{R}.
 \end{equation}
 Figure \ref{fig_rho} below shows that $G_{\sigma,\gamma,\rho}$ with $\psi^*$ in \eqref{psi-star}
 approximates $F_{\sigma,\gamma}$ every well for $\rho\ge 2$, though the model \eqref{Esurrogate}
 has the same global optimal solution set as the model \eqref{znorm-Moreau} does
 only when $\rho$ is over the theoretical threshold $\overline{\rho}$.
 Unless otherwise stated, the function $G_{\sigma,\gamma,\rho}$ appearing in the rest of
 this paper always represents the one associated to $\psi^*$ in \eqref{psi-star}.
%----------------------------------------------------------------------------------------------
 \begin{figure}[H]
   \centering
  \includegraphics[width=8cm,height=5cm]{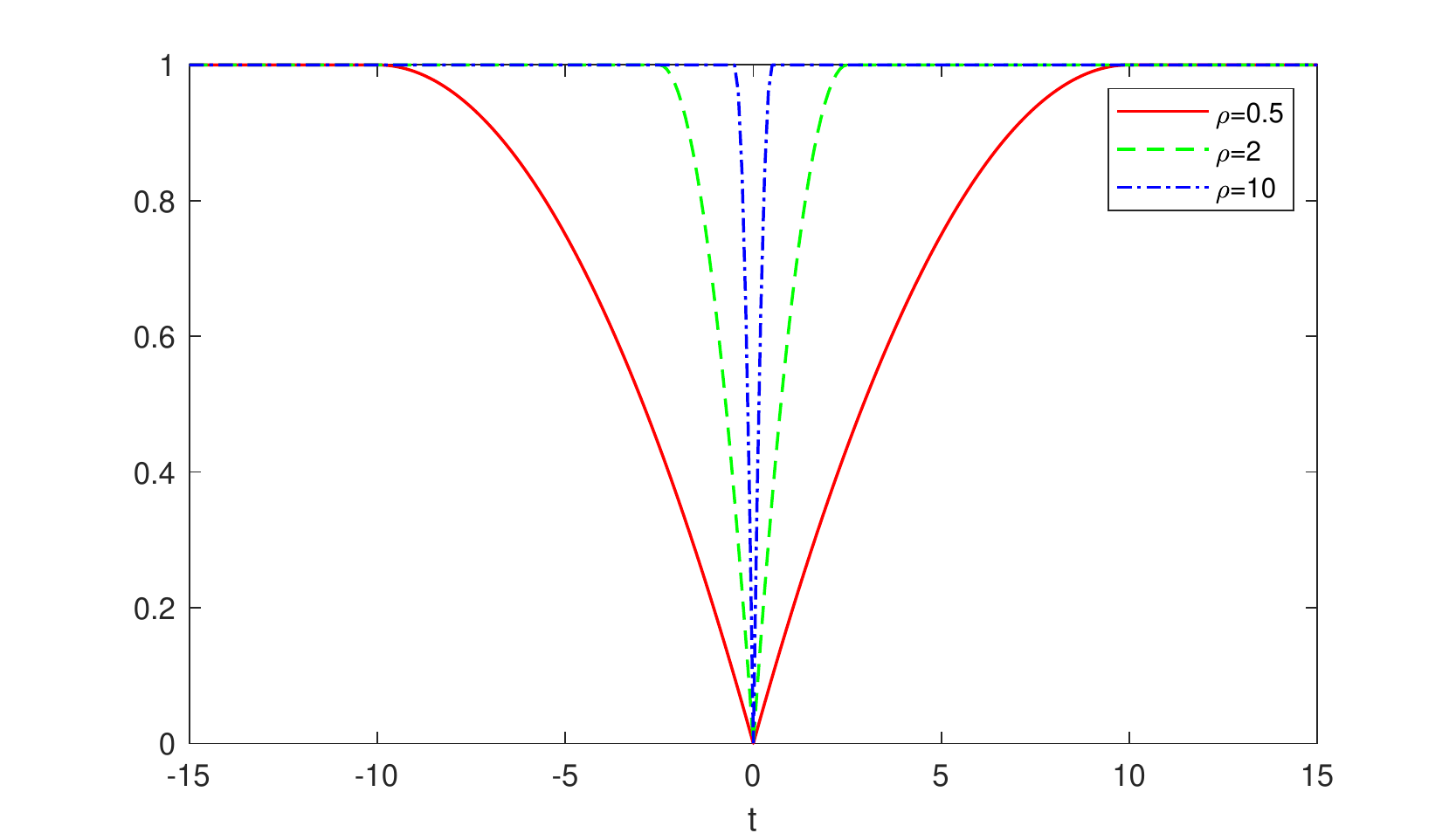}
   \caption{Approximation effect of $|t|-\rho^{-1}\psi^*(\rho|t|)$ with $\psi^*$ in \eqref{psi-star} to ${\rm sign}(|t|)$}
  \label{fig_rho}
 \end{figure}

 For the nonconvex nonsmooth optimization problems \eqref{znorm-Moreau} and \eqref{Esurrogate},
 we develop a proximal gradient (PG) method with extrapolation to solve them, establish
 the convergence of the whole iterate sequence generated, and analyze its local linear convergence rate
 under a mild condition. The main contributions of this paper can be summarized as follows:
 \begin{itemize}
  \item [{\bf (i)}] We introduce a smooth DC loss function well suited for the data with
                    a high sign flip ratio, and propose the zero-norm regularized smooth DC
                    loss model \eqref{znorm-Moreau} which, unlike those models in
                    \cite{Boufounos09,Dai16,Yan12,Zhou21}, does not require any priori information
                    on the sparsity of the true signal and the number of sign flips. In particular,
                    a family of equivalent nonconvex surrogates is derived for the model \eqref{znorm-Moreau}.
                    We also introduce a class of $\tau$-stationary points for the model \eqref{znorm-Moreau}
                   and its equivalent surrogate \eqref{Esurrogate} associated to $\psi^*$ in \eqref{psi-star},
                   which is stronger than the limiting critical points of the corresponding objective functions.

  \item[{\bf(ii)}] By characterizing the closed form of the proximal operators of
                   $\delta_{\mathcal{S}}(\cdot)+\lambda\|\cdot\|_0$ and $\delta_{\mathcal{S}}(\cdot)+\lambda\|\cdot\|_1$,
                   we develop a proximal gradient (PG) algorithm with extrapolation for
                   solving the problem \eqref{znorm-Moreau} (PGe-znorm)
                   and its surrogate \eqref{Esurrogate} associated to $\psi^*$ in \eqref{psi-star}
                   (PGe-scad) and establish the convergence of the whole iterate sequences.
                   Also, by analyzing the KL property of exponent $0$ of
                   $F_{\sigma,\gamma}$ and $G_{\sigma,\gamma,\rho}$,
                   the convergence is shown to have a linear rate under a mild condition.
                   It is worth pointing out that to verify if
                   a nonconvex and nonsmooth function has the KL property of exponent not more than $1/2$
                   is not an easy task because there is lack of a criterion for it.

 \item[{\bf(iii)}] Numerical experiments indicate that the proposed models armed with the PGe-znorm and PGe-scad
                  are robust to a large range of $\lambda$, and numerical comparisons with several state-of-art methods
                  demonstrate that the proposed models are well suited for high noise and/or high sign flip ratio.
                  The obtained solutions are remarkably superior to those yielded by other regularization models,
                  and for the data with a high flip ratio they are also superior to those yielded by the models
                  with the true sparsity as an input, in terms of MSE and Hamming error.
 \end{itemize}
%------------------------------------------------------------------------------
 \section{Notation and preliminaries}\label{sec2}

  Throughout this paper, $\overline{\mathbb{R}}$ denotes the extended real number set
  $(-\infty,\infty]$, $I$ and $e$ denote an identity matrix and a vector of all ones,
  whose dimensions are known from the context; and $\{e^1,\ldots,e^n\}$ denotes 
  the orthonormal basis of $\mathbb{R}^n$. For a integer $k>0$, write $[k]:=\{1,2,\ldots,k\}$.
  For a vector $z\in\mathbb{R}^n$, $|z|_{\rm nz}$ denotes the smallest nonzero entry of the vector $|z|$,
  $z^{\downarrow}$ means the vector of the entries of $z$ arranged in a nonincreasing order,
  and $z^{s,\downarrow}$ means the vector $(z_1^{\downarrow},\ldots,z_s^{\downarrow})^{\mathbb{T}}$.
  For given index sets $I\subseteq[m]$ and $J\subseteq[n]$, $A_{I\!J}\in\mathbb{R}^{|I|\times |J|}$
 denotes the submatrix of $A_{J}$ consisting of those rows $A_i$ with $i\in I$, and
 $A_{J}\in\mathbb{R}^{m\times |J|}$ denotes the submatrix of $A$ consisting of those
 columns $A_j$ with $j\in\!J$. For a proper $h\!:\mathbb{R}^n\to\overline{\mathbb{R}}$,
 ${\rm dom}\,h\!:=\!\{z\in\mathbb{R}^n\,|\,h(z)<+\infty\}$ denotes its effective domain,
 and for any given $-\infty<\eta_1<\eta_2<\infty$, $[\eta_1<h<\eta_2]$ represents
 the set $\{x\in\mathbb{R}^n\,|\,\eta_1<h(x)<\eta_2\}$.
  For any $\lambda>0,\rho>0,0<\gamma<\sigma/2$ and any $x\in\mathbb{R}^n$,
  write $\Gamma(x)\!:=\!\{i\in[m]\,|\,-\!\gamma\le (Ax)_i\le0\}\cup
  \{i\in[m]\,|\,-\!\sigma\!-\!\gamma\le (Ax)_i\le\!\gamma\!-\!\sigma\}$ and define
  \begin{align}\label{fvphi}
   f_{\sigma,\gamma}(x):=L_{\sigma,\gamma}(Ax),\ \
   \Xi_{\sigma,\gamma}(x):=f_{\sigma,\gamma}(x)-\lambda{\textstyle\sum_{i=1}^n}\psi^*(\rho|x|_i),\\
   \label{ghlambda}
   g_{\lambda}(x):=\delta_{\mathcal{S}}(x)+\lambda\|x\|_0\ \ {\rm and}\ \
   h_{\lambda,\rho}(x):=\delta_{\mathcal{S}}(x)+\lambda\rho\|x\|_1.
  \end{align}
 For a proper lsc $h\!:\mathbb{R}^n\to\overline{\mathbb{R}}$,
  the proximal mapping of $h$ associated to $\tau>0$ is defined as
  \[
    \mathcal{P}_{\!\tau}h(x)\!:=\mathop{\arg\min}_{z\in\mathbb{R}^n}\big\{\frac{1}{2\tau}\|z-x\|^2+h(z)\big\}
    \quad\ \forall x\in\mathbb{R}^n.
  \]
  When $h$ is convex, $\mathcal{P}_{\!\tau}h$ is a Lipschitz continuous mapping
  with modulus $1$. When $h$ is an indicator function of a closed set $C\subseteq\mathbb{R}^n$,
  $\mathcal{P}_{\!\tau}h$ is the projection mapping $\Pi_{C}$ onto $C$.
%--------------------------------------------------------------------------------
 \subsection{Proximal mappings of $g_{\lambda}$ and $h_{\lambda,\rho}$}\label{sec2.1}

  To characterize the proximal mapping of the nonconvex nonsmooth function $g_{\lambda}$,
  we need the following lemma, whose proof is not included due to the simplicity.
%---------------------------------------------------------------------------------
 \begin{lemma}\label{lemma-glam}
  Fix any $z\in\mathbb{R}^n\backslash\{0\}$ and an integer $s\ge 1$. Consider the following problem
  \begin{equation}\label{znorm-prob1}
   S^*(z):=\mathop{\arg\min}_{x\in\mathbb{R}^n}\Big\{\frac{1}{2}\|x-z\|^2\ \ {\rm s.t.}\ \ \|x\|=1,\|x\|_0=s\Big\}.
  \end{equation}
  Then, $S^*(z)=\big\{\frac{P^{\mathbb{T}}(|z|^{s-1,\downarrow};|z|_i;0)}{\|(|z|^{s-1,\downarrow};|z|_i;0)\|}
  \,|\ i\in\{s,\ldots,n\}\ {\rm is\ such\ that}\ |z|_i=|z|_{s}^{\downarrow}\big\}$,
  where $P$ is an $n\times n$ signed permutation matrix such that $Pz=|z|^{\downarrow}$.
 \end{lemma}
%-------------------------------------------------------------------------------------------
 \begin{proposition}\label{proxm-glam}
  Fix any $\lambda>0$ and $\tau>0$. For any $z\in\mathbb{R}^n$, by letting $P$ be an $n\times n$
  signed permutation matrix such that $Pz=|z|^{\downarrow}$, it holds that
  $\mathcal{P}_{\!\tau}g_{\lambda}(z)=P^{\mathbb{T}}\mathcal{Q}_{\tau\lambda}(|z|^{\downarrow})$ with
  \begin{equation}\label{prox-Eprob}
   \mathcal{Q}_{\nu}(y):=\mathop{\arg\min}_{x\in\mathbb{R}^n}\Big\{\frac{1}{2}\|x-y\|^2+\nu\|x\|_0\ \ {\rm s.t.}\ \|x\|=1\Big\}
   \quad\forall y\in\mathbb{R}^n.
  \end{equation}
  For any $y\ne 0$ with $y_1\ge\cdots\ge y_n\ge 0$,
  by letting $\chi_{j}(y)\!:=\!\|y^{j,\downarrow}\|-\!\|y^{j-1,\downarrow}\|$ with
  $y^{0,\downarrow}=0$, $\mathcal{Q}_{\nu}(y)\!=\!\{\frac{y}{\|y\|}\}$ if $\nu\le\!\chi_n(y)$;
  $\mathcal{Q}_{\nu}(y)\!=\!\big\{(\frac{y_i}{|y_i|},0,\ldots,0)^{\mathbb{T}}\,|\,i\in[n]\ {\rm is\ such\ that}\
  y_i=y_1\big\}$ if $\nu\ge\!\chi_1(y)$; otherwise
  $\mathcal{Q}_{\nu}(y)\!:=\!\big\{(\frac{y^{l,\downarrow}}{\|y^{l,\downarrow}\|};0)\ |\ l\in[n]
  \ {\rm is\ such\ that}\ \nu\in(\chi_{l+1}(y),\chi_{l}(y)]\big\}$.
 \end{proposition}
 \begin{proof}
  By the definition of $g_{\lambda}$, for any $z\in\mathbb{R}^n$,
  $\mathcal{P}_{\!\tau}g_{\lambda}(z)=\mathcal{Q}_{\tau\lambda}(z)$.
  Since for any $n\times n$ signed permutation matrix $Q$ and any $z\in\mathbb{R}^n$,
  $\|Qz\|=\|z\|$ and $\|Qz\|_0=\|z\|_0$, it is easy to verify that
  $\mathcal{Q}_{\nu}(z)=P^{\mathbb{T}}\mathcal{Q}_{\nu}(|z|^{\downarrow})$.
  The first part of the conclusions then follows. For the second part,
  we first argue that the following inequality relations hold:
  \begin{equation}\label{vtheta-equa}
   \chi_1(y)\ge\chi_2(y)\ge\cdots\ge \chi_n(y).
  \end{equation}
  Indeed, for each $j\in\{1,2,\ldots,n\!-\!1\}$, from the definition of $y^{j}$,
  it is immediate to have
  \begin{align*}
    \|y^{j}\|^2\!-\!\|y^{j-1}\|^2
    =y_{j}^2\ge y_{j+1}^2 =\|y^{j+1}\|^2\!-\!\|y^{j}\|^2\ \ {\rm and}\ \
   \|y^{j}\|\!+\!\|y^{j-1}\|\le\|y^{j+1}\|\!+\!\|y^{j}\|.
  \end{align*}
  Along with $\chi_{j}(y)=\frac{\|y^{j}\|^2-\|y^{j-1}\|^2}{\|y^{j}\|+\|y^{j-1}\|}$,
  we get $\chi_{j}(y)\ge\chi_{j+1}(y)$ and the relations in \eqref{vtheta-equa} hold.
  Let $\upsilon^*(y)$ denote the optimal value of \eqref{prox-Eprob}.
  Then $\upsilon^*(y)=\min\{\overline{\chi}_{1}(y),\ldots,\overline{\chi}_{n}(y)\}$ with
  \begin{equation}\label{prox-opt}
  \overline{\chi}_{s}(y):=\min_{x\in \mathbb{R}^n}
  \Big\{\frac{1}{2}\|x-y\|^2+\nu\|x\|_0\ \ {\rm s.t.}\ \ \|x\|_0=s,\|x\|=1\Big\}
  \ \ {\rm for}\ s=1,\ldots,n.
  \end{equation}
  From Lemma \ref{lemma-glam}, it follows that
  $\overline{\chi}_{s}(y)=\frac{1}{2}(1+\|y\|^2-2\|y^{s,\downarrow}\|)+\nu s$. Then,
  \[
   \Delta\overline{\chi}_{s}(y):=\overline{\chi}_{s+1}(y)-\overline{\chi}_{s}(y)
    =\|y^{s,\downarrow}\|-\|y^{s+1,\downarrow}\|+\nu=-\chi_{s}(y)+\nu.
  \]
  When $\nu\le\chi_n(y)$, we have $\nu\le\chi_{s}(y)$ for all $s=1,\ldots,n$.
  From the last equation, $\Delta\overline{\chi}_{s}(y)\le 0$ for
  $s=1,\ldots,n-\!1$, which means that
  \(
    \overline{\chi}_1(y)\ge\overline{\chi}_2(y)\ge\cdots\ge\overline{\chi}_n(y).
  \)
  Hence, $\upsilon^*(y)=\overline{\chi}_n(y)$, and $\mathcal{Q}_{\nu}(y)=\{\frac{y}{\|y\|}\}$
  follows by Lemma \ref{lemma-glam}. Using the similar arguments,
  we can obtain the rest of the conclusions.
 \end{proof}

  To characterize the proximal mapping of the nonconvex nonsmooth function $h_{\lambda,\rho}$,
  we need the following lemma, whose proof is omitted due to the simplicity.
 %----------------------------------------------------------------------------------------------
 \begin{lemma}\label{MP-lemma}
  Let $\mathcal{S}_{+}:=\mathcal{S}\cap\mathbb{R}_{+}^n$. For any $z\in\!\mathbb{R}^n$,
  by letting $P$ be an $n\times n$ permutation matrix
  such that $Pz=z^{\downarrow}$, it holds that $\Pi_{\mathcal{S}_{+}}(z)
  =P^{\mathbb{T}}\Pi_{\mathcal{S}_{+}}(z^{\downarrow})$. Also, for any $y\!\in\mathbb{R}^n$
  with $y_1\ge\cdots\ge y_n$, $\Pi_{\mathcal{S}_{+}}(y)=\big\{e_i\,|\,i\in[n]\ {\rm is\ such\ that}\ y_i=y_1\big\}$
  if $y_1\le 0$; $\Pi_{\mathcal{S}_{+}}(y)=\big\{\frac{y}{\|y\|}\big\}$ if $y_n\ge 0$,
  otherwise $\Pi_{\mathcal{S}_{+}}(y)=\big\{\frac{(y_1,\ldots,y_j,0,\ldots,0)^{\mathbb{T}}}
  {\|(y_1,\ldots,y_j,0,\ldots,0)^{\mathbb{T}}\|}\,|\, j\in[n\!-\!1]\ {\rm is\ such\ that}\ y_j>0\ge y_{j+1}\big\}$.
 \end{lemma}
%----------------------------------------------------------------------------------------------
 \begin{proposition}\label{proxm-hlam}
  Fix any $\lambda>0,\rho>0$ and $\tau>0$. For any $z\in\mathbb{R}^n$, by letting $P$
  be an $n\times n$ signed permutation matrix with $Pz=|z|^{\downarrow}$,
  $\mathcal{P}_{\!\tau}h_{\lambda,\rho}(z)=P^{\mathbb{T}}\Pi_{\mathcal{S}_{+}}(|z|^{\downarrow}\!-\!\tau\lambda\rho e)$.
 \end{proposition}
 \begin{proof}
  Fix any $\xi\in\mathbb{R}^n$ with $\xi_1\ge \xi_2\ge\cdots\ge \xi_n\ge 0$.
  Consider the following problem
  \begin{equation}\label{MQ-map2}
  \mathcal{P}_{\nu}(\xi):=\mathop{\arg\min}_{x\in\mathbb{R}^n}
  \Big\{\frac{1}{2}\big\|x-\xi\big\|^2+\nu\|x\|_1\ \ {\rm s.t.}\ \|x\|=1\Big\}
  \end{equation}
  where $\nu>0$ is a regularization parameter. By the definition of $h_{\lambda,\rho}$,
  $\mathcal{P}_{\!\tau}h_{\lambda,\rho}(z)=\mathcal{P}_{\!\tau\lambda\rho}(z)$,
  so it suffices to argue that $\mathcal{P}_{\!\nu}(\xi)=\Pi_{\mathcal{S}_{+}}(\xi-\nu e)$.
  Indeed, if $x^*$ is a global optimal solution of \eqref{MQ-map2},
  then $x^*\ge 0$ necessarily holds. If not, we will have $J:=\{j\,|\,x_j^*<0\}\ne\emptyset$.
  Let $\overline{J}=\{1,\ldots,n\}\backslash J$. Take $\widetilde{x}_i^*=x_i^*$
  for each $i\in\overline{J}$ and $\widetilde{x}_i^*=-x_i^*$ for each $i\in J$.
  Clearly, $\widetilde{x}^*\ge 0$ and $\|\widetilde{x}^*\|=1$. However, it holds that
  $\frac{1}{2}\big\|\widetilde{x}^*-\xi\big\|^2+\nu\|\widetilde{x}^*\|_1
    \le\frac{1}{2}\big\|x^*-\xi\big\|^2+\nu\|x^*\|_1$,
  which contradicts the fact that $x^*$ is a global optimal solution of \eqref{MQ-map2}.
  This implies that
  \(
    \mathcal{P}_{\nu}(\xi)=\mathop{\arg\min}_{x\in\mathbb{R}^n}
    \big\{\frac{1}{2}\big\|x-\xi\big\|^2+\nu\langle e,x\rangle\ \ {\rm s.t.}\ x\ge 0,\|x\|=1\big\}.
  \)
  Consequently, $\mathcal{P}_{\nu}(\xi)=\Pi_{\mathcal{S}_{+}}(\xi-\nu e)$.
  The desired equality then follows.
 \end{proof}
%-------------------------------------------------------------------------------------
 \subsection{Generalized subdifferentials}\label{sec2.2}
%-------------------------------------------------------------------------------------
 \begin{definition}\label{Gsubdiff-def}(see \cite[Definition 8.3]{RW98})\
  Consider a function $h\!:\mathbb{R}^n\to\overline{\mathbb{R}}$ and a point $x\in{\rm dom}h$.
  The regular subdifferential of $h$ at $x$, denoted by $\widehat{\partial}h(x)$, is defined as
  \[
    \widehat{\partial}h(x):=\bigg\{v\in\mathbb{R}^n\ \big|\
    \liminf_{x'\to x\atop x'\ne x}\frac{h(x')-h(x)-\langle v,x'-x\rangle}{\|x'-x\|}\ge 0\bigg\};
  \]
  and the (limiting) subdifferential of $h$ at $x$, denoted by $\partial h(x)$, is defined as
  \[
    \partial h(x):=\Big\{v\in\mathbb{R}^n\,|\,\exists\,x^k\to x\ {\rm with}\ h(x^k)\to h(x)\ {\rm and}\
    v^k\in\widehat{\partial}h(x^k)\to v\ {\rm as}\ k\to\infty\Big\}.
  \]
 \end{definition}
%-------------------------------------------------------------------------------
 \begin{remark}\label{remark-Fsubdiff}
  {\bf(i)} At each $x\in{\rm dom}h$, $\widehat{\partial}h(x)\subseteq\partial h(x)$,
  $\widehat{\partial}h(x)$ is always closed and convex, and $\partial h(x)$ is closed
  but generally nonconvex.
  When $h$ is convex, $\widehat{\partial}h(x)=\partial h(x)$, which is precisely
  the subdifferential of $h$ at $x$ in the sense of convex analysis.

  \noindent
  {\bf(ii)} Let $\{(x^k,v^k)\}_{k\in\mathbb{N}}$ be a sequence in the graph of
  $\partial h$ that converges to $(x,v)$ as $k\to\infty$.
  By invoking Definition \ref{Gsubdiff-def}, if $h(x^k)\to h(x)$ as $k\to\infty$,
  then $v\in\partial h(x)$.

  \noindent
  {\bf(iii)} A point $\overline{x}$ at which $0\in\partial h(\overline{x})$
  ($0\in\widehat{\partial}h(\overline{x})$) is called  a limiting (regular)
  critical point of $h$. In the sequel, we denote by ${\rm crit}\,h$
  the limiting critical point set of $h$.
 \end{remark}

 When $h$ is an indicator function of a closed set $C$,
 the subdifferential of $h$ at $x\in C$ is the normal cone to $C$ at $x$,
 denoted by $\mathcal{N}_{C}(x)$. The following lemma characterizes the (regular)
 subdifferentials of $F_{\sigma,\gamma}$ and $G_{\sigma,\gamma,\rho}$ at any point
 of their domains.
%------------------------------------------------------------------------------------------
 \begin{lemma}\label{lemma-critical}
  Fix any $\lambda>0,\rho>0$ and $0<\!\gamma<\!{\sigma}/{2}$. Consider any $x\in\mathcal{S}$. Then,
  \begin{itemize}
   \item[(i)] $f_{\sigma,\gamma}$ is a smooth function whose gradient $\nabla\!f_{\sigma,\gamma}$ 
              is Lipschitz continuous with the modulus $L_{\!f}\le\frac{1}{\gamma}\|A\|^2$.

  \item [(ii)] $\widehat{\partial}F_{\sigma,\gamma}(x)=\partial F_{\sigma,\gamma}(x)=\nabla\!f_{\sigma,\gamma}(x)
                +\mathcal{N}_{\mathcal{S}}(x)+\lambda\partial\|x\|_0$.

 \item[(iii)] $\widehat{\partial}G_{\sigma,\gamma,\rho}(x)
              =\partial G_{\sigma,\gamma,\rho}(x)\!=\!\nabla\Xi_{\sigma,\gamma}(x)
                \!+\!\mathcal{N}_{\mathcal{S}}(x)+\lambda\rho\partial\|x\|_1$.
 \item[(iv)] When $|x|_{\rm nz}\ge\frac{2a}{\rho(a-1)}$,
             it holds that $\partial G_{\sigma,\gamma,\rho}(x)\subseteq\partial F_{\sigma,\gamma}(x)$.
  \end{itemize}
 \end{lemma}
 \begin{proof}
  {\bf(i)} The result is immediate by the definition of $f_{\sigma,\gamma}$ and the expression of $L_{\sigma,\gamma}$.

  \noindent
  {\bf(ii)} From \cite[Lemma 3.1-3.2 \& 3.4]{WuPanBi21},
  $\widehat{\partial}g_{\lambda}(x)=\partial g_{\lambda}(x)=\mathcal{N}_{\mathcal{S}}(x)+\lambda\partial\|x\|_0$.
  Together with part (i) and \cite[Exercise 8.8]{RW98}, we obtain the desired result.

  \noindent
  {\bf(iii)} By the convexity and Lipschitz continuity of $\ell_1$-norm and \cite[Exercise 10.10]{RW98},
  it follows that $\partial h_{\lambda,\rho}(x)\!=\mathcal{N}_{\mathcal{S}}(x)+\lambda\rho\partial\|x\|_1$.
  Let $\theta_{\!\rho}(z)\!:=\!\rho^{-1}\sum_{i=1}^n\psi^*(\rho|z_i|)$ for $z\in\mathbb{R}^n$.
  Clearly, $\Xi_{\sigma,\gamma}=f_{\sigma,\gamma}-\lambda\rho\theta_{\!\rho}$.
  By the expression of $\psi^*$ in \eqref{psi-star}, it is easy to verify that
  $\theta_{\!\rho}$ is smooth and $\nabla\theta_{\!\rho}$ is Lipschitz continuous
  with modulus $\rho\max(\frac{a+1}{2},\frac{a+1}{2(a-1)})$. Hence,
  $\Xi_{\sigma,\gamma}$ is a smooth function whose gradient is Lipschitz continuous.
  Together with \cite[Exercise 8.8]{RW98} and $G_{\sigma,\gamma,\rho}=\Xi_{\sigma,\gamma}+h_{\lambda,\rho}$,
  we obtain the desired equalities.

  \noindent
  {\bf(iv)} Let $\theta_{\!\rho}$ be the function defined as above.
  After an elementary calculation, we have
  \[
   \nabla\theta_{\!\rho}(x)\!=\big((\psi^*)'(\rho|x_1|){\rm sign}(x_1),\ldots,
   (\psi^*)'(\rho|x_n|){\rm sign}(x_n)\big)^{\mathbb{T}}.
  \]
  Along with $|x|_{\rm nz}\ge\frac{2a}{\rho(a-1)}$ and the expression of $\psi^*$ in \eqref{psi-star},
  we have $\nabla\theta_{\!\rho}(x)={\rm sign}(x)$ and
  $\partial\|x\|_1\!-\!\nabla\theta_{\!\rho}(x)\subseteq\rho\partial\|x\|_0$.
  By part (iii), $\nabla\Xi_{\sigma,\gamma}(x)=\nabla\!f_{\sigma,\gamma}(x)-\lambda\rho\nabla\theta_{\!\rho}(x)$.
  Comparing $\partial G_{\sigma,\gamma,\rho}(x)$ in part (iii) with $\partial F_{\sigma,\gamma}(x)$
  in part (ii) yields that $\partial G_{\sigma,\gamma,\rho}(x)\subseteq\partial F_{\sigma,\gamma}(x)$.
 \end{proof}
%-------------------------------------------------------------------------------------
 \subsection{Stationary points}\label{sec2.3}

  Lemma \ref{lemma-critical} shows that for the functions $F_{\sigma,\gamma}$ and $G_{\sigma,\gamma,\rho}$
  the set of their regular critical points coincides with that of their limiting critical points,
  so we call the critical point of $F_{\sigma,\gamma}$ a stationary point of \eqref{znorm-Moreau},
  and the critical point of $G_{\sigma,\gamma,\rho}$ a stationary point of \eqref{Esurrogate}.
  Motivated by the work \cite{Beck19}, we introduce a class of $\tau$-stationary points for them.
%----------------------------------------------------------------------------------------
 \begin{definition}\label{tau-critical}
  Let $\tau>0$. A vector $x\in\mathbb{R}^n$ is called a $\tau$-stationary point of
  \eqref{znorm-Moreau} if $x\in\mathcal{P}_{\!\tau}g_{\lambda}(x\!-\!\tau\nabla\!f_{\sigma,\gamma}(x))$,
  and is called a $\tau$-stationary point of \eqref{Esurrogate} if
  $x\in\mathcal{P}_{\!\tau}h_{\lambda,\rho}(x\!-\!\tau\nabla\Xi_{\sigma,\gamma}(x))$.
 \end{definition}

 In the sequel, we denote by $S_{\tau,g_{\lambda}}$ and $S_{\tau,h_{\lambda,\rho}}$
 the $\tau$-stationary point set of \eqref{znorm-Moreau} and \eqref{Esurrogate}, respectively.
 By Proposition \ref{proxm-glam} and \ref{proxm-hlam}, we have the following result for them.
%-----------------------------------------------------------------------------------------------
 \begin{lemma}\label{relation-critical}
  Fix any $\tau>0,\lambda>0,\rho>0$ and $0<\!\gamma<\!{\sigma}/{2}$. Then, $S_{\tau,g_{\lambda}}\subseteq{\rm crit}F_{\sigma,\gamma}$
  and $S_{\tau,h_{\lambda,\rho}}\subseteq{\rm crit}G_{\sigma,\gamma,\rho}$.
 \end{lemma}
 \begin{proof}
  Pick any $\overline{x}\in\!S_{\tau,g_{\lambda}}$. Then 
  $\overline{x}=\mathcal{P}_{\!\tau}g_{\lambda}(\overline{x}\!-\!\tau\nabla\!f_{\sigma,\gamma}(\overline{x}))$.
  By Proposition \ref{proxm-glam}, for each $i\in{\rm supp}(\overline{x})$,
  $\overline{x}_i=\alpha^{-1}[\overline{x}_i\!-\!\tau(\nabla\!f_{\sigma,\gamma}(\overline{x}))_i]$
  for some $\alpha\!>0$ (depending on $\overline{x}$). Then, for each $i\in{\rm supp}(\overline{x})$,
  it holds that $(\nabla\!f_{\sigma,\gamma}(\overline{x}))_i+\tau^{-1}(\alpha\!-\!1)\overline{x}_i=0$.
  Recall that
  \begin{equation}\label{subdiff-sphere}
     \mathcal{N}_{\mathcal{S}}(\overline{x})=\big\{\beta\overline{x}\,|\,\beta\in\mathbb{R}\big\}
     \ \ {\rm and}\ \
    \partial\|\overline{x}\|_0=\big\{v\in\mathbb{R}^n\,|\, v_i=0\ {\rm for}\ i\in{\rm supp}(\overline{x})\big\}.
  \end{equation}
  We have $0\in \nabla\!f_{\sigma,\gamma}(\overline{x})+\mathcal{N}_{\mathcal{S}}(\overline{x})+\lambda\partial\|\overline{x}\|_0$,
  and hence $\overline{x}\in{\rm crit}F_{\sigma,\gamma}$ by Lemma \ref{lemma-critical} (ii).
  
  Pick any $\overline{x}\in\!S_{\tau,h_{\lambda,\rho}}$.
  Write $\overline{u}=\overline{x}\!-\!\tau\nabla\Xi_{\sigma,\gamma}(\overline{x})$. 
  Then, we have $\overline{x}=\mathcal{P}_{\!\tau}h_{\lambda,\rho}(\overline{u})$.
  Let $J\!:=\{i\in[n]\,|\,|\overline{u}_i|>\tau\lambda\rho\}$ and $\overline{J}=[n]\backslash J$.
  For each $i\in\overline{J}$,
  $|\nabla\Xi_{\sigma,\gamma}(\overline{x})-\tau^{-1}\overline{x}_i|\le\lambda\rho$.
  Since the subdifferential of the function $t\mapsto |t|$ at $0$ is $[-1,1]$,
  it holds that
  \[
    0\in[\nabla\Xi_{\sigma,\gamma}(\overline{x})]_{\overline{J}}-\tau^{-1}\overline{x}_{\overline{J}}
    +\lambda\rho\partial\|\overline{x}_{\overline{J}}\|_1.
  \]
   By Proposition \ref{proxm-hlam}, we have
  \(
    \overline{x}_J=\frac{\overline{u}_J-\tau\lambda\rho{\rm sign}(\overline{u}_J)}
    {\|\overline{u}_J-\tau\lambda\rho{\rm sign}(\overline{u}_J)\|}.
  \)
  Together with ${\rm sign}(\overline{u}_J)={\rm sign}(\overline{x}_J)$,
  \[
    (\nabla\Xi_{\sigma,\gamma}(\overline{x}))_J
    +\tau^{-1}(\|\overline{u}_J\!-\!\tau\lambda\rho{\rm sign}(\overline{u}_J)\|\!-\!1)\overline{x}_J
    +\lambda\rho{\rm sign}(\overline{x}_J)=0.
  \]
  By the expression of $\mathcal{N}_{\mathcal{S}}(\overline{x})$ in \eqref{subdiff-sphere},
  from the last two equations it follows that
  \[
    0\in \nabla\Xi_{\sigma,\gamma}(\overline{x})+\mathcal{N}_{\mathcal{S}}(\overline{x})+\lambda\rho\partial\|\overline{x}\|_1.
  \]
  By Lemma \ref{lemma-critical} (iii), this shows that $\overline{x}\in{\rm crit}G_{\sigma,\gamma,\rho}$.
  The proof is completed.
  \end{proof}

  Note that if $\overline{x}$ is a stationary point of \eqref{znorm-Moreau},
  then for $i\notin {\rm supp}(\overline{x})$, $[\mathcal{P}_{\!\tau}g_{\lambda}(\overline{x}\!-\!\tau\nabla\!f_{\sigma,\gamma}(\overline{x}))]_i$
  does not necessarily equal $0$. A similar case also occurs for the stationary point of \eqref{Esurrogate}.
  This means that the two inclusions in Lemma \ref{relation-critical} are generally strict.
  By combining Lemma \ref{relation-critical} with \cite[Theorem 10.1]{RW98},
  it is immediate to obtain the following conclusion.
%-------------------------------------------------------------------------------------------
  \begin{corollary}\label{corollary-opt}
   Fix any $\tau>0$. For the problems \eqref{znorm-Moreau} and \eqref{Esurrogate},
   their local optimal solution is necessarily is a stationary point,
   and consequently a $\tau$-stationary point.
  \end{corollary}

 \subsection{Kurdyka-{\L}\"{o}jasiewicz property}\label{sec2.3}
%-------------------------------------------------------------------------------------
 \begin{definition}\label{KL-def}
  (see \cite{Attouch10}) A proper lsc function $h\!:\mathbb{R}^n\to\overline{\mathbb{R}}$ is said to have
  the KL property at $\overline{x}\in{\rm dom}\,\partial h$ if there exist $\eta\in(0,+\infty]$,
  a neighborhood $\mathcal{U}$ of $\overline{x}$, and a continuous concave function
  $\varphi\!:[0,\eta)\to\mathbb{R}_{+}$ that is continuously differentiable on $(0,\eta)$
  with $\varphi'(s)>0$ for all $s\in(0,\eta)$ and $\varphi(0)=0$, such that
  for all $x\in\mathcal{U}\cap\big[h(\overline{x})<h<h(\overline{x})+\eta\big]$,
  \[
    \varphi'(h(x)-h(\overline{x})){\rm dist}(0,\partial h(x))\ge 1.
  \]
   If $\varphi$ can be chosen as $\varphi(t)=ct^{1-\theta}$ with $\theta\in[0,1)$
  for some $c>0$, then $h$ is said to have the KL property of exponent $\theta$
  at $\overline{x}$. If $h$ has the KL property (of exponent $\theta$) at each point
  of ${\rm dom}\,\partial h$, then it is called a KL function (of exponent $\theta$).
 \end{definition}
%---------------------------------------------------------------------
 \begin{remark}\label{KL-remark}
  {\bf(a)} As discussed thoroughly in \cite[Section 4]{Attouch10}, there are
  a large number of nonconvex nonsmooth functions are the KL functions,
  which include real semi-algebraic functions and those functions definable
  in an o-minimal structure.

  \noindent
  {\bf(b)} From \cite[Lemma 2.1]{Attouch10}, a proper lsc function has
  the KL property of exponent $\theta=0$ at any noncritical point.
  Thus, to prove that a proper lsc $h\!:\mathbb{R}^n\to\overline{\mathbb{R}}$
  is a KL function (of exponent $\theta$), it suffices to achieve its KL property
  (of exponent $\theta$) at critical points. On the calculation of KL exponent,
  please refer to the recent works \cite{LiPong18,YuLiPong21}.
 \end{remark}
%------------------------------------------------------------------------------
 \section{Equivalent surrogates of the model \eqref{znorm-Moreau}}\label{sec3}

 Pick any $\phi\in\!\mathscr{L}$. By invoking equation \eqref{phi-assump}, it is immediate to verify that
 for any $x\in\mathbb{R}^n$,
 \[
   \|x\|_0=\min_{w\in[0,e]}\Big\{\textstyle{\sum_{i=1}^n}\phi(w_i)\ \ {\rm s.t.}\ \langle e-\!w,|x|\rangle=0\Big\}.
 \]
 This means that the zero-norm regularized problem \eqref{znorm-Moreau} can be reformulated as
 \begin{equation}\label{Eznorm-MPEC}
  \min_{x\in\mathcal{S},w\in[0,e]}\Big\{f_{\sigma,\gamma}(x)+\lambda\textstyle{\sum_{i=1}^n}\phi(w_i)
  \quad\mbox{s.t.}\ \ \langle e-w,|x|\rangle=0\Big\}
 \end{equation}
 in the following sense: if $x^*$ is globally optimal to the problem \eqref{znorm-Moreau},
 then $(x^*\!,{\rm sign}(|x^*|))$ is a global optimal solution of the problem \eqref{Eznorm-MPEC},
 and conversely, if $(x^*,w^*)$ is a global optimal solution of \eqref{Eznorm-MPEC},
 then $x^*$ is globally optimal to \eqref{znorm-Moreau}. The problem \eqref{Eznorm-MPEC}
 is a mathematical program with an equilibrium constraint $e\!-\!w\ge 0,|x|\ge 0,\langle e-w,|x|\rangle=0$.
 In this section, we shall show that the penalty problem induced by this equilibrium constraint, i.e.,
 \begin{equation}\label{Penalty-MPEC}
  \min_{x\in\mathcal{S},w\in[0,e]}\Big\{f_{\sigma,\gamma}(x)+\lambda\textstyle{\sum_{i=1}^n}\phi(w_i)
  +\rho\lambda\langle e-w,|x|\rangle\Big\}
 \end{equation}
 is a global exact penalty of \eqref{Eznorm-MPEC} and from this global exact penalty achieve
 the equivalent surrogate in \eqref{Esurrogate}, where $\rho>0$ is the penalty parameter.
 For each $s\in[n]$, write
 \[
   \Omega_{s}:=\mathcal{S}\cap\mathcal{R}_{s}\ \ {\rm with}\ \ \mathcal{R}_{s}:=\{x\in\mathbb{R}^n\,|\,\|x\|_0\le s\}.
 \]
 To get the conclusion of this section, we need the following global error bound result.
%--------------------------------------------------------------------------------------
 \begin{lemma}\label{calmness}
  For each $s\in\{1,2,\ldots,n\}$, there exists $\kappa_{s}>0$ such that
  for all $x\in\mathcal{S}$,
  \[
    {\rm dist}(x,\Omega_{s})\le\kappa_{s}\big[\|x\|_1-\|x\|_{(s)}\big],
  \]
  where $\|x\|_{(s)}$ denotes the sum of the first $s$ largest entries of the vector $x\in\mathbb{R}^n$.
 \end{lemma}
 \begin{proof}
  Fix any $s\in\{1,2,\ldots,n\}$. We first argue that the following multifunction
  \[
    \Upsilon_{\!s}(\tau):=\big\{x\in\mathcal{S}\,|\,\|x\|_1-\|x\|_{(s)}=\tau\big\}
    \ \ {\rm for}\ \tau\in\mathbb{R}
  \]
  is calm at $0$ for every $x\in\Upsilon_{\!s}(0)$. Pick any $\widehat{x}\in\Upsilon_{\!s}(0)$.
  By \cite[Theorem 3.1]{QianPan21}, the calmness of $\Upsilon_{\!s}$ at $0$ for $\widehat{x}$
  is equivalent to the existence of $\delta>0$ and $\kappa>0$ such that
  \begin{equation}\label{aim-ineq1}
   {\rm dist}(x,\Omega_{s})\le\kappa\big[{\rm dist}(x,\mathcal{S})+{\rm dist}(x,\mathcal{R}_{s})\big]
   \ \ {\rm for\ all}\ x\in\mathbb{B}(\widehat{x},\delta).
  \end{equation}
  Since $\|\widehat{x}\|=1$, there exists $\varepsilon\in(0,1/2)$ such that
  for all $x\in\mathbb{B}(\widehat{x},\varepsilon)$, $x\ne 0$.
  Fix any $x\in\mathbb{B}(\widehat{x},{\varepsilon}/{2})$.
  Clearly, $\|x\|\ge\|\widehat{x}\|-{\varepsilon}/{2}\ge{3}/{4}$.
  This means that $\|x\|_{\infty}\!\ge\frac{3}{4\sqrt{n}}$.
  Pick any $x^*\in\Pi_{\mathcal{R}_{s}}(x)$. Clearly,
  $\|x^*\|_{\infty}=\|x\|_\infty\ge\frac{3}{4\sqrt{n}}$ and $\frac{x^*}{\|x^*\|}\in\Omega_s$.
  Then, with $\overline{x}=\frac{x}{\|x\|}$,
  \begin{align*}
  {\rm dist}(x,\Omega_{s})&\le\|x-x^*\!/{\|x^*\|}\|
  \le\|x-\overline{x}\|+\|\overline{x}-x^*\!/{\|x^*\|}\|\\
  &\le\|x-\overline{x}\|+\frac{\|(x-x^*)\|x\|+x(\|x^*\|-\|x\|)\|}{\|x\|\|x^*\|}\\
  &\le\|x-\overline{x}\|+(2/\|x^*\|)\|x-x^*\|
  \le{\rm dist}(x,\mathcal{S})+3\sqrt{n}{\rm dist}(x,\mathcal{R}_{s}).
  \end{align*}
  This shows that the inequality \eqref{aim-ineq1} holds for $\delta=\varepsilon/2$ and
  $\kappa=3\sqrt{n}$. Consequently, the mapping $\Upsilon_{\!s}$ is calm at $0$ for
  every $x\in\Upsilon_{\!s}(0)$. Now by invoking \cite[Theorem 3.3]{QianPan21}
  and the compactness of $\mathcal{S}$, we obtain the desired result.
  The proof is completed.
  \end{proof}

  Now we are ready to show that the problem \eqref{Penalty-MPEC}
  is a global exact penalty of \eqref{Eznorm-MPEC}.
%-----------------------------------------------------------------------------------------
 \begin{proposition}\label{partial-calm}
  Let $\overline{\rho}:=\frac{\kappa\phi_{-}'(1)(1-t^*)\alpha_{\!f}}{\lambda(1-t_0)}$
  where $t_0\in[0,1)$ is such that $\frac{1}{1-t^*}\in\partial\phi(t_0)$,
  $\phi_{-}'(1)$ is the left derivative of $\phi$ at $1$,
  $\alpha_{\!f}$ is the Lipschitz constant of $f_{\sigma,\gamma}$ on $\mathcal{S}$,
  and $\kappa=\max_{1\le s\le n}\kappa_s$ with $\kappa_s$ given by Lemma \ref{calmness}.
  Then, for any $(x,w)\in\mathcal{S}\times[0,e]$,
  \begin{equation}\label{aim-ineq}
   \big[f_{\sigma,\gamma}(x) +\lambda{\textstyle\sum_{i=1}^n}\phi(w_i)\big]
   -\big[f_{\sigma,\gamma}(x^*) +\lambda{\textstyle\sum_{i=1}^n}\phi(w_i^*)\big]
   +\overline{\rho}\lambda\langle e\!-\!w,|x|\rangle\ge 0,
 \end{equation}
  where $(x^*,w^*)$ is an arbitrary global optimal solution of \eqref{Eznorm-MPEC},
  and consequently the problem \eqref{Penalty-MPEC} associated to each $\rho>\overline{\rho}$
  has the same global optimal solution set as \eqref{Eznorm-MPEC} does.
 \end{proposition}
 \begin{proof}
  By Lemma \ref{calmness} and $\kappa=\max_{1\le s\le n}\kappa_s$,
  for each $s\in\{1,2,\ldots,n\}$ and any $z\in\mathcal{S}$,
  \begin{equation}\label{err-bound}
   {\rm dist}(z,\mathcal{S}\cap\mathcal{R}_s)\le\kappa\big[\|z\|_1-\|z\|_{(k)}\big].
  \end{equation}
  Fix any $(x,w)\in\mathcal{S}\times[0,e]$.
  Let $J=\big\{j\in[n]\,|\ \overline{\rho}|x|_j^{\downarrow}>\phi_{-}'(1)\big\}$ and
  $r=|J|$. By invoking \eqref{err-bound} for $s=r$ with $z=x$,
  there exists $x^{\overline{\rho}}\in\mathcal{S}\cap\mathcal{R}_r$ such that
  \begin{equation}\label{err-ineq1}
    \|x-x^{\overline{\rho}}\|\le\kappa\big[\|x\|_1-\|x\|_{(r)}\big]
    =\kappa{\textstyle\sum_{j=r+1}^n}|x|_j^{\downarrow}.
  \end{equation}
  Let $J_1=\!\big\{j\in[n]\,|\,\frac{1}{1-t^*}\!\le\!\overline{\rho}|x|_j^{\downarrow}\le\phi_{-}'(1)\big\}$
  and $J_2=\!\big\{j\in[n]\,|\,0\!\le\!\overline{\rho}|x|_j^{\downarrow}<\frac{1}{1-t^*}\big\}$.
  Note that
  \[
   {\textstyle\sum_{i=1}^n}\phi(w_i)+\overline{\rho}\big(\|x\|_1\!-\langle w,|x|\rangle\big)
   \ge{\textstyle \sum_{i=1}^n}\min_{t\in[0,1]}\big\{\phi(t)+\overline{\rho}|x|_i^{\downarrow}(1-t)\big\}.
  \]
  By invoking \cite[Lemma 1]{LiuBiPan18} with $\omega=|x|_j^{\downarrow}$
  for each $j$, it immediately follows that
  \[
   {\textstyle\sum_{i=1}^n}\phi(w_i)+\overline{\rho}\big(\|x\|_1\!-\langle w,x\rangle\big)
   \ge\|x^{\overline{\rho}}\|_0+\frac{\overline{\rho}(1\!-\!t_0)}{\phi_{-}'(1)(1\!-t^*)}
     \sum_{j\in J_1}\,|x|_j^{\downarrow}+\overline{\rho}(1\!-\!t_0)
     \sum_{j\in J_2}\,|x|_j^{\downarrow}.
  \]
  Notice that $1=\phi(1)=\phi(1)-\phi(t^*)\le\phi_{-}'(1)(1-t^*)$. From the last inequality, we have
  \begin{align}\label{temp-wxineq0}
   {\textstyle\sum_{i=1}^n}\phi(w_i)+\overline{\rho}\big(\|x\|_1\!-\langle w,x\rangle\big)
   \ge\|x^{\overline{\rho}}\|_0+\frac{\overline{\rho}(1-t_0)}{\phi_{-}'(1)(1\!-t^*)}
     \sum_{j\in J_1\cup J_2}|x|_j^{\downarrow}\nonumber\\
   =\|x^{\overline{\rho}}\|_0+\frac{\overline{\rho}(1-t_0)}{\phi_{-}'(1)(1\!-t^*)}
     \sum_{j=r+1}^n|x|_j^{\downarrow}
  \ge \|x^{\overline{\rho}}\|_0+\alpha_{\!f}\lambda^{-1}\|x-x^{\overline{\rho}}\|\qquad
  \end{align}
  where the last inequality is due to \eqref{err-ineq1} and the definition of $\overline{\rho}$.
  Since $x\in\mathcal{S}$ and $x^{\overline{\rho}}\in\mathcal{S}$, we have
  $f_{\sigma,\gamma}(x^{\overline{\rho}})-f_{\sigma,\gamma}(x)\le \alpha_{\!f}\|x-x^{\overline{\rho}}\|$.
  Together with the last inequality,
  \begin{equation}\label{temp-wxineq}
    {\textstyle\sum_{i=1}^n}\phi(w_i)+\overline{\rho}\big(\|x\|_1\!-\langle w,x\rangle\big)
    \ge \|x^{\overline{\rho}}\|_0+\lambda^{-1}\big[f_{\sigma,\gamma}(x^{\overline{\rho}})-f_{\sigma,\gamma}(x)\big].
  \end{equation}
  Now take $w_{i}^{\overline{\rho}}=1$ for $i\in{\rm supp}(x^{\overline{\rho}})$
  and $w_{i}^{\overline{\rho}}=0$ for $i\notin{\rm supp}(x^{\overline{\rho}})$.
  Clearly, $(x^{\overline{\rho}},w^{\overline{\rho}})$ is a feasible point
  of the MPEC \eqref{Eznorm-MPEC} with $\sum_{i=1}^n\phi(w_i^{\overline{\rho}})
  =\|x^{\overline{\rho}}\|_0$. Then, it holds that
  \[
   f_{\sigma,\gamma}(x^{\overline{\rho}})+\lambda\|x^{\overline{\rho}}\|_0
  \ge f_{\sigma,\gamma}(x^*) +\lambda{\textstyle\sum_{i=1}^n}\phi(w_i^*).
  \]
  Together with \eqref{temp-wxineq}, we obtain the inequality \eqref{aim-ineq}.
  Notice that $\langle e\!-\!w^*,|x^*|\rangle=0$. The inequality \eqref{aim-ineq}
  implies that every global optimal solution of \eqref{Eznorm-MPEC} is globally
  optimal to the problem \eqref{Penalty-MPEC} associated to every $\rho>\overline{\rho}$.
  Conversely, by fixing any $\rho>\overline{\rho}$ and letting $(\overline{x}^{\rho},\overline{w}^{\rho})$
  be a global optimal solution of the problem \eqref{Penalty-MPEC} associated to $\rho$, it holds that
  \begin{align*}
   &f_{\sigma,\gamma}(\overline{x}^{\rho})+\lambda\textstyle{\sum_{i=1}^n}\phi(\overline{w}_i^{\rho})
   +\rho\lambda\langle e-\overline{w}^{\rho},|\overline{x}^{\rho}|\rangle \\
   &\le f_{\sigma,\gamma}(x^*)+\lambda\textstyle{\sum_{i=1}^n}\phi(w_i^*)
   =f_{\sigma,\gamma}(x^*)+\lambda\textstyle{\sum_{i=1}^n}\phi(w_i^*)
    +\frac{\rho+\overline{\rho}}{2}\lambda\langle e-\overline{w}^{\rho},|\overline{x}^{\rho}|\rangle\\
   &\le f_{\sigma,\gamma}(\overline{x}^{\rho})+\lambda\textstyle{\sum_{i=1}^n}\phi(\overline{w}_i^{\rho})
      +\frac{\rho+\overline{\rho}}{2}\lambda\langle e-\overline{w}^{\rho},|\overline{x}^{\rho}|\rangle,
  \end{align*}
  which implies that $\frac{\rho-\overline{\rho}}{2}\lambda\langle e-\overline{w}^{\rho},|\overline{x}^{\rho}|\rangle\le 0$.
  Since $\rho>\overline{\rho}$ and $\langle e-\overline{w}^{\rho},|\overline{x}^{\rho}|\rangle\ge 0$,
  we obtain $\langle e-\overline{w}^{\rho},|\overline{x}^{\rho}|\rangle=0$.
  Together with the last inequality, it follows that $(\overline{x}^{\rho},\overline{w}^{\rho})$
  is a global optimal solution of \eqref{Eznorm-MPEC}.
  The second part then follows.
 \end{proof}

  By the definition of $\psi$, the penalty problem \eqref{Penalty-MPEC} can be rewritten in a compact form
  \begin{equation*}
  \min_{x\in\mathcal{S},w\in\mathbb{R}^n}\big\{f_{\sigma,\gamma}(x)+\lambda\textstyle{\sum_{i=1}^n}\psi(w_i)
  +\rho\lambda\langle e-w,|x|\rangle\big\},
  \end{equation*}
  which, by the definition of the conjugate function $\psi^*$, can be simplified to be \eqref{Esurrogate}.
  Then, Proposition \ref{partial-calm} implies that the problem \eqref{Esurrogate} associated
  to every $\phi\in\!\mathscr{L}$ and $\rho>\overline{\rho}$ is an equivalent surrogate
  of the problem \eqref{znorm-Moreau}. For a specific $\phi$, since $t^*,t_0$ and $\phi_{-}'(1)$
  are known, the threshold $\overline{\rho}$ is also known by Lemma \ref{calmness} though $\kappa=3\sqrt{n}$
  is a rough estimate.

  When $\phi$ is the one in Section \ref{sec1.2}, it is easy to verify that
  $\lambda\rho\varphi_{\rho}$ with $\lambda=\frac{(a+1)\nu^2}{2}$ and $\rho=\frac{2}{(a+1)\nu}$
  is exactly the SCAD function $x\mapsto\sum_{i=1}^n\rho_{\nu}(x_i)$ proposed in \cite{Fan01}.
  Since $t^*=0,t_0=1/2$ and $\phi_{-}'(1)=\frac{2a}{a+1}$ for this $\phi$,
  the SCAD function with $\nu<\frac{2}{(a+1)\overline{\rho}}$ is an equivalent
  surrogate of \eqref{znorm-Moreau}.
  When $\phi(t)=\frac{a^2}{4}t^2-\frac{a^2}{2}t+at+\frac{(a-2)^2}{4}\ (a>2)$ for $t\in\mathbb{R}$,
  \[
    \psi^*(\omega)=\left\{\begin{array}{cl}
                      -\frac{(a-2)^2}{4} & \textrm{if}\ \omega\leq  a-a^2/2,\\
                      \frac{1}{a^2}(\frac{a(a-2)}{2}+\omega)^2-\frac{(a-2)^2}{4}& \textrm{if}\  a-a^2/2 <\omega\leq a,\\
                      \omega-1 & \textrm{if}\ \omega>a.
                \end{array}\right.
  \]
  It is not hard to verify that the function $\lambda\rho\varphi_{\rho}$
  with $\lambda={a\nu^2}/{2}$ and $\rho={1}/{\nu}$ is exactly the one
  $x\mapsto\sum_{i=1}^ng_{\nu,b}(x_i)$ with $b=a$ used in \cite[Section 3.3]{HuangY18}.
  Since $t^*=\frac{a-2}{a},t_0=\frac{a-1}{a}$ and $\phi_{-}'(1)=a$ for this $\phi$,
  the MCP function used in \cite{HuangY18} with $\nu<1/{\overline{\rho}}$ and $b=a$
  is also an equivalent surrogate of the problem \eqref{znorm-Moreau}.
%------------------------------------------------------------------------------------
 \section{PG method with extrapolation}\label{sec4}
%------------------------------------------------------------------------------
 \subsection{PG with extrapolation for solving \eqref{znorm-Moreau}}\label{sec4.1}

  Recall that $f_{\sigma,\gamma}$ is a smooth function whose gradient $\nabla\!f_{\sigma,\gamma}$
  is Lipschitz continuous with modulus $L_{\!f}\le\gamma^{-1}\|A\|^2$.
  While by Proposition \ref{proxm-glam} the proximal mapping of $g_{\lambda}$ has a closed form.
  This inspires us to apply the PG method with extrapolation to solving \eqref{znorm-Moreau}.
 %----------------------------------------------------------------------------------------------
 \begin{algorithm}[H]
  \caption{\label{PGMe1}{\bf (PGe-znorm for solving the problem \eqref{znorm-Moreau})}}
  \textbf{Initialization:} Choose $\varsigma\in(0,1),0<\tau<(1\!-\!\varsigma)L_{\!f}^{-1},0<\beta_{\rm max}\le\frac{\sqrt{\varsigma(\tau^{-1}-L_{\!f})\tau^{-1}}}{2(\tau^{-1}+L_{\!f})}$ and
  an initial point $x^0\in\mathcal{S}$. Set $x^{-1}=x^{0}$ and $k:=0$.

  \medskip
  \noindent
 \textbf{while} the termination condition is not satisfied \textbf{do}
  \begin{itemize}
   \item[{\bf1.}] Let $\widetilde{x}^k=x^{k}+\beta_k(x^k-x^{k-1})$. Compute
                  $x^{k+1}\in\mathcal{P}_{\!\tau}g_{\lambda}(\widetilde{x}^k\!-\!\tau\nabla\!f_{\sigma,\gamma}(\widetilde{x}^k))$.

   \item[{\bf3.}] Choose $\beta_{k+1}\in[0,\beta_{\rm max}]$. Let $k\leftarrow k+1$ and go to Step 1.
  \end{itemize}
  \textbf{end (while)}
  \end{algorithm}
  \begin{remark}\label{remark1-PGMe1}
   The main computation work of Algorithm \ref{PGMe1} in each iteration is to seek
   \begin{equation}\label{xk-def}
    x^{k+1}\in\mathop{\arg\min}_{x\in\mathbb{R}^n}\Big\{\langle\nabla\!f_{\sigma,\gamma}(\widetilde{x}^k),x-\widetilde{x}^k\rangle
    +\frac{1}{2\tau}\|x-\widetilde{x}^k\|^2+g_{\lambda}(x)\Big\}.
   \end{equation}
   By Proposition \ref{proxm-glam}, to achieve a global optimal solution of
   the nonconvex problem \eqref{xk-def} requires about $2mn$ flops.
   Owing to the good performance of the Nesterov's acceleration strategy \cite{Nesterov83,Nesterov04},
   one can use this strategy to choose the extrapolation parameter $\beta_k$, i.e.,
   \begin{equation}\label{accelerate}
    \beta_k=\frac{t_{k-1}-1}{t_k}\ \ {\rm with}\ t_{k+1}=\frac{1}{2}\big(1+\!\sqrt{1+4t_k^2}\big)
    \ \ {\rm for}\ t_{-1}=t_{0}=1.
   \end{equation}
   In Algorithm \ref{PGMe1}, an upper bound $\beta_{\rm max}$ is imposed on $\beta_{k}$ just for
   the convergence analysis. It is easy to check that as $\varsigma$ approaches to $1$,
   say $\varsigma=0.999$, $\beta_{\rm max}$ can take $0.235$.
  \end{remark}

  The PG method with extrapolation, first proposed in \cite{Nesterov83} and
  extended to a general composite setting in \cite{Beck09}, is a popular
  first-order one for solving nonconvex nonsmooth composite optimization problems
  such as \eqref{znorm-Moreau} and \eqref{Esurrogate}. In the past several years,
  the PGe and its variants have received extensive attentions (see, e.g.,
  \cite{Ghadimi16,Li15,Wen17,Ochs14,Ochs19,Xu17,Yang21}). Due to the nonconvexity of
  the sphere constraint and the zero-norm, the results obtained in \cite{Ghadimi16,Wen17,Ochs14}
  are not applicable to \eqref{znorm-Moreau}. Although Algorithm \ref{PGMe1} is
  a special case of those studied in \cite{Li15,Xu17,Yang21}, the convergence results
  of \cite{Li15,Yang21} are obtained for the objective value sequence and the convergence
  result of \cite{Xu17} on the iterate sequence requires a strong restriction on $\beta_k$,
  i.e., it is such that the objective value sequence is nonincreasing.

  Next we provide the proof for the convergence and local convergence rate of
  the iterate sequence yielded by Algorithm \ref{PGMe1}.
  For any $\tau>0$ and $\varsigma\in(0,1)$, we define the function
  \begin{equation}\label{Htau-vsig}
    H_{\tau,\varsigma}(x,u):=F_{\sigma,\gamma}(x)+\frac{\varsigma}{4\tau}\|x-u\|^2
    \quad\ \forall (x,u)\in\mathbb{R}^n\times\mathbb{R}^n.
  \end{equation}
  The following lemma summarizes the properties of $H_{\tau,\zeta}$ on the sequence
  $\{x^k\}_{k\in\mathbb{N}}$.
%--------------------------------------------------------------------------------------------
 \begin{lemma}\label{lemma1-PGMe1}
  Let $\{x^k\}_{k\in\mathbb{N}}$ be the sequence generated by Algorithm \ref{PGMe1}. Then,
  \begin{itemize}
  \item[(i)] for each $k\in\mathbb{N}$,
              \(
                H_{\tau,\varsigma}(x^{k+1},x^k)\le H_{\tau,\varsigma}(x^{k},x^{k-1})
                  -\frac{\varsigma(\tau^{-1}-L_{\!f})}{2}\|x^{k+1}\!-\!x^k\|^2.
              \)

   \item[(ii)] The sequence $\{H_{\tau,\varsigma}(x^{k},x^{k-1})\}_{k\in\mathbb{N}}$ is convergent
                and $\sum_{k=1}^{\infty}\|x^{k+1}\!-\!x^k\|^2<\infty$.

   \item[(iii)] For each $k\in\mathbb{N}$, there exists $w^{k}\in\partial H_{\tau,\varsigma}(x^{k},x^{k-1})$
                with $\|w^{k+1}\|\le b_1\|x^{k+1}\!-\!x^k\|+b_2\|x^k\!-\!x^{k-1}\|$, where
                $b_1>0$ and $b_2>0$ are the constants independent of $k$.
  \end{itemize}
 \end{lemma}
 \begin{proof}
  {\bf(i)} Since $\nabla\!f_{\sigma,\gamma}$ is globally Lipschitz continuous,
  from the descent lemma we have
  \begin{equation}\label{fdecent}
    f_{\sigma,\gamma}(x')\le f_{\sigma,\gamma}(x)+\langle\nabla\!f_{\sigma,\gamma}(x),x'-x\rangle
    +({L_{\!f}}/{2})\|x'-x\|^2\quad\forall x',x\in\mathbb{R}^n.
  \end{equation}
  From the definition of $x^{k+1}$ or the equation \eqref{xk-def},
  for each $k\in\mathbb{N}$ it holds that
  \[
    \langle\nabla\!f_{\sigma,\gamma}(\widetilde{x}^k),x^{k+1}-x^k\rangle
    +\frac{1}{2\tau}\|x^{k+1}-\widetilde{x}^k\|^2+g_{\lambda}(x^{k+1})
    \le \frac{1}{2\tau}\|x^{k}-\widetilde{x}^k\|^2+g_{\lambda}(x^{k}).
  \]
  Together with the inequality \eqref{fdecent} for $x'=x^{k+1}$ and $x=x^k$,
  it follows that
  \begin{align*}
    f_{\sigma,\gamma}(x^{k+1})+g_{\lambda}(x^{k+1})
    &\le f_{\sigma,\gamma}(x^{k})+g_{\lambda}(x^{k})
       -\frac{1}{2\tau}\|x^{k+1}-\widetilde{x}^k\|^2
       +\frac{L_{\!f}}{2}\|x^{k+1}-x^k\|^2 \nonumber\\
    &\quad+\langle\nabla\!f_{\sigma,\gamma}(x^k)-\nabla\!f_{\sigma,\gamma}(\widetilde{x}^k),x^{k+1}-x^k\rangle
    +\frac{1}{2\tau}\|x^{k}-\widetilde{x}^k\|^2 \nonumber\\
    &=f_{\sigma,\gamma}(x^{k})+g_{\lambda}(x^{k}) -\frac{1}{2}(\tau^{-1}\!-\!L_{\!f})\|x^{k+1}-x^k\|^2\\
    &\quad -\frac{1}{\tau}\langle x^{k+1}\!-\!x^k,x^{k}\!-\!\widetilde{x}^k\rangle +\langle\nabla\!f_{\sigma,\gamma}(x^k)\!-\!\nabla\!f_{\sigma,\gamma}(\widetilde{x}^k),x^{k+1}\!-\!x^k\rangle.\nonumber
  \end{align*}
  Using $\widetilde{x}^k=x^{k}+\beta_k(x^k-x^{k-1})$ and the Lipschitz continuity
  of $\nabla\!f_{\sigma,\gamma}$ yields that
  \begin{align*}
   F_{\sigma,\gamma}(x^{k+1})
    %&\le F_{\sigma,\gamma}(x^k) -\frac{\tau^{-1}\!-\!L_{\!f}}{2}\|x^{k+1}\!-\!x^k\|^2
%        +(\tau^{-1}\!+\!L_{\!f})\|x^{k+1}\!-\!x^k\|\|x^{k}\!-\!\widetilde{x}^{k}\|\\
    &\le F_{\sigma,\gamma}(x^k) -\frac{\tau^{-1}\!-\!L_{\!f}}{2}\|x^{k+1}\!-\!x^k\|^2
        +(\tau^{-1}\!+\!L_{\!f})\beta_k\|x^{k+1}\!-\!x^k\|\|x^{k}\!-\!x^{k-1}\|\\
    &\le F_{\sigma,\gamma}(x^k) -\frac{\tau^{-1}\!-\!L_{\!f}}{4}\|x^{k+1}\!-\!x^k\|^2
        +\frac{(\tau^{-1}\!+\!L_{\!f})^2}{\tau^{-1}\!-\!L_{\!f}}\beta_k^2\|x^{k}\!-\!x^{k-1}\|^2\\
    &\le F_{\sigma,\gamma}(x^k) -\frac{\tau^{-1}\!-\!L_{\!f}}{4}\|x^{k+1}\!-\!x^k\|^2
        +\frac{\varsigma}{4\tau}\|x^{k}\!-\!x^{k-1}\|^2\\
    &=F_{\sigma,\gamma}(x^k) -\frac{(1\!-\!\varsigma)\tau^{-1}\!-\!L_{\!f}}{4}\|x^{k+1}\!-\!x^k\|^2
        -\frac{\varsigma}{4\tau}\|x^{k+1}\!-\!x^k\|^2+\frac{\varsigma}{4\tau}\|x^{k}\!-\!x^{k-1}\|^2
  \end{align*}
  where the second is due to $ab\le\frac{a^2}{4s}+b^2$ with $a=\|x^{k+1}\!-\!x^k\|$,
  $b=(\tau^{-1}\!+\!L_{\!f})\beta_k\|x^{k}\!-\!x^{k-1}\|$ and $s=\frac{1}{\tau^{-1}-L_{\!f}}>0$,
  and the last is due to $\beta_k\le\beta_{\rm max}\le\frac{\sqrt{\varsigma(\tau^{-1}-L_{\!f})\tau^{-1}}}{2(\tau^{-1}+L_{\!f})}$.
  Combining the last inequality with the definition of $H_{\tau,\varsigma}$,
  we obtain the result.

  \noindent
  {\bf(ii)} Note that $H_{\tau,\varsigma}$ is lower bounded by the lower boundedness of
  the function $F_{\sigma,\gamma}$. The nonincreasing of the sequence
  $\{H_{\tau,\varsigma}(x^{k},x^{k-1})\}_{k\in\mathbb{N}}$ in part (i) implies
  its convergence, and consequently, $\sum_{k=1}^{\infty}\|x^{k+1}\!-\!x^k\|^2<\infty$
  follows by using part (i) again.

  \noindent
  {\bf(iii)} From the definition of $H_{\tau,\varsigma}$ and \cite[Exercise 8.8]{RW98},
  for any $(x,u)\in\mathcal{S}\times\mathbb{R}^n$,
  \begin{equation}\label{subdiff1}
    \partial H_{\tau,\varsigma}(x,u)
    =\left(\begin{matrix}
     \partial F_{\sigma,\gamma}(x)+\frac{1}{2}\tau^{-1}\varsigma(x-u)\\
     \frac{1}{2}\tau^{-1}\varsigma(u-x)
    \end{matrix}\right).
  \end{equation}
  Fix any $k \in \mathbb{N}$. By the optimality of $x^{k+1}$ to the nonconvex problem \eqref{xk-def},
  it follows that
  \[
    0\in \nabla\!f_{\sigma,\gamma}(\widetilde{x}^k)+\tau^{-1}(x^{k+1}\!-\widetilde{x}^{k})+\partial h_{\lambda}(x^{k+1}),
  \]
  which is equivalent to
  \(
   \nabla\!f_{\sigma,\gamma}(x^{k+1})\!-\!\nabla\!f_{\sigma,\gamma}(\widetilde{x}^k)
   -\tau^{-1}(x^{k+1}\!-\!\widetilde{x}^{k})\in \partial F_{\sigma,\gamma}(x^{k+1}).
  \)
  Write
  \[
   w^{k}:=\left(\begin{matrix}
            \nabla\!f_{\sigma,\gamma}(x^{k})-\nabla\!f_{\sigma,\gamma}(\widetilde{x}^{k-1})
             -\tau^{-1}(x^{k}-\widetilde{x}^{k-1})+\frac{1}{2}\tau^{-1}\varsigma(x^{k}-x^{k-1})\\
             \frac{1}{2}\tau^{-1}\varsigma(x^{k-1}-x^{k})
           \end{matrix}\right).
  \]
 By comparing with \eqref{subdiff1}, we have $w^{k}\in\partial H_{\tau,\varsigma}(x^{k},x^{k-1})$.
 From the Lipschitz continuity of $\nabla\!f_{\sigma,\gamma}$ and Step 1,
 $\|w^{k+1}\|\le (\tau^{-1}\!+\!L_{\!f}\!+\!\tau^{-1}\varsigma)\|x^{k+1}\!-\!x^k\|
 +(\tau^{-1}\!+\!L_{\!f})\beta_{\rm max}\|x^{k}\!-\!x^{k-1}\|$.
 Since $\beta_{\rm max}\in(0,1)$, the result holds
 with $b_1=\tau^{-1}\!+\!L_{\!f}\!+\!\tau^{-1}\varsigma$ and $b_2=\tau^{-1}\!+\!L_{\!f}$.
 \end{proof}
%--------------------------------------------------------------------------------------
 \begin{lemma}\label{lemma2-PGMe1}
  Let $\{x^k\}_{k\in\mathbb{N}}$ be the sequence generated by Algorithm \ref{PGMe1}
  and denote by $\varpi(x^0)$ the set of accumulation points of $\{x^k\}_{k\in\mathbb{N}}$.
  Then, the following assertions hold:
  \begin{itemize}
    \item [(i)] $\varpi(x^0)$ is a nonempty compact set and $\varpi(x^0)\subseteq
                 S_{\tau,g_{\lambda}}\subseteq{\rm crit}F_{\sigma,\gamma}$;

    \item [(ii)] $\lim_{k\to\infty}{\rm dist}((x^k,x^{k-1}),\Omega)=0$ with
                   $\Omega:=\{(x,x)\,|\,x\in\varpi(x^0)\}\subseteq{\rm crit}H_{\tau,\varsigma}$;

    \item[(iii)] the function $H_{\tau,\varsigma}$ is finite and keeps the constant on the set $\Omega$.
  \end{itemize}
 \end{lemma}
 \begin{proof}
 {\bf (i)} Since $\{x^k\}_{k\in\mathbb{N}}\subseteq \mathcal{S}$, we have $\varpi(x^0)\ne\emptyset$.
 Since $\varpi(x^0)$ can be viewed as an intersection of compact sets, i.e.,
 $\varpi(x^0)=\bigcap_{q\in\mathbb{N}}\overline{\bigcup_{k\ge q}\{x^k\}}$, it is also compact.
 Now pick any $x^*\in\varpi(x^0)$. There exists a subsequence $\{x^{k_j}\}_{j\in\mathbb{N}}$
 with $x^{k_j}\rightarrow x^*$ as $j\rightarrow \infty$. Note that
 $\lim_{j\to\infty}\|x^{k_j}-x^{k_j-1}\|=0$ implied by Lemma \ref{lemma1-PGMe1} (ii).
 Then, $x^{k_j-1}\rightarrow x^*$ and $x^{k_j+1}\rightarrow x^*$ as $j\rightarrow \infty$.
 Recall that $\widetilde{x}^{k_j}=x^{k_j}+\beta_{k_j}(x^{k_j}-x^{k_j-1})$
 and $\beta_{k_j}\in[0,\beta_{\rm max})$. When $j\rightarrow \infty$,
 we have $\widetilde{x}^{k_j}\rightarrow x^*$ and then
 $\widetilde{x}^{k_j}\!-\!\tau\nabla\!f_{\sigma,\gamma}(\widetilde{x}^{k_j})
 \rightarrow x^*\!-\!\tau\nabla\!f_{\sigma,\gamma}(x^*)$.
 In addition, since $g_{\lambda}$ is proximally bounded with threshold $+\infty$,
 i.e., for any $\tau'>0$ and $x\in\mathbb{R}^n$,
 $\min_{z\in\mathbb{R}^n}\big\{\frac{1}{2\tau'}\|z-x\|^2+g_{\lambda}(z)\big\}>-\infty$,
 from \cite[Example 5.23]{RW98} it follows that $\mathcal{P}_{\!\tau}g_{\lambda}$
 is outer semicontinuous. Thus, from $x^{k_j+1}\in\mathcal{P}_{\!\tau}g_{\lambda}(\widetilde{x}^{k_j}\!-\!\tau\nabla\!f_{\sigma,\gamma}(\widetilde{x}^{k_j}))$
 for each $j\in\mathbb{N}$, we have $x^*\in\mathcal{P}_{\!\tau}g_{\lambda}(x^*\!-\!\tau\nabla\!f_{\sigma,\gamma}(x^*))$,
 and then $x^*\in S_{\tau,g_{\lambda}}$. By the arbitrariness of $x^*\in\varpi(x^0)$,
 the first inclusion follows. The second inclusion is given by Lemma \ref{relation-critical}.

 \noindent
 {\bf (ii)-(iii)} The result of part (ii) is immediate, so it suffices to prove part (iii).
 By Lemma \ref{lemma1-PGMe1} (i), the sequence $\{H_{\tau,\varsigma}(x^k,x^{k-1})\}_{k\in\mathbb{N}}$
 is convergent and denote its limit by $\omega^*$. Pick any $(x^*,x^*)\in\Omega$.
 There exists a subsequence $\{x^{k_j}\}_{j\in\mathbb{N}}$ with $x^{k_j}\rightarrow x^*$
 as $j\rightarrow \infty$. If $\lim_{j \rightarrow \infty}H_{\tau,\varsigma}(x^{k_j},x^{k_j-1})
 =H_{\tau,\varsigma}(x^*,x^*)$, then the convergence of $\{H_{\tau,\varsigma}(x^k,x^{k-1})\}_{k\in\mathbb{N}}$
 implies that $H_{\tau,\varsigma}(x^*,x^*)=\omega^*$, which by the arbitrariness of
 $(x^*,x^*)\in\Omega$ shows that the function $H_{\tau,\varsigma}$ is finite and
 keeps the constant on $\Omega$. Hence, it suffices to argue that
 $\lim_{j \rightarrow \infty}H_{\tau,\varsigma}(x^{k_j},x^{k_j-1})=H_{\tau,\varsigma}(x^*,x^*)$.
 Recall that $\lim_{j\to\infty}\|x^{k_j}-x^{k_j-1}\|=0$ by Lemma \ref{lemma1-PGMe1} (ii).
 We only need argue that $\lim_{j \rightarrow \infty}F_{\sigma,\gamma}(x^{k_j})=F_{\sigma,\gamma}(x^*)$.
 From \eqref{xk-def}, it holds that
 \[
    \langle\nabla\!f_{\sigma,\gamma}(\widetilde{x}^{k_j-1}),x^{k_j}-x^*\rangle
    +\frac{1}{2\tau}\|x^{k_j}-\widetilde{x}^{k_j-1}\|^2+g_{\lambda}(x^{k_j})
    \le \frac{1}{2\tau}\|x^*-\widetilde{x}^{k_j-1}\|^2+g_{\lambda}(x^*).
  \]
 Together with the inequality \eqref{fdecent} with $x'=x^{k_j}$ and $x=x^*$, we obtain that
 \begin{align*}
  F_{\sigma,\gamma}(x^{k_j})
  &\le F_{\sigma,\gamma}(x^*)+\langle\nabla\!f_{\sigma,\gamma}(x^*)-\nabla\!f_{\sigma,\gamma}(\widetilde{x}^{k_j-1}),x^{k_j}-x^*\rangle
    -\frac{1}{2\tau}\|x^{k_j}-\widetilde{x}^{k_j-1}\|^2\\
  &\quad +\frac{1}{2\tau}\|x^*-\widetilde{x}^{k_j-1}\|^2+\frac{L_{\!f}}{2}\|x^{k_j}-x^*\|^2,
 \end{align*}
 which by $\lim_{j\rightarrow \infty}x^{k_j}=x^*=\lim_{j\rightarrow \infty}x^{k_j-1}$ implies that
 $\limsup_{j\rightarrow\infty}F_{\sigma,\gamma}(x^{k_j})\le F_{\sigma,\gamma}(x^*)$.
 In addition, by the lower semicontinuity of $F_{\sigma,\gamma}$,
 $\liminf_{j\rightarrow\infty}F_{\sigma,\gamma}(x^{k_j})\ge F_{\sigma,\gamma}(x^*)$.
 The two sides show that $\lim_{j \rightarrow \infty}F_{\sigma,\gamma}(x^{k_j})=F_{\sigma,\gamma}(x^*)$.
 The proof is then completed.
 \end{proof}

 Since $f_{\sigma,\gamma}$ is a piecewise linear-quadratic function, it is semi-algebraic.
 Recall that the zero-norm is semi-algebraic. Hence, $F_{\sigma,\gamma}$ and $H_{\tau,\varsigma}$
 are also semi-algebraic, and then the KL functions. By using Lemma \ref{lemma1-PGMe1}-\ref{lemma2-PGMe1}
 and following the arguments as those for \cite[Theorem 1]{Bolte14} and \cite[Theorem 2]{Attouch09}
 we can establish the following convergence results.
%---------------------------------------------------------------------------------
 \begin{theorem}\label{global-conv}
  Let $\{x^k\}_{k\in\mathbb{N}}$ be the sequence generated by Algorithm \ref{PGMe1}.
  Then,
 \begin{itemize}
 \item[(i)] $\sum_{k=1}^{\infty}\|x^{k+1}-x^k\|<\infty$ and consequently $\{x^k\}_{k\in \mathbb{N}}$
             converges to some $x^*\in S_{\tau,g_{\lambda}}$.

 \item[(ii)] If $F_{\sigma,\gamma}$ is a KL function of exponent $1/2$,
             then there exist $c_1>0$ and $\varrho\in(0,1)$ such that
             for all sufficiently large $k$, $\|x^k-x^*\|\le c_1\varrho^k$.
 \end{itemize}
 \end{theorem}
 \begin{proof}
 {\bf(i)} For each $k\in\mathbb{N}$, write $z^k\!:=(x^k,x^{k-1})$.
 Since $\{x^k\}_{k\in\mathbb{N}}$ is bounded, there exists a subsequence $\{x^{k_q}\}_{q\in\mathbb{N}}$
 with $x^{k_q}\to\overline{x}$ as $q\to\infty$. By the proof of Lemma \ref{lemma2-PGMe1} (iii),
 $\lim_{k\to\infty}H_{\tau,\varsigma}(z^k)=H_{\tau,\varsigma}(\overline{z})$ with
 $\overline{z}=(\overline{x},\overline{x})$. If there exists $\overline{k}\in\mathbb{N}$
 such that $H_{\tau,\varsigma}(z^{\overline{k}})=H_{\tau,\varsigma}(\overline{z})$,
 by Lemma \ref{lemma1-PGMe1} (i) we have $x^{k}=x^{\overline{k}}$ for all $k\ge\overline{k}$
 and the result follows. Thus, it suffices to consider that
 $H_{\tau,\varsigma}(z^{k})>H_{\tau,\varsigma}(\overline{z})$ for all $k\in\mathbb{N}$.
 Since $\lim_{k\to\infty}H_{\tau,\varsigma}(z^k)=H_{\tau,\varsigma}(\overline{z})$,
 for any $\eta>0$ there exists $k_0\in\mathbb{N}$ such that for all $k\ge k_0$,
 $H_{\tau,\varsigma}(z^k)< H_{\tau,\varsigma}(\overline{z})+\eta$.
 In addition, from Lemma \ref{lemma2-PGMe1} (ii),
 for any $\varepsilon>0$ there exists $k_1\in\mathbb{N}$ such that for all $k\ge k_1$,
 ${\rm dist}(z^k,\Omega)<\varepsilon$. Then, for all $k\ge\overline{k}:=\max(k_0,k_1)$,
 \[
   z^k\in\big\{z\,|\,{\rm dist}(z,\Omega)<\varepsilon\big\}
   \cap[H_{\tau,\varsigma}(\overline{z})<H_{\tau,\varsigma}<H_{\tau,\varsigma}(\overline{z})+\eta].
 \]
 By combining Lemma \ref{lemma2-PGMe1} (iii) and \cite[Lemma 6]{Bolte14},
 there exist $\delta>0$, $\eta>0$ and a continuous concave function
 $\varphi\!:[0,\eta)\to\mathbb{R}_{+}$ satisfying the conditions in Definition \ref{KL-def}
 such that for all $\overline{z}\in\Omega$ and
 all $z\in\big\{z\,|\,{\rm dist}(z,\Omega)<\varepsilon\big\}
 \cap[H_{\tau,\varsigma}(\overline{z})<H_{\tau,\varsigma}<H_{\tau,\varsigma}(\overline{z})+\eta]$,
 \[
   \varphi'(H_{\tau,\varsigma}(z)-H_{\tau,\varsigma}(\overline{z}))
    {\rm dist}(0,\partial H_{\tau,\varsigma}(z))\ge 1.
 \]
 Consequently, for all $k>\overline{k}$,
 $\varphi'(H_{\tau,\varsigma}(z^k)-H_{\tau,\varsigma}(\overline{z})){\rm dist}(0,\partial H_{\tau,\varsigma}(z^k))\ge 1$.
 By Lemma \ref{lemma1-PGMe1} (iii), there exists $w^{k}\in\partial H_{\tau,\varsigma}(z^k)$
 with $\|w^k\|\le b_1\|x^{k}-x^{k-1}\|+b_2\|x^{k-1}-x^{k-2}\|$. Then,
 \[
   \varphi'(H_{\tau,\varsigma}(z^k)-H_{\tau,\varsigma}(\overline{z}))\|w^k\|\ge 1.
 \]
 Together with the concavity of $\varphi$ and Lemma \ref{lemma1-PGMe1} (i), it follows that
 for all $k>\overline{k}$,
 \begin{align*}
  & [\varphi(H_{\tau,\varsigma}(z^k)-H_{\tau,\varsigma}(\overline{z}))-
   \varphi(H_{\tau,\varsigma}(z^{k+1})-H_{\tau,\varsigma}(\overline{z}))]\|w_k\|\\
  &\ge \varphi'(H_{\tau,\varsigma}(z^k)-H_{\tau,\varsigma}(\overline{z}))
       [H_{\tau,\varsigma}(z^k)-H_{\tau,\varsigma}(z^{k+1})]\|w_k\|\\
  &\ge H_{\tau,\varsigma}(z^k)-H_{\tau,\varsigma}(z^{k+1})\ge a\|x^{k+1}-x^k\|^2
  \ \ {\rm with}\ a={\varsigma(\tau^{-1}\!-\!L_{\!f})}/{2}.
 \end{align*}
 For each $k\in\mathbb{N}$, let $\Delta_k:=\varphi(H_{\tau,\varsigma}(z^k)-H_{\tau,\varsigma}(\overline{z}))-
 \varphi(H_{\tau,\varsigma}(z^{k+1})-H_{\tau,\varsigma}(\overline{z}))$. For all $k>\overline{k}$,
 \begin{align*}
  2\|x^{k+1}-x^k\|&\le 2\sqrt{a^{-1}\Delta_k\|w_k\|}
  \le 2\sqrt{a^{-1}\Delta_k[b_1\|x^{k}\!-\!x^{k-1}\|+b_2\|x^{k-1}\!-\!x^{k-2}\|]}\\
  &\le \frac{1}{2}\big(\|x^{k}-x^{k-1}\|+\|x^{k-1}-x^{k-2}\|\big)+2a^{-1}\max(b_1,b_2)\Delta_k,
 \end{align*}
 where the second inequality is due to $2\sqrt{st}\le s/2+2t$ for any $s,t\ge 0$.
 For any $\nu>k>\overline{k}$, summing the last inequality from $k$ to $\nu$ yields that
 \begin{equation}\label{use-ineq2}
  \sum_{j=k}^{\nu}\|x^{j+1}\!-\!x^j\|\le \|x^k-x^{k-1}\|+\frac{1}{2}\|x^{k-1}\!-\!x^{k-2}\|
   +\frac{2\max(b_1,b_2)}{a}\varphi(H_{\tau,\varsigma}(z^{k})-H_{\tau,\varsigma}(\overline{z})).
 \end{equation}
 By passing the limit $\nu\to\infty$ to the last inequality, we obtain the desired result.

 \noindent
 {\bf(ii)} Since $F_{\sigma,\gamma}$ is a KL function of exponent $1/2$,
 by \cite[Theorem 3.6]{LiPong18} and the expression of $H_{\tau,\varsigma}$,
 it follows that $H_{\tau,\varsigma}$ is also a KL function of exponent $1/2$.
 From the arguments for part (i) with $\varphi(t)=c\sqrt{t}$ for $t\ge 0$
 and Lemma \ref{lemma1-PGMe1} (iii), for all $k\ge\overline{k}$ it holds that
 \[
   \sqrt{H_{\tau,\varsigma}(z^k)-H_{\tau,\varsigma}(\overline{z})}
   \le\frac{c}{2}{\rm dist}(0,\partial H_{\tau,\varsigma}(z^k))
   \le\frac{c}{2}[b_1\|x^{k}-x^{k-1}\|+b_2\|x^{k-1}-x^{k-2}\|].
 \]
 Consequently, $\varphi(H_{\tau,\varsigma}(z^k)-H_{\tau,\varsigma}(\overline{z}))
 \le \frac{c^2}{2}[b_1\|x^{k}-x^{k-1}\|+b_2\|x^{k-1}-x^{k-2}\|]$.
 Together with the inequality \eqref{use-ineq2}, by letting $c'=c^2a^{-1}[\max(b_1,b_2)]^2$,
 for any $\nu>k>\overline{k}$ we have
  \[
  {\textstyle\sum_{j=k}^{\nu}}\|x^{j+1}-x^j\|\le (1+c')\|x^k-x^{k-1}\|+(1/2+c')\|x^{k-1}-x^{k-2}\|
  \]
  For each $k\in\mathbb{N}$, let $\Delta_k:=\sum_{j=k}^{\infty}\|x^{j+1}-x^j\|$.
  Passing the limit $\nu\to +\infty$ to this inequality, we obtain
  \(
    \Delta_k\le (1+c')[\Delta_{k-1}-\Delta_{k}]+(1/2+c')[\Delta_{k-2}-\Delta_{k-1}]
    \le (1+c')[\Delta_{k-2}-\Delta_{k}],
  \)
  which means that $\Delta_k\le\varrho\Delta_{k-2}$ for $\varrho=\frac{1+c'}{2+c'}$.
  The result follows by this recursion.
 \end{proof}

  It is worthwhile to point out that by Lemma \ref{lemma1-PGMe1}-\ref{lemma2-PGMe1}
  and the proof of Lemma \ref{lemma2-PGMe1} (iii), applying \cite[Theorem 10]{Ochs19}
  directly can yield $\sum_{k=1}^{\infty}\|x^{k+1}-x^k\|<\infty$.
  Here, we include its proof just for the convergence rate analysis in
  Theorem \ref{global-conv} (ii). Notice that Theorem \ref{global-conv} (ii) requires
  the KL property of exponent $1/2$ of the function $F_{\sigma,\gamma}$. The following
  lemma shows that $F_{\sigma,\gamma}$ indeed has such an important property under a mild condition.
%---------------------------------------------------------------------------------------------
 \begin{lemma}\label{KL-exponent}
  If any $\overline{x}\in\!{\rm crit}F_{\sigma,\gamma}$ has
  $\Gamma(\overline{x})\!=\emptyset$, then $F_{\sigma,\gamma}$ is a KL function of exponent $0$.
 \end{lemma}
 \begin{proof}
  Write $\widetilde{f}_{\sigma,\gamma}(x):=f_{\sigma,\gamma}(x)+\delta_{\mathcal{S}}(x)$
  for $x\in\mathbb{R}^n$. For any $x\in\mathcal{S}$, by \cite[Exercise 8.8]{RW98},
  \begin{equation}\label{subdiff-Theta}
   \partial\!\widetilde{f}_{\sigma,\gamma}(x)=\nabla\!f_{\sigma,\gamma}(x)+\mathcal{N}_{\mathcal{S}}(x)
   =A^{\mathbb{T}}\nabla\!L_{\sigma,\gamma}(Ax)+\mathcal{N}_{\mathcal{S}}(x).
  \end{equation}
  Fix any $\overline{x}\in{\rm crit}F_{\sigma,\gamma}$ with $\Gamma(\overline{x})=\emptyset$.
  Let $J:={\rm supp}(\overline{x}),I_{0}\!:=\{i\in[m]\,|\, (A\overline{x})_i>0\}$,
  $I_{1}\!:=\{i\in[m]\,|\, (A\overline{x})_i<-(\sigma\!+\!\gamma)\}$
  and $I_{2}\!:=\!\{i\in[m]\,|\, -\sigma\!+\!\gamma<(A\overline{x})_i<-\gamma\}$.
  Since $\Gamma(\overline{x})=\emptyset$, we have $I_{0}\cup I_{1}\cup I_{2}=[m]$. Moreover, from the continuity,
  there exists $\varepsilon'>0$ such that for all $x\in\mathbb{B}(\overline{x},\varepsilon')$,
  ${\rm supp}(x)\supseteq J$ and the following inequalities hold:
  \begin{equation}\label{temp-indx}
   (Ax)_i>0\ {\rm for}\ i\in I_{0},\ (Ax)_i<-(\sigma\!+\!\gamma)\ {\rm for}\ i\in I_{1}\ {\rm and}\
   \gamma\!-\!\sigma<(Ax)_i<-\gamma\ {\rm for}\ i\in I_{2}.
  \end{equation}
  By the continuity of $f_{\sigma,\gamma}$,
  there exists $\varepsilon''>0$ such that for all $x\in\mathbb{B}(\overline{x},\varepsilon'')$,
  $f_{\sigma,\gamma}(x)>f_{\sigma,\gamma}(\overline{x})-\lambda/2$.
  Set $\varepsilon=\min(\varepsilon',\varepsilon'')$ and pick any $\eta\in(0,\lambda/4]$.
  Next we argue that
  \[
   \mathbb{B}(\overline{x},\varepsilon)
  \cap[F_{\sigma,\gamma}(\overline{x})<F_{\sigma,\gamma}<F_{\sigma,\gamma}(\overline{x})+\eta]=\emptyset,
  \]
  which by Definition \ref{KL-def} implies that $F_{\sigma,\gamma}$
  is a KL function of exponent $1/2$. Suppose on the contradiction that
  there exists $x\in\mathbb{B}(\overline{x},\varepsilon)
  \cap[F_{\sigma,\gamma}(\overline{x})<F_{\sigma,\gamma}<F_{\sigma,\gamma}(\overline{x})+\eta]$.
  From $F_{\sigma,\gamma}(x)<F_{\sigma,\gamma}(\overline{x})+\eta$, we have $x\in\mathcal{S}$.
  Together with ${\rm supp}(x)\supseteq J$, we deduce that ${\rm supp}(x)=J$ (if not,
  $f_{\sigma,\gamma}(x)+\lambda\|\overline{x}\|_0+\lambda<F_{\sigma,\gamma}(x)
  <f_{\sigma,\gamma}(\overline{x})+\lambda\|\overline{x}\|_0+\eta$,
  which along with $f_{\sigma,\gamma}(x)>f_{\sigma,\gamma}(\overline{x})-\lambda/2$ implies
  $\eta>\lambda/2$, a contradiction to $\eta\le\lambda/4$).
  Now from $x\in\mathcal{S}$, equation \eqref{temp-indx},
  the expression of $\vartheta_{\sigma,\gamma}$ and $I_{0}\cup I_{1}\cup I_{2}=[m]$,
  it follows that
  \begin{equation}\label{Phi-gamma}
   0<F_{\sigma,\gamma}(x)-F_{\sigma,\gamma}(\overline{x})
   =L_{\sigma,\gamma}(Ax)-L_{\sigma,\gamma}(A\overline{x})={\textstyle\sum_{i\in I_{2}}}[(A\overline{x})_i-(Ax)_i].
 \end{equation}
 Recall that $[\nabla\!L_{\sigma,\gamma}(Ax')]_{I_{2}}=-e$ and $[\nabla\!L_{\sigma,\gamma}(Ax')]_{I_{0}\cup I_{1}}=0$
 with $x'=x$ and $\overline{x}$. Hence,
 \begin{align}\label{temp-equa31}
  \|A_{J}^{\mathbb{T}}\nabla\!L_{\sigma,\gamma}(Ax)\|^2
  =\|A_{I_{2}J}^{\mathbb{T}}e\|^2,\
  \langle \nabla\!L_{\sigma,\gamma}(Ax),Ax\rangle^2
  =[{\textstyle\sum_{i\in I_{2}}}(Ax)_i]^2,\\
  \label{temp-equa32}
  \|A_{J}^{\mathbb{T}}\nabla\!L_{\sigma,\gamma}(A\overline{x})\|^2
  =\|A_{I_{2}J}^{\mathbb{T}}e\|^2,\
  \langle \nabla\!L_{\sigma,\gamma}(A\overline{x}),A\overline{x}\rangle^2
  =[{\textstyle\sum_{i\in I_{2}}}(A\overline{x})_i]^2.
 \end{align}
 By comparing \eqref{subdiff-Theta} with Lemma \ref{lemma-critical} (ii), we have
 $\partial F_{\sigma,\gamma}(x) =\partial\!\widetilde{f}_{\sigma,\gamma}(x)+\lambda\partial\|x\|_0$.
 Since ${\rm supp}(x)=J$, we also have $\partial\|x\|_0=\{v\in\mathbb{R}^n\,|\,v_i=0\ {\rm for}\ i\in J\}$.
 Then, it holds that
 \begin{align}\label{temp-equa33}
  {\rm dist}^2(0, \partial F_{\sigma,\gamma}(x))
  &=\min_{u\in\partial\!\widetilde{f}_{\sigma,\gamma}(x),v\in \lambda\partial\|x\|_0}\|u+v\|^2
  =\min_{u\in\partial\!\widetilde{f}_{\sigma,\gamma}(x)}\|u_{J}\|^2\nonumber\\
  &=\min_{\alpha\in\mathbb{R}}\|A_{J}^{\mathbb{T}}\nabla\!L_{\sigma,\gamma}(Ax)+\alpha x_{J}\|^2\nonumber\\
  &=\min_{\alpha\in\mathbb{R}}\alpha^2+2\langle Ax,\nabla\!L_{\sigma,\gamma}(Ax)\rangle\alpha+\|A_{J}^{\mathbb{T}}\nabla\!L_{\sigma,\gamma}(Ax)\|^2\nonumber\\
  &=\|A_{J}^{\mathbb{T}}\nabla\!L_{\sigma,\gamma}(Ax)\|^2-\langle Ax,\nabla\!L_{\sigma,\gamma}(Ax)\rangle^2.
 \end{align}
  Since $0\in\partial F_{\sigma,\gamma}(\overline{x})=\nabla\!f_{\sigma,\gamma}(\overline{x})
  +\mathcal{N}_{\mathcal{S}}(\overline{x})+\lambda\partial\|\overline{x}\|_0$,
  from the expression of $\partial\|\overline{x}\|_0$ we have
  \[
   A_{J}^{\mathbb{T}}\nabla\!L_{\sigma,\gamma}(A\overline{x})=\overline{\alpha}\,\overline{x}_{J}
   \ \ {\rm with}\ \ \overline{\alpha}=\langle A\overline{x},\nabla\!L_{\sigma,\gamma}(A\overline{x})\rangle.
  \]
  Together with the equations \eqref{temp-equa31}-\eqref{temp-equa33}, it immediately follows that
 \begin{align*}
  0&\le {\rm dist}^2(0, \partial F_{\sigma,\gamma}(x))
   =\|A_{J}^{\mathbb{T}}\nabla\!L_{\sigma,\gamma}(Ax)\|^2-\langle Ax,\nabla\!L_{\sigma,\gamma}(Ax)\rangle^2
   -\|A_{J}^{\mathbb{T}}\nabla\!L_{\sigma,\gamma}(A\overline{x})-\overline{\alpha}\,\overline{x}_{J}\|^2\\
  &=\|A_{J}^{\mathbb{T}}\nabla\!L_{\sigma,\gamma}(Ax)\|^2-\|A_{J}^{\mathbb{T}}\nabla\!L_{\sigma,\gamma}(A\overline{x})\|^2
    +\big[\langle\nabla\!L_{\sigma,\gamma}(A\overline{x}),A\overline{x}\rangle^2
     \!-\!\langle \nabla\!L_{\sigma,\gamma}(Ax),Ax\rangle^2\big]\nonumber\\
  &={\textstyle\sum_{i\in I_{2}}}[(A\overline{x})_i-(Ax)_i]\cdot{\textstyle\sum_{i\in I_{2}}}[(A\overline{x})_i+(Ax)_i].
 \end{align*}
 Since ${\textstyle\sum_{i\in I_{2}}}[(A\overline{x})_i+(Ax)_i]<0$,
 the last inequality implies that ${\textstyle\sum_{i\in I_{2}}}[(A\overline{x})_i-(Ax)_i]\le 0$,
 which is a contradiction to the inequality \eqref{Phi-gamma}.
 The proof is then completed.
 \end{proof}
 \begin{remark}\label{remark2-PGMe1}
  By the definition of $\Gamma(\overline{x})$, when $\gamma$ is small enough, 
  it is highly possible for $\Gamma(\overline{x})=\emptyset$ and then 
  for $F_{\sigma,\gamma}$ to be a KL function of exponent $0$. 
 \end{remark}
%------------------------------------------------------------------------------
 \subsection{PG with extrapolation for solving \eqref{Esurrogate}}\label{sec4.2}

  By the proof of Lemma \ref{lemma-critical}, $\Xi_{\sigma,\gamma}$ is a smooth function
  and $\nabla\Xi_{\sigma,\gamma}$ is globally Lipschitz continuous with
  Lipschitz constant $L_{\Xi}\le \gamma^{-1}\|A\|^2+\!\lambda\rho^2\max(\frac{a+1}{2},\frac{a+1}{2(a-1)})$.
  While by Proposition \ref{proxm-hlam}, the proximal mapping of $h_{\lambda,\rho}$ has a closed form.
  This motivates us to apply the PG method with extrapolation to solving the problem \eqref{Esurrogate}.
 %----------------------------------------------------------------------------------------------
 \begin{algorithm}[H]
  \caption{\label{PGMe2}{\bf (PGe-scad for solving the problem \eqref{Esurrogate})}}
  \textbf{Initialization:} Choose $\varsigma\in(0,1),0<\tau<(1\!-\!\varsigma)L_{\Xi}^{-1},0<\beta_{\rm max}\le{\frac{\sqrt{\varsigma(\tau^{-1}-L_{\Xi})\tau^{-1}}}{2(\tau^{-1}+L_{\Xi})}}$ and
  an initial point $x^0\in\mathcal{S}$. Set $x^{-1}=x^{0}$ and $k:=0$.

  \medskip
  \noindent
 \textbf{while} the termination condition is not satisfied \textbf{do}
  \begin{itemize}
   \item[{\bf1.}] Let $\widetilde{x}^k=x^{k}+\beta_k(x^k-x^{k-1})$ and compute $x^{k+1}\in\mathcal{P}_{\!\tau}h_{\lambda,\rho}(\widetilde{x}^k\!-\!\tau\nabla\Xi_{\sigma,\gamma}(\widetilde{x}^k))$.

   \item[{\bf3.}] Choose $\beta_{k+1}\in[0,\beta_{\rm max}]$. Let $k\leftarrow k+1$ and go to Step 1.
  \end{itemize}
  \textbf{end (while)}
  \end{algorithm}
  
  Similar to Algorithm \ref{PGMe1}, the extrapolation parameter
 $\beta_k$ in Algorithm \ref{PGMe2} can be chosen in terms of
 the rule in \eqref{accelerate}. For any $\tau>0$ and $\varsigma\in(0,1)$,
 we define the potential function
 \begin{equation}\label{Upsion}
    \Upsilon_{\!\tau,\varsigma}(x,u):=G_{\sigma,\gamma,\rho}(x)+\frac{\varsigma}{4\tau}\|x-u\|^2
    \quad\ \forall (x,u)\in\mathbb{R}^n\times\mathbb{R}^n.
  \end{equation}
  Then, by following the same arguments as those for Lemma \ref{lemma1-PGMe1} and \ref{lemma2-PGMe1},
  we can establish the following properties of $\Upsilon_{\!\tau,\varsigma}$ on
  the sequence $\{x^k\}_{k\in\mathbb{N}}$ generated by Algorithm \ref{PGMe2}.
  %--------------------------------------------------------------------------------------------
 \begin{lemma}\label{lemma1-PGMe2}
  Let $\{x^k\}_{k\in\mathbb{N}}$ be the sequence generated by Algorithm \ref{PGMe2}
  and denote by $\pi(x^0)$ the set of accumulation points of $\{x^k\}_{k\in\mathbb{N}}$.
  Then, the following assertions hold.
  \begin{itemize}
  \item[(i)] For each $k\in\mathbb{N}$,
              \(
                \Upsilon_{\!\tau,\varsigma}(x^{k+1},x^k)\le \Upsilon_{\!\tau,\varsigma}(x^{k},x^{k-1})
                  -\frac{\varsigma(\tau^{-1}-L_{\Xi})}{2}\|x^{k+1}\!-\!x^k\|^2.
              \)
             Consequently, $\{\Upsilon_{\!\tau,\varsigma}(x^{k},x^{k-1})\}_{k\in\mathbb{N}}$ is convergent
                and $\sum_{k=1}^{\infty}\|x^{k+1}\!-\!x^k\|^2<\infty$.

   \item[(ii)] For each $k\in\mathbb{N}$, there exists $w^{k}\in\partial\Upsilon_{\!\tau,\varsigma}(x^{k},x^{k-1})$
                with $\|w^{k+1}\|\le b_1'\|x^{k+1}\!-\!x^k\|+b_2'\|x^k\!-\!x^{k-1}\|$, where
                $b_1'>0$ and $b_2'>0$ are the constants independent of $k$.

\item [(iii)] $\pi(x^0)$ is a nonempty compact set and $\pi(x^0)\subseteq S_{\tau,h_{\lambda,\rho}}$.

\item [(iv)] $\lim_{k\to\infty}{\rm dist}((x^k,x^{k-1}),\pi(x^0)\times \pi(x^0))=0$,
             and $\Upsilon_{\!\tau,\varsigma}$ is finite and keeps the constant
             on the set $\pi(x^0)\times \pi(x^0)$.
  \end{itemize}
 \end{lemma}

  By using Lemma \ref{lemma1-PGMe2} and following the same arguments
  as those for Theorem \ref{global-conv}, it is not difficult to achieve
  the following convergence results for Algorithm \ref{PGMe2}.
%---------------------------------------------------------------------------------
 \begin{theorem}\label{global2-conv}
  Let $\{x^k\}_{k\in\mathbb{N}}$ be the sequence generated by Algorithm \ref{PGMe2}.
  Then,
 \begin{itemize}
 \item[(i)] $\sum_{k=1}^{\infty}\|x^{k+1}-x^k\|<\infty$ and consequently $\{x^k\}_{k\in \mathbb{N}}$
             converges to some $x^*\in S_{\tau,h_{\lambda,\rho}}$.

 \item[(ii)] If $G_{\sigma,\gamma,\rho}$ is a KL function of exponent $1/2$,
             then there exist $c_2>0$ and $\varrho\in(0,1)$ such that
             for all sufficiently large $k$, $\|x^k-x^*\|\le c_2\varrho^k$.
 \end{itemize}
 \end{theorem}

 Theorem \ref{global2-conv} (ii) requires that $G_{\sigma,\gamma,\rho}$ is a KL function
 of exponent $1/2$. We next show that it indeed holds under a little
 stronger condition than the one used in Lemma \ref{KL-exponent}.
%---------------------------------------------------------------------------------------------
 \begin{lemma}\label{KL-exponent1}
  If $\lambda$ and $\rho$ are chosen with
  $\lambda\rho\!>\!{\displaystyle\max_{z\in{\rm crit}G_{\sigma,\gamma,\rho}}}\|\nabla\!f_{\sigma,\gamma}(z)\|_\infty$ 
  and all $\overline{x}\!\in{\rm crit}G_{\sigma,\gamma,\rho}$ satisfy
  $\Gamma(\overline{x})=\emptyset$ and $|\overline{x}|_{\rm nz}>\frac{2a}{\rho(a-1)}$,
  then $G_{\sigma,\gamma,\rho}$ is a KL function of exponent $0$.
 \end{lemma}
 \begin{proof}
  Fix any $\overline{x}\in{\rm crit}G_{\sigma,\gamma,\rho}$ with $\Gamma(\overline{x})=\emptyset$
  and $|\overline{x}|_{\rm nz}>\frac{2a}{\rho(a-1)}$. Let $J={\rm supp}(\overline{x})$
  and $\overline{J}=[n]\backslash J$. Let $\theta_{\!\rho}$ be the function in the proof
  of Lemma \ref{lemma-critical}. Since $[\nabla\theta_{\!\rho}(\overline{x})]_{\overline{J}}=0$,
  the given assumption means that  $\|[\nabla\!f_{\sigma,\gamma}(\overline{x})\!-\!\lambda\rho\nabla\theta_{\!\rho}(\overline{x})]_{\overline{J}}\|_\infty<\lambda\rho$.
  By the continuity, there exists $\delta_0>0$ such that for all $x\in\mathbb{B}(\overline{x},\delta_0)$,
  \(
    \|[\nabla\!f_{\sigma,\gamma}(x)\!-\!\lambda\rho\nabla\theta_{\!\rho}(x)]_{\overline{J}}\|_\infty<\lambda\rho.
  \)
  Let $I_0,I_1$ and $I_2$ be same as
  in the proof of Lemma \ref{KL-exponent}. Then, there exists
  $\delta_1>0$ such that for all $x\in\mathbb{B}(\overline{x},\delta_1)$,
  ${\rm supp}(x)\supseteq J$ and the relations in \eqref{Phi-gamma} hold.
  By the continuity, there exist $\delta_2>0$ such that
  for all $x\in\mathbb{B}(\overline{x},\delta_2)$, $|x_i|>\frac{2a}{\rho(a-1)}$
  with $i\in{\rm supp}(x)$ and
  \begin{equation}\label{temp-equaXi}
    \Xi_{\sigma,\gamma}(x)+\lambda\rho{\textstyle\sum_{i\in J}}|x_i|
    >\Xi_{\sigma,\gamma}(\overline{x})+\lambda\rho{\textstyle\sum_{i\in J}}|\overline{x}_i|-{a\lambda}/{(a\!-\!1)}.
  \end{equation}
  Set $\delta=\min(\delta_0,\delta_1,\delta_2)$. Pick any $\eta\in(0,\frac{a\lambda}{2(a-1)})$.
  Next we argue that
  $\mathbb{B}(\overline{x},\varepsilon)
  \cap[G_{\sigma,\gamma,\rho}(\overline{x})<G_{\sigma,\gamma,\rho}<G_{\sigma,\gamma,\rho}(\overline{x})+\eta]=\emptyset$,
  which by Definition \ref{KL-def} implies that $G_{\sigma,\gamma,\rho}$
  is a KL function of exponent $1/2$. Suppose on the contradiction that
  there exists $x\in\mathbb{B}(\overline{x},\varepsilon)
  \cap[G_{\sigma,\gamma,\rho}(\overline{x})<G_{\sigma,\gamma,\rho}<G_{\sigma,\gamma,\rho}(\overline{x})+\eta]$.
  From $G_{\sigma,\gamma,\rho}(x)<G_{\sigma,\gamma,\rho}(\overline{x})+\eta$, we have $x\in\mathcal{S}$,
  which along with ${\rm supp}(x)\supseteq J$ implies that ${\rm supp}(x)=J$ (if not,
  we will have $\Xi_{\sigma,\gamma}(x)+\lambda\rho\sum_{i\in J}|x_i|
  +\lambda\rho\sum_{i\in{\rm supp}(x)\backslash J}|x_i|<G_{\sigma,\gamma,\rho}(x)
  <\Xi_{\sigma,\gamma}(\overline{x})+\lambda\rho\sum_{i\in J}|\overline{x}_i|+\eta$,
  which along with \eqref{temp-equaXi} and $|x_i|>\frac{2a}{\rho(a-1)}$ for
  $i\in{\rm supp}(x)\backslash J$ implies that $\eta>\frac{a\lambda}{a-1}$,
  a contradiction to $\eta<\frac{a\lambda}{2(a-1)}$). Now from ${\rm supp}(x)=J$
  and $|x_i|>\frac{2a}{\rho(a-1)}$ for $i\in{\rm supp}(x)$, it is not hard to
  verify that $\|x\|_1-\theta_{\!\rho}(x)=\|\overline{x}\|_1-\theta_{\!\rho}(\overline{x})$.
  Together with $x\in\mathcal{S}$ and the expression of $L_{\sigma,\gamma}$, we have
  \begin{equation}\label{G-gamma}
   0<G_{\sigma,\gamma,\rho}(x)-G_{\sigma,\gamma,\rho}(\overline{x})
   =L_{\sigma,\gamma}(Ax)-L_{\sigma,\gamma}(A\overline{x})
   ={\textstyle\sum_{i\in I_{2}}}[(A\overline{x})_i-(Ax)_i].
 \end{equation}
 Moreover, the equalities in \eqref{temp-equa31}-\eqref{temp-equa32} still hold for $x$.
 Let $\widetilde{f}_{\sigma,\gamma}$ be same as in the proof of Lemma \ref{KL-exponent}.
 Clearly, $\partial G_{\sigma,\gamma,\rho}(x)=\partial\!\widetilde{f}_{\sigma,\gamma}(x)
 +\lambda\rho[\partial\|x\|_1-\nabla\theta_{\!\rho}(x)]$. Then, it holds that
 \[
   {\rm dist}^2(0, \partial G_{\sigma,\gamma,\rho}(x))
   =\min_{u\in\partial\!\widetilde{f}_{\sigma,\gamma}(x),v\in\lambda\rho[\partial\|x\|_1-\nabla\theta_{\!\rho}(x)]}\|u+v\|^2.
 \]
 Notice that $\partial\!\widetilde{f}_{\sigma,\gamma}(x)=\{\nabla\!f_{\sigma,\gamma}(x)+\alpha x\,|\,\alpha\in\mathbb{R}\},
 \|[\nabla\!f_{\sigma,\gamma}(x)\!-\!\lambda\rho\nabla\theta_{\!\rho}(x)]_{\overline{J}}\|_\infty<\lambda\rho$
 and $v_{\overline{J}}\in[-\lambda\rho,\lambda\rho]-\lambda\rho[\nabla\theta_{\!\rho}(x)]_{\overline{J}}$.
 From the last equation, it follows that
 \begin{align}\label{temp-equa34}
  {\rm dist}^2(0, \partial G_{\sigma,\gamma,\rho}(x))
  &=\min_{u\in\partial\!\widetilde{f}_{\sigma,\gamma}(x)}\|u_{J}\|^2
   =\min_{\alpha\in\mathbb{R}}\|A_{J}^{\mathbb{T}}\nabla\!L_{\sigma,\gamma}(Ax)+\alpha x_{J}\|^2\nonumber\\
  &=\min_{\alpha\in\mathbb{R}}\alpha^2+2\langle Ax,\nabla\!L_{\sigma,\gamma}(Ax)\rangle\alpha+\|A_{J}^{\mathbb{T}}\nabla\!L_{\sigma,\gamma}(Ax)\|^2\nonumber\\
  &=\|A_{J}^{\mathbb{T}}\nabla\!L_{\sigma,\gamma}(Ax)\|^2-\langle Ax,\nabla\!L_{\sigma,\gamma}(Ax)\rangle^2.
 \end{align}
  Since $0\in\partial G_{\sigma,\gamma,\rho}(\overline{x})=\nabla\!f_{\sigma,\gamma}(\overline{x})
  +\mathcal{N}_{\mathcal{S}}(\overline{x})+\lambda\rho[\partial\|\overline{x}\|_1-\nabla\theta_{\rho}(\overline{x})]$
  and $|\overline{x}|_{\rm nz}>\frac{2a}{\rho(a-1)}$, by the proof of Lemma \ref{lemma-critical} (iii)
  and $\mathcal{N}_{\mathcal{S}}(\overline{x})=\{\alpha\overline{x}\,|\,\alpha\in\mathbb{R}\}$,
  there exists $\overline{\alpha}\in\mathbb{R}$ such that
  $A_{J}^{\mathbb{T}}\nabla\!L_{\sigma,\gamma}(A\overline{x})=\overline{\alpha}\,\overline{x}_{J}$ 
  with $\overline{\alpha}=\langle A\overline{x},\nabla\!L_{\sigma,\gamma}(A\overline{x})\rangle$.
  Together with \eqref{temp-equa31}-\eqref{temp-equa32} and \eqref{temp-equa34},
 \begin{align*}
  0&\le\|A_{J}^{\mathbb{T}}\nabla\!L_{\sigma,\gamma}(Ax)\|^2-\langle Ax,\nabla\!L_{\sigma,\gamma}(Ax)\rangle^2
   -\|A_{J}^{\mathbb{T}}\nabla\!L_{\sigma,\gamma}(A\overline{x})-\overline{\alpha}\,\overline{x}_{J}\|^2\\
  &=\|A_{J}^{\mathbb{T}}\nabla\!L_{\sigma,\gamma}(Ax)\|^2-\|A_{J}^{\mathbb{T}}\nabla\!L_{\sigma,\gamma}(A\overline{x})\|^2
    +\big[\langle\nabla\!L_{\sigma,\gamma}(A\overline{x}),A\overline{x}\rangle^2
     \!-\!\langle \nabla\!L_{\sigma,\gamma}(Ax),Ax\rangle^2\big]\nonumber\\
  &={\textstyle\sum_{i\in I_{2}}}[(A\overline{x})_i-(Ax)_i]\cdot{\textstyle\sum_{i\in I_{2}}}[(A\overline{x})_i+(Ax)_i].
 \end{align*}
 Since ${\textstyle\sum_{i\in I_{2}}}[(A\overline{x})_i+(Ax)_i]<0$,
 the last inequality implies that ${\textstyle\sum_{i\in I_{2}}}[(A\overline{x})_i-(Ax)_i]\le 0$,
 which is a contradiction to the inequality \eqref{G-gamma}.
 The proof is then completed.
 \end{proof}
%-----------------------------------------------------------------------------------------------------
 \section{Numerical experiments}\label{sec5}

 In this section we demonstrate the performance of the zero-norm regularized DC loss model \eqref{znorm-Moreau} and its surrogate \eqref{Esurrogate}, which are respectively solved with PGe-znorm and PGe-scad. All numerical experiments are performed in MATLAB on a laptop running on 64-bit Windows System with an Intel(R) Core(TM) i7-7700HQ CPU 2.80GHz and 16 GB RAM.
 The MATLAB package for reproducing all the numerical results can be found at \url{https://github.com/SCUT-OptGroup/onebit}.
%---------------------------------------------------------------------------------------------------
 \subsection{Experiment setup}\label{sec5.1}

  The setup of our experiments is similar to the one in \cite{Yan12,HuangJ18}.
  Specifically, we generate the original $s^*$-sparse signal $x^{\rm true}$ with
  the support $T$ chosen uniformly from $\{1,2,\ldots,n\}$ and $(x^{\rm true})_{T}$
  taking the form of ${\xi}/{\|\xi\|}$, where the entries of $\xi\in\mathbb{R}^{s^*}$
  are drawn from the standard normal distribution. Then, we obtain the observation vector $b$
  via \eqref{observation}, where the sampling matrix $\Phi\in\mathbb{R}^{m\times n}$
  is generated in the two ways: (I) the rows of $\Phi$ are i.i.d. samples of $N(0,\Sigma)$
  with $\Sigma_{ij}=\mu^{|i-j|}$ for $i,j\in[n]$ (II) the entries of $\Phi$ are i.i.d.
  and follow the standard normal distribution; the noise $\varepsilon\in\mathbb{R}^m$
  is generated from $N(0,\varpi^2I)$; and the entries of $\zeta$ is set by
  $\mathbb{P}(\zeta_i=1)=1-\mathbb{P}(\zeta_i=-1)=1-r$. In the sequel, we denote
  the corresponding data with the two triples $(m,n,s^*)$ and $(\mu,\varpi,r)$,
  where $\mu$ means the correlation factor,
  $\varpi$ denotes the noise level and $r$ means the sign flip ratio.

  We evaluate the quality of an output $x^{\rm sol}$ of a solver in terms of
  the mean squared error (MSE), the Hamming error (Herr), the ratio of missing support (FNR)
  and the ratio of misidentified support (FPR), which are defined as follows
  \begin{align*}
    {\rm MSE}:=\|x^{\rm sol}\!-\!x^{\rm true}\|,\
    {\rm Herr}:=\dfrac{1}{m}\|{\rm sign}(\Phi x^{\rm sol})-{\rm sign}(\Phi x^{\rm true})\|_0,\\
    {\rm FNR}:= \frac{|T\backslash {\rm supp}(x^{\rm sol})|}{|T|}\ \ {\rm and}\ \
    {\rm FPR}:= \frac{|{\rm supp}(x^{\rm sol})\backslash T|}{n-|T|},\qquad
  \end{align*}
  where, in our numerical experiments, a component of a vector $z\in\mathbb{R}^n$
  being nonzero means that its absolute value is larger than $10^{-5}\|z\|_\infty$.
  Clearly, a solver has a better performance if its output has the smaller MSE,
  ${\rm Herr}$, {\rm FNR} and {\rm FPR}.
  %---------------------------------------------------------------------------------------------------
 \subsection{Implementation of PGe-znorm and PGe-scad}\label{sec5.2}

  From the definition of $x^{k+1}$ in PGe-znorm and PGe-scad,
  we have $(x^{k+1}\!-\!\widetilde{x}^k)+\widetilde{x}^k
  \in\mathcal{P}_{\!\tau}g_{\lambda}(\widetilde{x}^k\!-\!\tau\nabla\!f_{\sigma,\gamma}(\widetilde{x}^k))$
  and $(x^{k+1}-\widetilde{x}^k)+\widetilde{x}^k
  \in\mathcal{P}_{\!\tau}h_{\lambda,\rho}(\widetilde{x}^k\!-\!\tau\nabla\!f_{\sigma,\gamma}(\widetilde{x}^k))$.
  Together with the expression of $\mathcal{P}_{\!\tau}g_{\lambda}$ and $\mathcal{P}_{\!\tau}h_{\lambda,\rho}$,
  when $\|x^{k+1}\!-\!\widetilde{x}^k\|$ is small enough,
  $\widetilde{x}^k$ can be viewed as an approximate $\tau$-stationary point.
  Hence, we terminate PGe-znorm and PGe-scad at the iterate $x^k$
  once $\|x^{k+1}-\widetilde{x}^k\|\le 10^{-6}$ or $k\ge2000$. In addition, we also terminate
  the two algorithms at $x^k$ when $\frac{|F_{\sigma,\gamma}(x^{k-j})-F_{\sigma,\gamma}(x^{k-j})|}
  {\max(1,F_{\sigma,\gamma}(x^{k-j}))}\le 10^{-10}$ for $k\ge 100$ and $j=0,1,\ldots,9$.
  The extrapolation parameters $\beta_k$ in the two algorithms are chosen by \eqref{accelerate}
  with $\beta_{\rm max}=0.235$. The starting point $x^0$ of PGe-znorm and PGe-scad
  is always chosen to be ${e^{\mathbb{T}}A}/{\|e^{\mathbb{T}}A\|}$.
%---------------------------------------------------------------------------------------------------
 \subsection{Choice of the model parameters}\label{sec5.3}

  The model \eqref{znorm-Moreau} and its surrogate \eqref{Esurrogate} involve
  the parameters $\lambda>0,\rho>0$ and $0<\gamma<\sigma/2$. By
  Figure \ref{fig_gam} and \ref{fig_rho}, we choose $\gamma=0.05,\rho=10$
  for the subsequent tests. To choose an appropriate $\sigma>2\gamma$,
  we generate the original signal $x^{\rm true}$, the sampling matrix $\Phi$ of type I
  and the observation $b$ with $(m,n,s^*,r)=(500,1000,5,1.0)$,
  and then solve the model \eqref{znorm-Moreau} associated to $\gamma=0.05,\lambda=10$ for
  each $\sigma\in\{0.2,0.4,\ldots,3\}$ with PGe-znorm and the model \eqref{Esurrogate}
  associated to $\gamma=0.05,\lambda=5,\rho=10$ for each $\sigma\in\{0.2,0.4,\ldots,3\}$ with PGe-scad.
  Figure \ref{fig_sig} plots the average MSE of $50$ trials for each $\sigma$.
  We see that $\sigma\in[0.6,1.2]$ is a desirable choice, so choose $\sigma=0.8$
  for the two models in the subsequent experiments.
%----------------------------------------------------------------------------------------------
 \begin{figure}[!h]
   \centering
    \setlength{\abovecaptionskip}{10pt}
  \includegraphics[width=16cm,height=6cm]{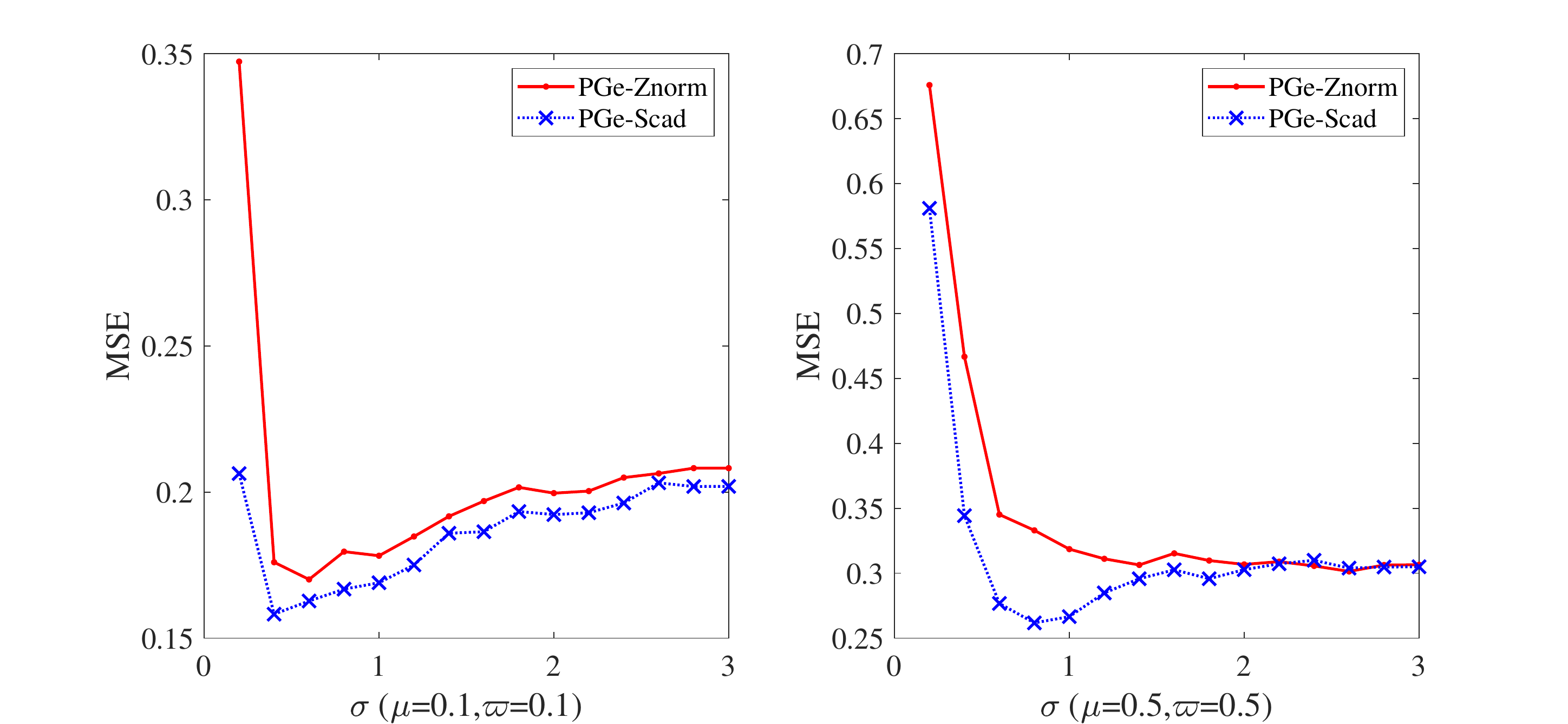}
   \caption{Influence of $\sigma$ on the performance of the models \eqref{znorm-Moreau} and \eqref{Esurrogate}}
  \label{fig_sig}
 \end{figure}

  Next we take a closer look at the influence of $\lambda$ on the models \eqref{znorm-Moreau}
  and \eqref{Esurrogate}. To this end, we generate the signal $x^{\rm true}$,
  the sampling matrix $\Phi$ of type I, and the observation $b$ with $(m,n,s^*)=(500,1000,5)$
  and $(\mu,\varpi)=(0.3,0.1)$, and then solve the model \eqref{znorm-Moreau} associated to
  $\sigma=0.8,\gamma=0.05$ for each $\lambda\in\{1,3,5,\ldots,49\}$ with PGe-znorm and
  solve the model \eqref{Esurrogate} associated to $\sigma=0.8,\gamma=0.05,\rho=10$
  for each $\lambda\in\{0.5,1.5,2.5,\ldots,24.5\}$ with PGe-scad.
  Figure \ref{fig_lam} plots the average MSE of $50$ trials for each $\lambda$.
  When $r=0.15$, the MSE from the model \eqref{znorm-Moreau} has a small variation
  for $\lambda\in[7,49]$, while the MSE from the model \eqref{Esurrogate} has
  a small variation for $\lambda\in[3,24.5]$. When $r=0.05$, the MSE from the model
  \eqref{znorm-Moreau} has a small variation for $\lambda\in[5,49]$ and is relatively
  low for $\lambda\in[5,16]$, while the MSE from the model \eqref{Esurrogate} has a tiny
  change for $\lambda\in[1.5,24.5]$. In view of this, we always choose $\lambda=8$
  for the model \eqref{znorm-Moreau}, and choose $\lambda=4$ and $\lambda=8$ for the model \eqref{Esurrogate}
  with $n\le 5000$ and $n>5000$, respectively, in the subsequent experiments.
%----------------------------------------------------------------------------------------------
 \begin{figure}[!h]
   \centering
    \setlength{\abovecaptionskip}{10pt}
  \includegraphics[width=16cm,height=6cm]{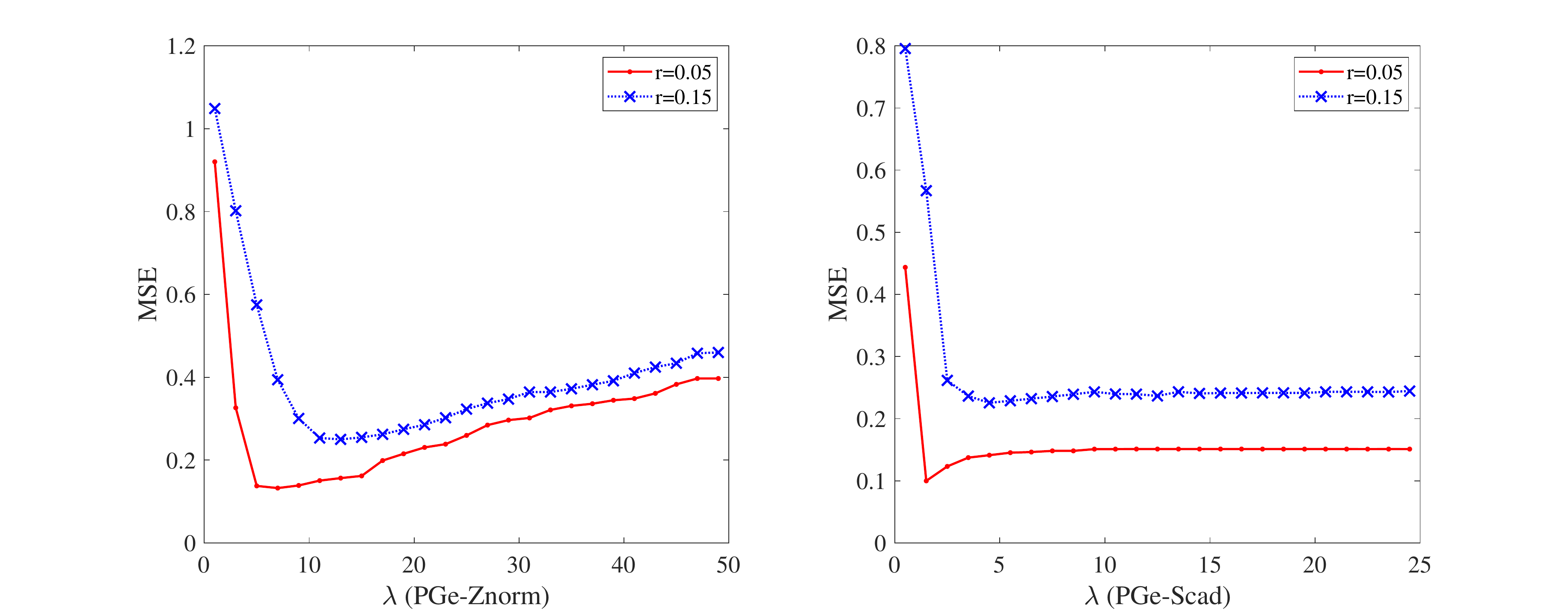}
   \caption{Influence of $\lambda$ on the MSE from the models \eqref{znorm-Moreau} and \eqref{Esurrogate}}
  \label{fig_lam}
 \end{figure}
%----------------------------------------------------------------------------------------------------
 \subsection{Numerical comparisons}\label{sec5.4}

  We compare PGe-znorm and PGe-scad with six state-of-the-art solvers, which are
  BIHT-AOP \cite{Yan12}, PIHT \cite{HuangY18}, PIHT-AOP \cite{HuangS18}, GPSP \cite{Zhou21}
  (\url{https://github.com/ShenglongZhou/GPSP}), PDASC \cite{HuangJ18} and WPDASC \cite{Fan21}.
  Among others, the codes for BIHT-AOP, PIHT and PIHT-AOP can be found at \url{http://www.esat.kuleuven.be/stadius/ADB/huang/downloads/1bitCSLab.zip} and
  the codes for PDASC and WPDASC can be found at \url{https://github.com/cjia80/numericalSimulation}.
  It is worth pointing out that BIHT-AOP, GPSP and PIHT-AOP all require
  an estimation on $s^*$ and $r$ as an input, PIHT require an estimation
  on $s^*$ as an input, while PDASC, WPDASC, PGe-znorm and PGe-scad do not need
  these prior information. For the solvers to require an estimation on $s^*$
  and $r$, we directly input the true sparsity $s^*$ and $r$ as those papers do.
  During the testing, PGe-znorm and PGe-scad use the parameters described before,
  and other solvers use their default setting except the PIHT is terminated once
  its iteration is over $100$.
  
  We first apply the eight solvers to solving the test problems with the sampling matrix
  of type I and low noise. Table \ref{table1} reports their average MSE, Herr, FNR, FPR
  and CPU time for $50$ trials. We see that among the four solvers without requiring any
  information on $x^{\rm true}$, PGe-scad and PGe-znorm yield the lower MSE, {\rm Herr}
  and FNR than PDASC and WPDASC do, and PGe-scad is the best one in terms of MSE, Herr and FNR;
  while among the four solvers requiring some information on $x^{\rm true}$,
  BIHT-AOP and PIHT-AOP yield the smaller MSE and {\rm Herr} than PIHT and GPSP do,
  and the former also yields the lower FNR and FPR under the scenario of $r=0.05$.
  When comparing PGe-scad with BIHT-AOP and PIHT-AOP, the former yields the smaller
  MSE, Herr, FNR and FPR under the scenario of $r=0.15$, and under the scenario
  of $r=0.05$, it also yields the comparable MSE, Herr and FNR as BIHT-AOP and PIHT-AOP do.
%-------------------------------------------------------------------------------------------------------------
\begin{table}[!h]
\centering
\tiny
\captionsetup{font={scriptsize}}
\caption{Numerical comparisons of eight solvers for test problems with $\Phi$ of type I and low noise}
\label{table1}
\resizebox{\textwidth}{35mm}{
\begin{tabular}{|llllllllllllllll|}
\hline
\multicolumn{16}{|c|}{$m=800,n=2000,s^*=10,\varpi=0.1,{\bf r=0.05}$}                                                                                                                                                                                                                                 \\ \hline
\multicolumn{1}{|l|}{}&\multicolumn{5}{c|}{$\mu=0.1$}&\multicolumn{5}{c|}{$\mu=0.3$}&\multicolumn{5}{c|}{$\mu=0.5$}
\\ \hline
\multicolumn{1}{|c|}{solvers}&
\multicolumn{1}{c|}{MSE}&\multicolumn{1}{c|}{Herr}&\multicolumn{1}{c|}{FNR}&\multicolumn{1}{c|}{FPR}&\multicolumn{1}{c|}{time(s)}& \multicolumn{1}{c|}{MSE}&\multicolumn{1}{c|}{Herr}&\multicolumn{1}{c|}{FNR}&\multicolumn{1}{c|}{FPR}&\multicolumn{1}{c|}{time(s)}&
\multicolumn{1}{c|}{MSE}&\multicolumn{1}{c|}{Herr}&\multicolumn{1}{c|}{FNR}&\multicolumn{1}{c|}{FPR}&\multicolumn{1}{c|}{time(s)}
\\ \hline
\multicolumn{1}{|c|}{PIHT}&\multicolumn{1}{c|}{2.57e-1}&\multicolumn{1}{c|}{6.80e-2}&\multicolumn{1}{c|}{3.26e-1}&\multicolumn{1}{c|}{1.64e-3}&\multicolumn{1}{c|}{1.55e-1}& \multicolumn{1}{c|}{2.75e-1}&\multicolumn{1}{c|}{7.27e-2}&\multicolumn{1}{c|}{3.52e-1}&\multicolumn{1}{c|}{1.77e-3}&\multicolumn{1}{c|}{1.57e-1}&
\multicolumn{1}{c|}{3.52e-1}&\multicolumn{1}{c|}{9.22e-2}&\multicolumn{1}{c|}{4.24e-1}&\multicolumn{1}{c|}{2.13e-3}&\multicolumn{1}{c|}{1.58e-1}\\ \hline
\multicolumn{1}{|c|}{BIHT-AOP}&\multicolumn{1}{c|}{\color{red}1.46e-1}&\multicolumn{1}{c|}{\color{red}4.36e-2}&\multicolumn{1}{c|}{\color{red}1.94e-1}&\multicolumn{1}{c|}{\color{red}9.75e-4}&\multicolumn{1}{c|}{5.30e-1}&
\multicolumn{1}{c|}{\color{red}1.32e-1}&\multicolumn{1}{c|}{\color{red}3.85e-2}&\multicolumn{1}{c|}{\color{red}1.80e-1}&\multicolumn{1}{c|}{\color{red}9.05e-4}&\multicolumn{1}{c|}{5.47e-1}&
\multicolumn{1}{c|}{1.46e-1}&\multicolumn{1}{c|}{4.18e-2}&\multicolumn{1}{c|}{2.06e-1}&\multicolumn{1}{c|}{1.04e-3}&\multicolumn{1}{c|}{5.45e-1}\\ \hline
\multicolumn{1}{|c|}{PIHT-AOP}&\multicolumn{1}{c|}{1.61e-1}&\multicolumn{1}{c|}{4.67e-2}&\multicolumn{1}{c|}{2.06e-1}&\multicolumn{1}{c|}{1.04e-3}&\multicolumn{1}{c|}{1.81e-1}& \multicolumn{1}{c|}{1.55e-1}&\multicolumn{1}{c|}{4.60e-2}&\multicolumn{1}{c|}{1.90e-1}&\multicolumn{1}{c|}{9.55e-4}&\multicolumn{1}{c|}{1.92e-1}&
\multicolumn{1}{c|}{\color{red}1.40e-1}&\multicolumn{1}{c|}{\color{red}4.17e-2}&\multicolumn{1}{c|}{\color{red}2.02e-1}&\multicolumn{1}{c|}{\color{red}1.02e-3}&\multicolumn{1}{c|}{1.86e-1}\\ \hline
\multicolumn{1}{|c|}{GPSP}&\multicolumn{1}{c|}{1.91e-1}&\multicolumn{1}{c|}{5.02e-2}&\multicolumn{1}{c|}{2.56e-1}&\multicolumn{1}{c|}{1.29e-3}&\multicolumn{1}{c|}{1.60e-2}
&\multicolumn{1}{c|}{1.87e-1}&\multicolumn{1}{c|}{4.83e-2}&\multicolumn{1}{c|}{2.40e-1}&\multicolumn{1}{c|}{1.21e-3}&\multicolumn{1}{c|}{1.83e-2}&
\multicolumn{1}{c|}{1.89e-1}&\multicolumn{1}{c|}{4.78e-2}&\multicolumn{1}{c|}{2.52e-1}&\multicolumn{1}{c|}{1.27e-3}&\multicolumn{1}{c|}{2.31e-2}\\
\hline\hline
\multicolumn{1}{|c|}{\bf PGe-scad}&\multicolumn{1}{c|}{2.15e-1}&\multicolumn{1}{c|}{6.70e-2}&\multicolumn{1}{c|}{\color{red}3.34e-1}&\multicolumn{1}{c|}{\color{red}0}&\multicolumn{1}{c|}{2.29e-1}&
\multicolumn{1}{c|}{\color{red}2.04e-1}&\multicolumn{1}{c|}{\color{red}6.36e-2}&\multicolumn{1}{c|}{\color{red}3.32e-1}&\multicolumn{1}{c|}{\color{red}0}&\multicolumn{1}{c|}{2.79e-1}&
\multicolumn{1}{c|}{\color{red}2.10e-1}&\multicolumn{1}{c|}{\color{red}6.42e-2}&\multicolumn{1}{c|}{\color{red}3.44e-1}&\multicolumn{1}{c|}{\color{red}1.01e-5}&\multicolumn{1}{c|}{2.82e-1}\\
\hline
\multicolumn{1}{|c|}{\bf PGe-znorm}&\multicolumn{1}{c|}{\color{red}2.10e-1}&\multicolumn{1}{c|}{\color{red}6.52e-2}&\multicolumn{1}{c|}{3.58e-1}&\multicolumn{1}{c|}{2.01e-5}&\multicolumn{1}{c|}{1.19e-1}&
\multicolumn{1}{c|}{2.10e-1}&\multicolumn{1}{c|}{6.41e-2}&\multicolumn{1}{c|}{3.62e-1}&\multicolumn{1}{c|}{4.02e-5}&\multicolumn{1}{c|}{1.24e-1}&
\multicolumn{1}{c|}{2.22e-1}&\multicolumn{1}{c|}{6.82e-2}&\multicolumn{1}{c|}{3.72e-1}&\multicolumn{1}{c|}{3.02e-5}&\multicolumn{1}{c|}{1.26e-1}\\ \hline
\multicolumn{1}{|c|}{\bf PDASC}&\multicolumn{1}{c|}{4.29e-1}&\multicolumn{1}{c|}{1.34e-1}&\multicolumn{1}{c|}{5.94e-1}&\multicolumn{1}{c|}{0}&\multicolumn{1}{c|}{6.04e-2}&
\multicolumn{1}{c|}{4.27e-1}&\multicolumn{1}{c|}{1.33e-1}&\multicolumn{1}{c|}{5.92e-1}&\multicolumn{1}{c|}{0}&\multicolumn{1}{c|}{5.98e-2}&
\multicolumn{1}{c|}{4.53e-1}&\multicolumn{1}{c|}{1.37e-1}&\multicolumn{1}{c|}{6.08e-1}&\multicolumn{1}{c|}{1.01e-5}&\multicolumn{1}{c|}{5.93e-2}\\ \hline
\multicolumn{1}{|c|}{\bf WPDASC}&\multicolumn{1}{c|}{4.38e-1}&\multicolumn{1}{c|}{1.37e-1}&\multicolumn{1}{c|}{6.02e-1}&\multicolumn{1}{c|}{0}&\multicolumn{1}{c|}{9.31e-2}&
\multicolumn{1}{c|}{4.20e-1}&\multicolumn{1}{c|}{1.30e-1}&\multicolumn{1}{c|}{5.78e-1}&\multicolumn{1}{c|}{0}&\multicolumn{1}{c|}{9.32e-2}&
\multicolumn{1}{c|}{3.97e-1}&\multicolumn{1}{c|}{1.20e-1}&\multicolumn{1}{c|}{5.56e-1}&\multicolumn{1}{c|}{1.01e-5}&\multicolumn{1}{c|}{9.69e-2}\\ \hline
\multicolumn{16}{|c|}{$m=800,n=2000,s^*=10,\varpi=0.1,{\bf r=0.15}$}
\\ \hline
\multicolumn{1}{|l|}{}&\multicolumn{5}{c|}{$\mu=0.1$}&\multicolumn{5}{c|}{$\mu=0.3$}&\multicolumn{5}{c|}{$\mu=0.5$}
\\ \hline
\multicolumn{1}{|c|}{}&
\multicolumn{1}{c|}{MSE}&\multicolumn{1}{c|}{Herr}&\multicolumn{1}{c|}{FNR}&\multicolumn{1}{c|}{FPR}&\multicolumn{1}{c|}{time(s)}&
\multicolumn{1}{c|}{MSE}&\multicolumn{1}{c|}{Herr}&\multicolumn{1}{c|}{FNR}&\multicolumn{1}{c|}{FPR}&\multicolumn{1}{c|}{time(s)}&
\multicolumn{1}{c|}{MSE}&\multicolumn{1}{c|}{Herr}&\multicolumn{1}{c|}{FNR}&\multicolumn{1}{c|}{FPR}&\multicolumn{1}{c|}{time(s)}
\\ \hline
\multicolumn{1}{|c|}{PIHT}&\multicolumn{1}{c|}{4.10e-1}&\multicolumn{1}{c|}{1.13e-1}&\multicolumn{1}{c|}{4.04e-1}&\multicolumn{1}{c|}{2.03e-3}&\multicolumn{1}{c|}{1.61e-1}& \multicolumn{1}{c|}{4.01e-1}&\multicolumn{1}{c|}{1.08e-1}&\multicolumn{1}{c|}{\color{red}2.90e-1}&\multicolumn{1}{c|}{1.96e-3}&\multicolumn{1}{c|}{1.57e-1}&
\multicolumn{1}{c|}{4.18e-1}&\multicolumn{1}{c|}{1.12e-1}&\multicolumn{1}{c|}{4.02e-1}&\multicolumn{1}{c|}{2.02e-1}&\multicolumn{1}{c|}{1.60e-1}\\ \hline
\multicolumn{1}{|c|}{BIHT-AOP}&\multicolumn{1}{c|}{3.77e-1}&\multicolumn{1}{c|}{1.04e-1}&\multicolumn{1}{c|}{4.10e-1}&\multicolumn{1}{c|}{2.06e-3}&\multicolumn{1}{c|}{5.41e-1}&
\multicolumn{1}{c|}{3.74e-1}&\multicolumn{1}{c|}{\color{red}9.88e-2}&\multicolumn{1}{c|}{4.16e-1}&\multicolumn{1}{c|}{2.09e-3}&\multicolumn{1}{c|}{5.51e-1}&
\multicolumn{1}{c|}{\color{red}3.56e-1}&\multicolumn{1}{c|}{\color{red}9.49e-2}&\multicolumn{1}{c|}{3.94e-1}&\multicolumn{1}{c|}{1.98e-3}&\multicolumn{1}{c|}{5.38e-1}\\ \hline
\multicolumn{1}{|c|}{PIHT-AOP}&\multicolumn{1}{c|}{\color{red}3.48e-1}&\multicolumn{1}{c|}{\color{red}9.80e-2}&\multicolumn{1}{c|}{\color{red}3.82e-1}&\multicolumn{1}{c|}{\color{red}1.92e-3}&\multicolumn{1}{c|}{1.85e-1}&
\multicolumn{1}{c|}{\color{red}3.70e-1}&\multicolumn{1}{c|}{1.01e-1}&\multicolumn{1}{c|}{4.10e-1}&\multicolumn{1}{c|}{2.06e-3}&\multicolumn{1}{c|}{1.91e-1}&
\multicolumn{1}{c|}{3.65e-1}&\multicolumn{1}{c|}{9.68e-2}&\multicolumn{1}{c|}{4.08e-1}&\multicolumn{1}{c|}{2.05e-3}&\multicolumn{1}{c|}{1.87e-1}\\ \hline
\multicolumn{1}{|c|}{GPSP}&\multicolumn{1}{c|}{3.90e-1}&\multicolumn{1}{c|}{1.05e-1}&\multicolumn{1}{c|}{3.86e-1}&\multicolumn{1}{c|}{1.94e-3}&\multicolumn{1}{c|}{1.86e-2}&
\multicolumn{1}{c|}{3.73e-1}&\multicolumn{1}{c|}{1.01e-1}&\multicolumn{1}{c|}{3.76e-1}&\multicolumn{1}{c|}{\color{red}1.89e-3}&\multicolumn{1}{c|}{2.04e-2}&
\multicolumn{1}{c|}{3.63e-1}&\multicolumn{1}{c|}{9.31e-2}&\multicolumn{1}{c|}{\color{red}3.74e-1}&\multicolumn{1}{c|}{\color{red}1.88e-3}&\multicolumn{1}{c|}{2.48e-2}\\
\hline\hline
\multicolumn{1}{|c|}{\bf PGe-scad}&\multicolumn{1}{c|}{\color{red}2.72e-1}&\multicolumn{1}{c|}{\color{red}8.54e-2}&\multicolumn{1}{c|}{\color{red}3.98e-1}&\multicolumn{1}{c|}{1.51e-4}&\multicolumn{1}{c|}{2.31e-1}&
\multicolumn{1}{c|}{\color{red}2.78e-1}&\multicolumn{1}{c|}{\color{red}8.67e-2}&\multicolumn{1}{c|}{\color{red}3.90e-1}&\multicolumn{1}{c|}{1.81e-4}&\multicolumn{1}{c|}{2.83e-1}&
\multicolumn{1}{c|}{\color{red}2.83e-1}&\multicolumn{1}{c|}{\color{red}8.39e-2}&\multicolumn{1}{c|}{\color{red}4.02e-1}&\multicolumn{1}{c|}{1.91e-4}&\multicolumn{1}{c|}{2.63e-1}\\
\hline
\multicolumn{1}{|c|}{\bf PGe-znorm}&\multicolumn{1}{c|}{3.33e-1}&\multicolumn{1}{c|}{9.82e-2}&\multicolumn{1}{c|}{4.20e-1}&\multicolumn{1}{c|}{8.24e-4}&\multicolumn{1}{c|}{1.34e-1}&
\multicolumn{1}{c|}{3.31e-1}&\multicolumn{1}{c|}{9.64e-2}&\multicolumn{1}{c|}{4.16e-1}&\multicolumn{1}{c|}{9.45e-3}&\multicolumn{1}{c|}{1.34e-1}&
\multicolumn{1}{c|}{3.42e-1}&\multicolumn{1}{c|}{9.56e-2}&\multicolumn{1}{c|}{4.14e-1}&\multicolumn{1}{c|}{9.55e-4}&\multicolumn{1}{c|}{1.50e-1}\\ \hline
\multicolumn{1}{|c|}{\bf PDASC}&\multicolumn{1}{c|}{5.63e-1}&\multicolumn{1}{c|}{1.80e-1}&\multicolumn{1}{c|}{6.88e-1}&\multicolumn{1}{c|}{0}&\multicolumn{1}{c|}{5.08e-2}&
\multicolumn{1}{c|}{5.89e-1}&\multicolumn{1}{c|}{1.85e-1}&\multicolumn{1}{c|}{7.12e-1}&\multicolumn{1}{c|}{0}&\multicolumn{1}{c|}{5.18e-2}&
\multicolumn{1}{c|}{5.58e-1}&\multicolumn{1}{c|}{1.73e-1}&\multicolumn{1}{c|}{6.90e-1}&\multicolumn{1}{c|}{4.02e-5}&\multicolumn{1}{c|}{5.18e-2}\\ \hline
\multicolumn{1}{|c|}{\bf WPDASC}&\multicolumn{1}{c|}{5.40e-1}&\multicolumn{1}{c|}{1.71e-1}&\multicolumn{1}{c|}{6.72e-1}&\multicolumn{1}{c|}{0}&\multicolumn{1}{c|}{7.93e-2}&
\multicolumn{1}{c|}{5.87e-1}&\multicolumn{1}{c|}{1.83e-1}&\multicolumn{1}{c|}{7.10e-1}&\multicolumn{1}{c|}{1.01e-5}&\multicolumn{1}{c|}{8.32e-2}&
\multicolumn{1}{c|}{5.63e-1}&\multicolumn{1}{c|}{1.75e-1}&\multicolumn{1}{c|}{6.98e-1}&\multicolumn{1}{c|}{0}&\multicolumn{1}{c|}{8.08e-2}\\ \hline
\end{tabular}}
\end{table}

 Next we use the eight solvers to solve the test problems with the sampling matrix
 of type I and high noise. Table \ref{table2} reports the average MSE, Herr, FNR,
 FPR and CPU time for $50$ trials. Now among the four solvers requiring partial
 information on $x^{\rm true}$, GPSP yields the smallest MSE, {\rm Herr}, FNR
 and FPR, and among the four solvers without requiring any information on $x^{\rm true}$,
 PGe-scad is still the best one. Also, for those problems with $r=0.15$,
 PGe-scad yields the smaller MSE, Herr and FNR than GPSP does.
%-------------------------------------------------------------------------------------------------------------
\begin{table}[!h]
\centering
\tiny
\captionsetup{font={scriptsize}}
\caption{Numerical comparisons of eight solvers for test problems with $\Phi$ of type I and high noise}
\label{table2}
\resizebox{\textwidth}{35mm}{
\begin{tabular}{|llllllllllllllll|}
\hline
\multicolumn{16}{|c|}{$m=1000,n=5000,s^*=15,\varpi=0.3,{\bf r=0.05}$}                                                                                                                                                                                                                                 \\ \hline
\multicolumn{1}{|l|}{}&\multicolumn{5}{c|}{$\mu=0.1$}&\multicolumn{5}{c|}{$\mu=0.3$}&\multicolumn{5}{c|}{$\mu=0.5$}
\\ \hline
\multicolumn{1}{|c|}{solvers}&
\multicolumn{1}{c|}{MSE}&\multicolumn{1}{c|}{Herr}&\multicolumn{1}{c|}{FNR}&\multicolumn{1}{c|}{FPR}&\multicolumn{1}{c|}{time(s)}& \multicolumn{1}{c|}{MSE}&\multicolumn{1}{c|}{Herr}&\multicolumn{1}{c|}{FNR}&\multicolumn{1}{c|}{FPR}&\multicolumn{1}{c|}{time(s)}&
\multicolumn{1}{c|}{MSE}&\multicolumn{1}{c|}{Herr}&\multicolumn{1}{c|}{FNR}&\multicolumn{1}{c|}{FPR}&\multicolumn{1}{c|}{time(s)}
\\ \hline
\multicolumn{1}{|c|}{PIHT}&\multicolumn{1}{c|}{3.48e-1}&\multicolumn{1}{c|}{9.53e-2}&\multicolumn{1}{c|}{4.15e-1}&\multicolumn{1}{c|}{1.25e-3}&\multicolumn{1}{c|}{5.42e-1}& \multicolumn{1}{c|}{3.40e-1}&\multicolumn{1}{c|}{9.54e-2}&\multicolumn{1}{c|}{4.19e-1}&\multicolumn{1}{c|}{1.26e-3}&\multicolumn{1}{c|}{5.46e-1}&
\multicolumn{1}{c|}{3.62e-1}&\multicolumn{1}{c|}{9.67e-2}&\multicolumn{1}{c|}{4.47e-1}&\multicolumn{1}{c|}{1.34e-3}&\multicolumn{1}{c|}{5.56e-1}\\ \hline
\multicolumn{1}{|c|}{BIHT-AOP}&\multicolumn{1}{c|}{3.57e-1}&\multicolumn{1}{c|}{1.11e-1}&\multicolumn{1}{c|}{3.80e-1}&\multicolumn{1}{c|}{1.14e-3}&\multicolumn{1}{c|}{1.59e-0}&
\multicolumn{1}{c|}{3.47e-1}&\multicolumn{1}{c|}{1.09e-1}&\multicolumn{1}{c|}{3.67e-1}&\multicolumn{1}{c|}{1.10e-3}&\multicolumn{1}{c|}{1.60e-0}&
\multicolumn{1}{c|}{3.25e-1}&\multicolumn{1}{c|}{1.05e-1}&\multicolumn{1}{c|}{3.61e-1}&\multicolumn{1}{c|}{1.09e-3}&\multicolumn{1}{c|}{1.61e-0}\\ \hline
\multicolumn{1}{|c|}{PIHT-AOP}&\multicolumn{1}{c|}{3.71e-1}&\multicolumn{1}{c|}{1.17e-1}&\multicolumn{1}{c|}{3.83e-1}&\multicolumn{1}{c|}{1.15e-3}&\multicolumn{1}{c|}{5.66e-1}& \multicolumn{1}{c|}{3.47e-1}&\multicolumn{1}{c|}{1.10e-1}&\multicolumn{1}{c|}{3.71e-1}&\multicolumn{1}{c|}{1.11e-3}&\multicolumn{1}{c|}{5.79e-1}&
\multicolumn{1}{c|}{3.30e-1}&\multicolumn{1}{c|}{1.10e-1}&\multicolumn{1}{c|}{\color{red}3.59e-1}&\multicolumn{1}{c|}{\color{red}1.08e-3}&\multicolumn{1}{c|}{5.83e-1}\\ \hline
\multicolumn{1}{|c|}{GPSP}&\multicolumn{1}{c|}{\color{red}2.64e-1}&\multicolumn{1}{c|}{\color{red}7.33e-2}&\multicolumn{1}{c|}{\color{red}3.31e-1}&\multicolumn{1}{c|}{\color{red}9.95e-4}&\multicolumn{1}{c|}{5.77e-2}
&\multicolumn{1}{c|}{\color{red}2.68e-1}&\multicolumn{1}{c|}{\color{red}7.52e-2}&\multicolumn{1}{c|}{\color{red}3.32e-1}&\multicolumn{1}{c|}{\color{red}9.99e-4}&\multicolumn{1}{c|}{5.19e-2}&
\multicolumn{1}{c|}{\color{red}2.95e-1}&\multicolumn{1}{c|}{\color{red}8.08e-2}&\multicolumn{1}{c|}{3.63e-1}&\multicolumn{1}{c|}{1.09e-3}&\multicolumn{1}{c|}{4.89e-2}\\
\hline\hline
\multicolumn{1}{|c|}{\bf PGe-scad}&\multicolumn{1}{c|}{\color{red}2.65e-1}&\multicolumn{1}{c|}{\color{red}8.36e-2}&\multicolumn{1}{c|}{\color{red}3.87e-1}&\multicolumn{1}{c|}{\color{red}2.21e-4}&\multicolumn{1}{c|}{1.05e-0}&
\multicolumn{1}{c|}{\color{red}2.67e-1}&\multicolumn{1}{c|}{\color{red}8.35e-2}&\multicolumn{1}{c|}{\color{red}3.85e-1}&\multicolumn{1}{c|}{\color{red}2.45e-4}&\multicolumn{1}{c|}{1.01e-0}&
\multicolumn{1}{c|}{\color{red}2.65e-1}&\multicolumn{1}{c|}{\color{red}8.13e-2}&\multicolumn{1}{c|}{\color{red}3.93e-1}&\multicolumn{1}{c|}{\color{red}2.29e-4}&\multicolumn{1}{c|}{1.10e-0}\\
\hline
\multicolumn{1}{|c|}{\bf PGe-znorm}&\multicolumn{1}{c|}{2.89e-1}&\multicolumn{1}{c|}{8.85e-2}&\multicolumn{1}{c|}{4.67e-1}&\multicolumn{1}{c|}{8.83e-5}&\multicolumn{1}{c|}{4.05e-1}&
\multicolumn{1}{c|}{2.92e-1}&\multicolumn{1}{c|}{8.89e-2}&\multicolumn{1}{c|}{4.69e-1}&\multicolumn{1}{c|}{7.22e-5}&\multicolumn{1}{c|}{4.17e-1}&
\multicolumn{1}{c|}{3.00e-1}&\multicolumn{1}{c|}{8.95e-2}&\multicolumn{1}{c|}{4.77e-1}&\multicolumn{1}{c|}{9.63e-5}&\multicolumn{1}{c|}{4.20e-1}\\ \hline
\multicolumn{1}{|c|}{\bf PDASC}&\multicolumn{1}{c|}{5.55e-1}&\multicolumn{1}{c|}{1.77e-1}&\multicolumn{1}{c|}{7.24e-1}&\multicolumn{1}{c|}{0}&\multicolumn{1}{c|}{1.63e-1}&
\multicolumn{1}{c|}{5.80e-1}&\multicolumn{1}{c|}{1.84e-1}&\multicolumn{1}{c|}{7.44e-1}&\multicolumn{1}{c|}{0}&\multicolumn{1}{c|}{1.64e-1}&
\multicolumn{1}{c|}{5.96e-1}&\multicolumn{1}{c|}{1.89e-1}&\multicolumn{1}{c|}{7.48e-1}&\multicolumn{1}{c|}{4.01e-6}&\multicolumn{1}{c|}{1.65e-1}\\ \hline
\multicolumn{1}{|c|}{\bf WPDASC}&\multicolumn{1}{c|}{5.73e-1}&\multicolumn{1}{c|}{1.84e-1}&\multicolumn{1}{c|}{7.36e-1}&\multicolumn{1}{c|}{0}&\multicolumn{1}{c|}{2.66e-1}&
\multicolumn{1}{c|}{5.54e-1}&\multicolumn{1}{c|}{1.76e-1}&\multicolumn{1}{c|}{7.19e-1}&\multicolumn{1}{c|}{0}&\multicolumn{1}{c|}{2.68e-1}&
\multicolumn{1}{c|}{5.96e-1}&\multicolumn{1}{c|}{1.88e-1}&\multicolumn{1}{c|}{7.49e-1}&\multicolumn{1}{c|}{0}&\multicolumn{1}{c|}{2.70e-1}\\ \hline
\multicolumn{16}{|c|}{$m=1000,n=5000,s^*=15,\varpi=0.3,{\bf r=0.15}$}
\\ \hline
\multicolumn{1}{|l|}{}&\multicolumn{5}{c|}{$\mu=0.1$}&\multicolumn{5}{c|}{$\mu=0.3$}&\multicolumn{5}{c|}{$\mu=0.5$}
\\ \hline
\multicolumn{1}{|c|}{}&
\multicolumn{1}{c|}{MSE}&\multicolumn{1}{c|}{Herr}&\multicolumn{1}{c|}{FNR}&\multicolumn{1}{c|}{FPR}&\multicolumn{1}{c|}{time(s)}&
\multicolumn{1}{c|}{MSE}&\multicolumn{1}{c|}{Herr}&\multicolumn{1}{c|}{FNR}&\multicolumn{1}{c|}{FPR}&\multicolumn{1}{c|}{time(s)}&
\multicolumn{1}{c|}{MSE}&\multicolumn{1}{c|}{Herr}&\multicolumn{1}{c|}{FNR}&\multicolumn{1}{c|}{FPR}&\multicolumn{1}{c|}{time(s)}
\\ \hline
\multicolumn{1}{|c|}{PIHT}&\multicolumn{1}{c|}{5.28e-1}&\multicolumn{1}{c|}{1.48e-1}&\multicolumn{1}{c|}{4.97e-1}&\multicolumn{1}{c|}{1.50e-3}&\multicolumn{1}{c|}{5.52e-1}& \multicolumn{1}{c|}{5.42e-1}&\multicolumn{1}{c|}{1.51e-1}&\multicolumn{1}{c|}{5.09e-1}&\multicolumn{1}{c|}{1.53e-3}&\multicolumn{1}{c|}{5.50e-1}&
\multicolumn{1}{c|}{5.30e-1}&\multicolumn{1}{c|}{1.47e-1}&\multicolumn{1}{c|}{4.81e-1}&\multicolumn{1}{c|}{1.45e-3}&\multicolumn{1}{c|}{5.46e-1}\\ \hline
\multicolumn{1}{|c|}{BIHT-AOP}&\multicolumn{1}{c|}{5.23e-1}&\multicolumn{1}{c|}{1.51e-1}&\multicolumn{1}{c|}{5.13e-1}&\multicolumn{1}{c|}{1.54e-3}&\multicolumn{1}{c|}{1.61e-0}&
\multicolumn{1}{c|}{4.97e-1}&\multicolumn{1}{c|}{1.45e-1}&\multicolumn{1}{c|}{5.00e-1}&\multicolumn{1}{c|}{1.50e-3}&\multicolumn{1}{c|}{1.60e-0}&
\multicolumn{1}{c|}{5.19e-1}&\multicolumn{1}{c|}{1.46e-1}&\multicolumn{1}{c|}{5.35e-1}&\multicolumn{1}{c|}{1.61e-3}&\multicolumn{1}{c|}{1.61e-0}\\ \hline
\multicolumn{1}{|c|}{PIHT-AOP}&\multicolumn{1}{c|}{5.13e-1}&\multicolumn{1}{c|}{1.46e-1}&\multicolumn{1}{c|}{5.09e-1}&\multicolumn{1}{c|}{1.53e-3}&\multicolumn{1}{c|}{5.78e-1}&
\multicolumn{1}{c|}{5.04e-1}&\multicolumn{1}{c|}{1.47e-1}&\multicolumn{1}{c|}{5.13e-1}&\multicolumn{1}{c|}{1.54e-3}&\multicolumn{1}{c|}{5.79e-1}&
\multicolumn{1}{c|}{5.30e-1}&\multicolumn{1}{c|}{1.52e-1}&\multicolumn{1}{c|}{5.45e-1}&\multicolumn{1}{c|}{1.64e-3}&\multicolumn{1}{c|}{5.70e-1}\\ \hline
\multicolumn{1}{|c|}{GPSP}&\multicolumn{1}{c|}{\color{red}4.60e-1}&\multicolumn{1}{c|}{\color{red}1.29e-1}&\multicolumn{1}{c|}{\color{red}4.59e-1}&\multicolumn{1}{c|}{\color{red}1.38e-3}&\multicolumn{1}{c|}{8.21e-2}&
\multicolumn{1}{c|}{\color{red}4.60e-1}&\multicolumn{1}{c|}{\color{red}1.29e-1}&\multicolumn{1}{c|}{\color{red}4.65e-1}&\multicolumn{1}{c|}{\color{red}1.40e-3}&\multicolumn{1}{c|}{8.13e-2}&
\multicolumn{1}{c|}{\color{red}4.90e-1}&\multicolumn{1}{c|}{\color{red}1.33e-1}&\multicolumn{1}{c|}{\color{red}4.77e-1}&\multicolumn{1}{c|}{\color{red}1.44e-3}&\multicolumn{1}{c|}{6.03e-2}\\
\hline\hline
\multicolumn{1}{|c|}{\bf PGe-scad}&\multicolumn{1}{c|}{\color{red}3.55e-1}&\multicolumn{1}{c|}{\color{red}1.09e-1}&\multicolumn{1}{c|}{\color{red}4.76e-1}&\multicolumn{1}{c|}{1.34e-3}&\multicolumn{1}{c|}{1.05e-0}&
\multicolumn{1}{c|}{\color{red}3.63e-1}&\multicolumn{1}{c|}{\color{red}1.12e-1}&\multicolumn{1}{c|}{\color{red}4.81e-1}&\multicolumn{1}{c|}{1.50e-3}&\multicolumn{1}{c|}{1.11e-0}&
\multicolumn{1}{c|}{\color{red}3.63e-1}&\multicolumn{1}{c|}{\color{red}1.12e-1}&\multicolumn{1}{c|}{\color{red}4.80e-1}&\multicolumn{1}{c|}{1.38e-3}&\multicolumn{1}{c|}{1.27e-0}\\
\hline
\multicolumn{1}{|c|}{\bf PGe-znorm}&\multicolumn{1}{c|}{4.22e-1}&\multicolumn{1}{c|}{1.24e-2}&\multicolumn{1}{c|}{5.17e-1}&\multicolumn{1}{c|}{6.38e-4}&\multicolumn{1}{c|}{4.19e-1}&
\multicolumn{1}{c|}{4.51e-1}&\multicolumn{1}{c|}{1.30e-1}&\multicolumn{1}{c|}{5.19e-1}&\multicolumn{1}{c|}{7.94e-4}&\multicolumn{1}{c|}{4.06e-1}&
\multicolumn{1}{c|}{4.41e-1}&\multicolumn{1}{c|}{1.27e-1}&\multicolumn{1}{c|}{5.21e-1}&\multicolumn{1}{c|}{6.62e-4}&\multicolumn{1}{c|}{4.49e-1}\\ \hline
\multicolumn{1}{|c|}{\bf PDASC}&\multicolumn{1}{c|}{6.90e-1}&\multicolumn{1}{c|}{2.24e-1}&\multicolumn{1}{c|}{8.12e-1}&\multicolumn{1}{c|}{4.01e-6}&\multicolumn{1}{c|}{1.35e-1}&
\multicolumn{1}{c|}{7.07e-1}&\multicolumn{1}{c|}{2.26e-1}&\multicolumn{1}{c|}{8.24e-1}&\multicolumn{1}{c|}{4.01e-6}&\multicolumn{1}{c|}{1.34e-1}&
\multicolumn{1}{c|}{7.11e-1}&\multicolumn{1}{c|}{2.27e-1}&\multicolumn{1}{c|}{8.27e-1}&\multicolumn{1}{c|}{0}&\multicolumn{1}{c|}{1.33e-1}\\ \hline
\multicolumn{1}{|c|}{\bf WPDASC}&\multicolumn{1}{c|}{6.62e-1}&\multicolumn{1}{c|}{2.14e-1}&\multicolumn{1}{c|}{7.92e-1}&\multicolumn{1}{c|}{0}&\multicolumn{1}{c|}{2.35e-1}&
\multicolumn{1}{c|}{6.82e-1}&\multicolumn{1}{c|}{2.18e-1}&\multicolumn{1}{c|}{8.03e-1}&\multicolumn{1}{c|}{4.01e-6}&\multicolumn{1}{c|}{2.34e-1}&
\multicolumn{1}{c|}{7.04e-1}&\multicolumn{1}{c|}{2.25e-1}&\multicolumn{1}{c|}{8.20e-1}&\multicolumn{1}{c|}{4.01e-6}&\multicolumn{1}{c|}{2.30e-1}\\ \hline
\end{tabular}}
\end{table}

 Finally, we use the eight solvers to solve the test problems with the sampling matrix of type II.
 Table \ref{table3} reports the average MSE, Herr, FNR, FPR and CPU time for $50$ trials.
 From Table \ref{table3}, among the four solvers requiring partial information on $x^{\rm true}$,
 PIHT yields the better MSE, {\rm Herr},FNR and FPR than others for those examples with
 high noise, and among the four solvers without needing any information on $x^{\rm true}$,
 PGe-scad is still the best one. Moreover, PGe-scad yields the smaller MSE, Herr and FNR
 than PIHT does for $\varpi=0.3$ and $0.5$. We also observe that among the eight solvers,
 GPSP always requires the least CPU time, and PGe-scad and PGe-znorm requires the comparable CPU time as
 PIHT, BIHT-AOP and PIHT-AOP do for all test examples.

%-------------------------------------------------------------------------------------------------------------
\begin{table}[!h]
\centering
\tiny
\captionsetup{font={scriptsize}}
\caption{Numerical comparisons of eight solvers for test problems with $\Phi$ of type II with different noise levels}
\label{table3}
\resizebox{\textwidth}{15mm}{
\begin{tabular}{|llllllllllllllll|}
\hline
\multicolumn{16}{|c|}{$m=2500,n=10000,s^*=20,r=0.1$}
\\ \hline
\multicolumn{1}{|l|}{}&\multicolumn{5}{c|}{$\varpi=0.1$}&\multicolumn{5}{c|}{$\varpi=0.3$}&\multicolumn{5}{c|}{$\varpi=0.5$}
\\ \hline
\multicolumn{1}{|l|}{}&
\multicolumn{1}{c|}{MSE}&\multicolumn{1}{c|}{Herr}&\multicolumn{1}{c|}{FNR}&\multicolumn{1}{c|}{FPR}&\multicolumn{1}{c|}{time(s)}&
\multicolumn{1}{c|}{MSE}&\multicolumn{1}{c|}{Herr}&\multicolumn{1}{c|}{FNR}&\multicolumn{1}{c|}{FPR}&\multicolumn{1}{c|}{time(s)}&
\multicolumn{1}{c|}{MSE}&\multicolumn{1}{c|}{Herr}&\multicolumn{1}{c|}{FNR}&\multicolumn{1}{c|}{FPR}&\multicolumn{1}{c|}{time(s)}\\ \hline
\multicolumn{1}{|c|}{PIHT}&\multicolumn{1}{c|}{2.57e-1}&\multicolumn{1}{c|}{7.23e-1}&\multicolumn{1}{c|}{3.44e-1}&\multicolumn{1}{c|}{6.89e-4}&\multicolumn{1}{c|}{2.55}& \multicolumn{1}{c|}{\color{red}2.74e-1}&\multicolumn{1}{c|}{\color{red}8.15e-2}&\multicolumn{1}{c|}{\color{red}3.49e-1}&\multicolumn{1}{c|}{\color{red}6.99e-4}&\multicolumn{1}{c|}{2.54}&
\multicolumn{1}{c|}{\color{red}3.14e-1}&\multicolumn{1}{c|}{\color{red}9.58e-2}&\multicolumn{1}{c|}{\color{red}3.69e-1}&\multicolumn{1}{c|}{\color{red}7.39e-4}&\multicolumn{1}{c|}{2.54}\\ \hline
\multicolumn{1}{|c|}{BIHT-AOP}&\multicolumn{1}{c|}{\color{red}1.54e-1}&\multicolumn{1}{c|}{\color{red}4.61e-2}&\multicolumn{1}{c|}{\color{red}2.40e-1}&\multicolumn{1}{c|}{\color{red}4.81e-4}&\multicolumn{1}{c|}{7.70}&
\multicolumn{1}{c|}{3.06e-1}&\multicolumn{1}{c|}{9.84e-2}&\multicolumn{1}{c|}{3.36e-1}&\multicolumn{1}{c|}{6.73e-3}&\multicolumn{1}{c|}{7.68}&
\multicolumn{1}{c|}{4.23e-1}&\multicolumn{1}{c|}{1.29e-1}&\multicolumn{1}{c|}{3.95e-1}&\multicolumn{1}{c|}{7.92e-4}&\multicolumn{1}{c|}{7.64}\\ \hline
\multicolumn{1}{|c|}{PIHT-AOP}&\multicolumn{1}{c|}{1.68e-1}&\multicolumn{1}{c|}{5.08e-2}&\multicolumn{1}{c|}{2.44e-1}&\multicolumn{1}{c|}{4.89e-4}&\multicolumn{1}{c|}{2.66}&
\multicolumn{1}{c|}{3.16e-1}&\multicolumn{1}{c|}{1.03e-1}&\multicolumn{1}{c|}{3.38e-1}&\multicolumn{1}{c|}{6.77e-4}&\multicolumn{1}{c|}{2.66}&
\multicolumn{1}{c|}{4.62e-1}&\multicolumn{1}{c|}{1.52e-1}&\multicolumn{1}{c|}{4.14e-1}&\multicolumn{1}{c|}{8.30e-4}&\multicolumn{1}{c|}{2.66}\\ \hline
\multicolumn{1}{|c|}{GPSP}&\multicolumn{1}{c|}{2.45e-1}&\multicolumn{1}{c|}{6.88e-2}&\multicolumn{1}{c|}{3.36e-1}&\multicolumn{1}{c|}{6.73e-4}&\multicolumn{1}{c|}{0.19}&
\multicolumn{1}{c|}{2.77e-1}&\multicolumn{1}{c|}{8.22e-2}&\multicolumn{1}{c|}{3.51e-1}&\multicolumn{1}{c|}{7.03e-4}&\multicolumn{1}{c|}{0.19}&
\multicolumn{1}{c|}{3.23e-1}&\multicolumn{1}{c|}{9.68e-2}&\multicolumn{1}{c|}{3.73e-1}&\multicolumn{1}{c|}{7.47e-4}&\multicolumn{1}{c|}{0.19}\\
\hline\hline
\multicolumn{1}{|c|}{\bf PGe-scad}&\multicolumn{1}{c|}{\color{red}2.10e-1}&\multicolumn{1}{c|}{\color{red}6.65e-2}&\multicolumn{1}{c|}{\color{red}3.08e-1}&\multicolumn{1}{c|}{2.61e-5}&\multicolumn{1}{c|}{2.79}&
\multicolumn{1}{c|}{\color{red}2.44e-1}&\multicolumn{1}{c|}{\color{red}7.82e-2}&\multicolumn{1}{c|}{\color{red}3.61e-1}&\multicolumn{1}{c|}{2.10e-4}&\multicolumn{1}{c|}{2.71}&
\multicolumn{1}{c|}{\color{red}2.92e-1}&\multicolumn{1}{c|}{\color{red}9.34e-2}&\multicolumn{1}{c|}{\color{red}4.14e-1}&\multicolumn{1}{c|}{6.89e-4}&\multicolumn{1}{c|}{2.75}\\
\hline
\multicolumn{1}{|c|}{\bf PGe-znorm}&\multicolumn{1}{c|}{2.34e-1}&\multicolumn{1}{c|}{7.36e-2}&\multicolumn{1}{c|}{4.25e-1}&\multicolumn{1}{c|}{2.00e-6}&\multicolumn{1}{c|}{1.78}&
\multicolumn{1}{c|}{2.44e-1}&\multicolumn{1}{c|}{7.71e-2}&\multicolumn{1}{c|}{4.27e-1}&\multicolumn{1}{c|}{8.02e-6}&\multicolumn{1}{c|}{1.78}&
\multicolumn{1}{c|}{2.74e-1}&\multicolumn{1}{c|}{8.70e-2}&\multicolumn{1}{c|}{4.45e-1}&\multicolumn{1}{c|}{3.21e-5}&\multicolumn{1}{c|}{1.77}\\ \hline
\multicolumn{1}{|c|}{\bf PDASC}&\multicolumn{1}{c|}{5.41e-1}&\multicolumn{1}{c|}{1.73e-1}&\multicolumn{1}{c|}{6.84e-1}&\multicolumn{1}{c|}{0}&\multicolumn{1}{c|}{8.28e-1}&
\multicolumn{1}{c|}{6.26e-1}&\multicolumn{1}{c|}{2.01e-1}&\multicolumn{1}{c|}{7.50e-1}&\multicolumn{1}{c|}{0}&\multicolumn{1}{c|}{8.07e-1}&
\multicolumn{1}{c|}{6.28e-1}&\multicolumn{1}{c|}{2.03e-1}&\multicolumn{1}{c|}{7.56e-1}&\multicolumn{1}{c|}{4.02e-5}&\multicolumn{1}{c|}{7.66e-1}\\ \hline
\multicolumn{1}{|c|}{\bf WPDASC}&\multicolumn{1}{c|}{5.41e-1}&\multicolumn{1}{c|}{1.73e-1}&\multicolumn{1}{c|}{6.83e-1}&\multicolumn{1}{c|}{0}&\multicolumn{1}{c|}{1.01}&
\multicolumn{1}{c|}{5.62e-1}&\multicolumn{1}{c|}{1.81e-1}&\multicolumn{1}{c|}{7.03e-1}&\multicolumn{1}{c|}{0}&\multicolumn{1}{c|}{9.98e-1}&
\multicolumn{1}{c|}{6.17e-1}&\multicolumn{1}{c|}{1.99e-1}&\multicolumn{1}{c|}{7.37e-1}&\multicolumn{1}{c|}{0}&\multicolumn{1}{c|}{9.71e-1}\\ \hline
\end{tabular}}
\end{table}
%----------------------------------------------------------------------------------------
 \section{Conclusion}\label{sec6}

  We proposed a zero-norm regularized smooth DC loss model and derived a family of
  equivalent nonconvex surrogates that cover the MCP and SCAD surrogates as special cases.
  For the proposed model and its SCAD surrogate, we developed the PG method with extrapolation
  to compute their $\tau$-stationary points and provided its convergence certificate by
  establishing the convergence of the whole iterate sequence and its local linear convergence rate.
  Numerical comparisons with several state-of-art methods demonstrate that the two new models
  are well suited for high noise and/or high sign flip ratio.
  An interesting future topic is to analyze the statistical error bound for them.

%---------------------------------------------------------------------------------------------------

 \end{document}